\documentclass[11pt, reqno]{amsart}
\usepackage{a4,amsmath,amssymb,amscd,verbatim,bbm,graphicx,enumerate, blkarray}
\usepackage{amsfonts,amsthm}
\usepackage[driverfallback=dvipdfm]{hyperref}
\usepackage{graphicx}
\usepackage{bmpsize}
\usepackage{MnSymbol}
\usepackage{color}

\newtheorem{theorem}{Theorem}
\newtheorem{proposition}[theorem]{Proposition}
\newtheorem{corollary}[theorem]{Corollary}
\newtheorem{lemma}[theorem]{Lemma}

\theoremstyle{definition}

\newtheorem{definition}[theorem]{Definition}

\newtheorem{remark}[theorem]{Remark}

\newtheorem{example}[theorem]{Example}

\numberwithin{equation}{section}
\numberwithin{theorem}{section}

\renewcommand{\a}{\alpha}
\renewcommand{\b}{\beta}
\newcommand{\g}{\gamma}
\newcommand{\G}{\Gamma}
\renewcommand{\d}{\delta}
\newcommand{\D}{\Delta}

\renewcommand{\i}{\iota}
\renewcommand{\k}{\kappa}

\renewcommand{\l}{\lambda}
\renewcommand{\L}{\Lambda}
\newcommand{\m}{\mu}
\newcommand{\n}{\nu}
\renewcommand{\o}{\omega}
\renewcommand{\O}{\Omega}
\renewcommand{\r}{\rho}
\newcommand{\s}{\sigma}
\renewcommand{\SS}{\Sigma}
\renewcommand{\t}{\tau}
\renewcommand{\th}{\theta}
\newcommand{\e}{\varepsilon}
\newcommand{\f}{\varphi}
\newcommand{\F}{\Phi}
\newcommand{\x}{\xi}

\newcommand{\y}{\eta}
\newcommand{\z}{\zeta}
\newcommand{\p}{\psi}

\newcommand{\Ac}{{\mathcal A}}

\newcommand{\Cc}{{\mathcal C}}
\newcommand{\Dc}{{\mathcal D}}

\newcommand{\Fc}{{\mathcal F}}
\newcommand{\Gc}{{\mathcal G}}

\newcommand{\Lc}{{\mathcal L}}

\newcommand{\Pc}{{\mathcal P}}
\newcommand{\Qc}{{\mathcal Q}}
\newcommand{\Hc}{{\mathcal H}}

\newcommand{\R}{{\mathbb R}}

\newcommand{\Z}{{\mathbb Z}}

\renewcommand{\Pr}{{\mathbb P}}

\def\({\left(}
\def\){\right)}

\def\l\{{\left\{}
\def\r\}{\right\}}
\def\wt{\widetilde}
\def\wh{\widehat}
\def\wbar{\overline}
\def\id{{\rm id}}

\def\Cay{{\rm Cay}}
\def\Leb{{\rm Leb}}
\def\ev{{\rm ev\,}}
\def\1{{\bf 1}}

\def\gr{{\rm gr}}
\def\conj{{\bf conj}}

\def\ac{{\rm ac}}
\def\sing{{\rm sing}}

\def\H{{\rm H}}

\def\bS{{\bf S}}

\def\CAT{{\rm CAT}}

\def\Dil{{\rm Dil}}

\newcommand\numberthis{\addtocounter{equation}{1}\tag{\theequation}}

   \newcommand{\SC}[1]{
     {\color{red} \bf{(SC: #1)}}
   }
  \newcommand{\RT}[1]{
    {\color{blue} \bf{(RT: #1)}}
  }

\author{Stephen Cantrell}
\address{Department of Mathematics, 
University of Warwick, Coventry, CV4 7AL, UK}
\email{stephen.cantrell@warwick.ac.uk}

\author{Ryokichi Tanaka}
\address{Department of Mathematics, 
Kyoto University, Kyoto 606-8502, JAPAN}
\email{rtanaka@math.kyoto-u.ac.jp}

\date{\today. 
}

\title[The Manhattan curve and rough similarity rigidity]
{The Manhattan curve, ergodic theory of\\ topological flows and rigidity}

\begin{document}

\maketitle

\begin{abstract}
For every non-elementary hyperbolic group, we introduce the Manhattan curve associated to each pair of left-invariant hyperbolic metrics which are quasi-isometric to a word metric. 
It is convex; we show that it is continuously differentiable
and moreover is a straight line if and only if the corresponding two metrics are roughly similar, i.e., they are within bounded distance after multiplying by a positive constant.
Further, we prove that  the Manhattan curve associated to two strongly hyperbolic metrics
 is twice continuously differentiable.
The proof is based on the ergodic theory of topological flows associated to general hyperbolic groups and analyzing the multifractal structure of Patterson-Sullivan measures.
We exhibit some explicit examples including a hyperbolic triangle group and compute the exact value of the mean distortion for pairs of word metrics.
\end{abstract}

\section{Introduction}

Let $\Gamma$ be a non-elementary hyperbolic group.
Given a pair of hyperbolic metrics $d$ and $d_\ast$ which are left-invariant and quasi-isometric to a word metric on $\G$ (hence they are quasi-isometric each other), 
we determine exactly when they are \textit{roughly similar}, i.e., $d$ and $d_\ast$ are within bounded distance after rescaling by a positive constant,
in terms of the \textit{Manhattan curve} for the pair of metrics.

For $d$ (resp.\ $d_\ast$), let us define the stable translation length by
\[
\ell[x]:=\lim_{n \to \infty} \frac{1}{n}d(o, x^n) \quad \text{for $x \in \G$},
\]
(resp.\ $\ell_\ast[x]$),
where the limit exists since the function $n \mapsto d(o, x^n)$ is subadditive,
and $o$ denotes the identity element in $\G$.
Note that the stable translation length for $x$ depends only on the conjugacy class of $x$, and thus defines the function on the set of conjugacy classes $\conj$ in $\G$.
Let us consider the following series with two parameters
\[
\mathcal{Q}(a,b) = \sum_{[x]  \in  \conj }\exp(-a\ell_\ast[x] - b\ell[x]) \quad \text{for $a,b \in \mathbb{R}$}.
\]
The Manhattan curve $\Cc_M$ for a pair $(d, d_\ast)$ is defined by 
the boundary of the following convex set
\[
\big\{(a,b) \in \mathbb{R}^2 \ : \ \Qc(a, b)<\infty\big\}.
\]
This curve was introduced by Burger in \cite{BurgerManhattan} for a class of groups acting on a non-compact symmetric space of rank one.

Let $\Dc_\G$ be the set of hyperbolic metrics which are left-invariant and quasi-isometric to some (equivalently, every) word metric in $\G$.
Recall that for $d, d_\ast \in \Dc_\G$, 
we say that $d$ and $d_\ast$ are roughly similar if there exist constants $\t>0$ and $C \ge 0$ such that
\[
|d_\ast(x, y)-\t d(x, y)| \le C \quad \text{for all $x, y \in \G$}.
\]

\begin{theorem}\label{Thm:thm1}
Let $\G$ be a non-elementary hyperbolic group.
For every pair $(d, d_\ast)$ in $\Dc_\G$,
the Manhattan curve $\Cc_M$ for $(d, d_\ast)$ is continuously differentiable,
and it is a straight line
if and only if $d$ and $d_\ast$ are roughly similar.
\end{theorem}

Note that, if $v$ (resp.\ $v_\ast$) is the abscissa of convergence for $\Qc(0, b)$ in $b$ (resp.\ $\Qc(a, 0)$ in $a$),
then 
\[
v:=\lim_{r \to \infty}\frac{1}{r}\log \#B(o, r) \quad \text{where \hspace{1mm} $B(o, r):=\{x \in \G \ : \ d(o, x) \le r\}$},
\]
(similarly $v_\ast$ for $d_\ast$) and $\#B$ denotes the cardinality of $B$. In particular $(0, v)$ and $(v_\ast, 0)$ lie on $\Cc_M$ (see Section \ref{Sec:general}).

\begin{theorem}\label{Thm:thm2}
Let $\G$ be a non-elementary hyperbolic group.
For every pair $(d, d_\ast)$ in $\Dc_\G$,
the following limit exists:
\begin{equation}\label{Eq:distortion}
\t(d_\ast/d):=\lim_{r \to \infty}\frac{1}{\#B(o, r)}\sum_{x \in B(o, r)}\frac{d_\ast(o, x)}{r},
\end{equation}
where the balls $B(o, r)$ are defined for $d$, and we have
\begin{equation}\label{Eq:distortion_ineq}
\t(d_\ast/d) \ge \frac{v}{v_\ast}.
\end{equation}
Moreover, the following are equivalent:
\begin{enumerate}
\item[\upshape{(1)}] the equality $\t(d_\ast/d)=v/v_\ast$ holds,
\item[\upshape{(2)}] there exists a constant $c>0$ such that $\ell_\ast[x]=c\ell[x]$ for all $[x] \in \conj$, and
\item[\upshape{(3)}] $d$ and $d_\ast$ are roughly similar.
\end{enumerate}
\end{theorem}

Let us call $\t(d_\ast/d)$ defined by \eqref{Eq:distortion} the \textit{mean distortion} of $d_\ast$ over $d$, and the inequality \eqref{Eq:distortion_ineq} the \textit{distortion inequality}.
The proof of Theorem \ref{Thm:thm2} is based on Theorem \ref{Thm:thm1} and relies on the following property of the Manhattan curve $\Cc_M$:
the slope of $\Cc_M$ at $(0, v)$ is $-\t(d_\ast/d)$ (Theorem \ref{Thm:distortion}).

If both metrics $d$ and $d_\ast$ are strongly hyperbolic (e.g., induced by an isometric cocompact action on a $\CAT(-1)$-space; see Definition \ref{Def:SH}),
then we have the following.

\begin{theorem}\label{Thm:thm3}
Let $\Gamma$ be a non-elementary hyperbolic group. 
If $d$ and $d_\ast$ are strongly hyperbolic metrics in $\Dc_\G$,
then the Manhattan curve $\mathcal{C}_M$ for $(d, d_\ast)$ is twice continuously differentiable.
\end{theorem}
We show this in Theorem \ref{Thm:2nd-d-sh}
and prove the analogous result for  pairs of word metrics in Theorem \ref{Thm:2nd-d-words}.

The various methods we use throughout our work allow us to connect the geometrical features of $\Cc_M$ with properties of the corresponding metrics.
For example, by comparing our methods to those of the first author in \cite{stats}, we connect the second differential  of $\Cc_M$ at $(0, v)$ with the variance of a central limit theorem for uniform counting measures on spheres (see Theorem \ref{thmwords} and the remark thereafter).  
We also connect the asymptotic gradients of $\Cc_M$ with the dilation constants
\[
\Dil_-:=\inf_{[x] \in \conj_{>0}}\frac{\ell_\ast[x]}{\ell[x]} \quad \text{and} \quad \Dil_+:=\sup_{[x] \in \conj_{>0}}\frac{\ell_\ast[x]}{\ell[x]},
\]
where $\conj_{>0}$ is the set of $[x] \in \conj$ such that $\ell[x]$ (and hence $\ell_\ast[x]$) is non-zero. We also show that, for every pair of word metrics, $\Dil_-$ and $\Dil_+$ are rational (Proposition \ref{Prop:extreme}).
Note that $\Dil_-=\Dil_+$ if and only if $\t(d_\ast/d)=v/v_\ast$ by Theorem \ref{Thm:thm2}.

\subsection{Historical backgrounds}

Burger introduced the Manhattan curve associated to a finitely generated, non-elementary group $\Gamma$ which acts on a rank one symmetric space $X$ properly discontinuously, convex cocompactly and without fixed points \cite{BurgerManhattan}.
For each convex cocompact realization of $\Gamma$ into the isometry group of $X$ there is a natural length function defined on the conjugacy classes of $\Gamma$: one can assign to each conjugacy class in $\Gamma$ the geometric length of the corresponding closed geodesic in the quotient.
Burger's Manhattan curve is defined using two of these length functions.
He showed that the curve is continuously differentiable and it is a straight line if and only if the isomorphism of lattices associated to the two corresponding realizations extends to an isomorphism of the ambient Lie groups \cite[Theorem 1]{BurgerManhattan}.
An important special case includes two isomorphic copies of torsion-free cocompact Fuchsian groups acting on the hyperbolic plane.
In this case, Sharp has shown that the associated Manhattan curve is real analytic by employing thermodynamic formalism for geodesic flows \cite{SharpManhattan}.
Recently, Kao has shown that the Manhattan curve is real analytic for a class of non-compact hyperbolic surfaces \cite{Kao}.
A similar rigidity problem related to the Manhattan curve is discussed for cusped Hitchin representations in \cite{BrayCanaryKaoMartone}.

Kaimanovich, Kapovich and Schupp have extensively studied similar problems for a free group $F$ of rank at least $2$ \cite{KKS}. 
They compared the pair of word metrics  for the generating sets $S$ and $\phi(S)$ where $S$ is the free generating set and $\phi$ is an automorphism of $F$.
They introduced the \textit{generic stretching factor} $\lambda(\phi)$ which is defined as the average or typical growth rate of $|\phi(x_n)|/n$ when $x_n$ is chosen uniformly at random from the words of length $n$ in $S$.
Using our terminology it amounts to considering $\t(S/\phi(S))=\lambda(\phi)$.
An automorphism $\phi$ of $F$ is called simple if it is a composition of an inner automorphism and a permutation of $S$.
It has been shown that $\phi$ is simple if and only if $\lambda(\phi)=1$ \cite[Theorem F]{KKS} (in which case $\phi$ gives rise to a rough similarity on the Cayley graph of $(F, S)$).
Sharp has pointed out its connection to the corresponding Manhattan curve \cite{SharpFree}:
he has identified $\lambda(\phi)$ with the slope of the normal line at the point $(\log(2k-1), 0)$ where $k$ is the rank of $F$. 
The mean distortion is a generalization of the generic stretching factor for hyperbolic groups and has appeared in the work of Calegari and Fujiwara \cite{CalegariFujiwara2010}
(though the identification is not immediately clear at first sight).
They have shown that for every pair of word metrics the distortion inequality \cite[Remark 4.28]{CalegariFujiwara2010} holds,
and also that the mean distortion is an algebraic integer \cite[Corollary 4.27]{CalegariFujiwara2010}.
Furthermore, a (possibly degenerate) central limit theorem (CLT) has been shown \cite[Theorem 4.25]{CalegariFujiwara2010} (see also \cite[Section 3.6]{Calegari}),
and their result has been generalized in
\cite[Theorem 1.2]{stats} and \cite[Theorem 1.1]{GTT}.
Moreover, the variance of CLT is zero if and only if two metrics are roughly similar \cite[Lemma 5.1]{stats} and \cite[Theorem 1.1]{GTT}.

Furman has proposed a general framework which can be used to compare metrics belonging to $\Dc_\Gamma$ when $\Gamma$ is a torsion-free hyperbolic group \cite{FurmanCoarse}.
In his work he introduced an abstract geodesic flow in a measurable category for a general hyperbolic group,
and showed that two metrics are roughly similar if and only if the associated Bowen-Margulis currents are not mutually singular on the boundary square $\partial^2 \G$ (which is the set of two distinct ordered pairs of points in the Gromov boundary $\partial \G$) \cite[Theorem 2]{FurmanCoarse}. He also claims that two metrics in $\Dc_\Gamma$ are roughly similar if and only if their associated translation length functions are proportional. A main motivation of the present paper is to  incorporate the properties of the Manhattan curve into rigidity statements that characterize rough similarity.
As a consequence we strengthen and generalize Furman's result for all non-elementary hyperbolic groups.

\subsection{Outline of proofs}

Let us sketch the proof of Theorem \ref{Thm:thm1}.
First we consider the following series in $a, b \in \R$,
\[
\Pc(a, b)=\sum_{x \in \G}\exp(-a d_\ast(o, x)-b d(o, x)),
\]
and identify the Manhattan curve $\Cc_M$ with the graph of $b=\th(a)$ 
where $\th(a)$ is the abscissa of convergence in $b$ for each fixed $a$ (Proposition \ref{Prop:conjugacy}).
In what follows, we also call the function $\th$ which parametrizes $\Cc_M$ the Manhattan curve.
Next we perform the Patterson-Sullivan construction for $a d_\ast+b d$ for $(a, b)$ with $b=\th(a)$ 
and construct a one-parameter family of measures $\m_{a, b}$ on $\partial \G$ for $(a, b) \in \Cc_M$ (Corollary \ref{Cor:ab}).
A key step in the proof of Theorem \ref{Thm:thm1} is understanding the variation of $\m_{a, b}$ in $a$.
Every measure $\m_{a, b}$ (which is not necessarily unique) is ergodic with respect to the $\G$-action on $\partial \G$ for each fixed $(a, b) \in \Cc_M$ and the
proof of this is adapted from classical arguments in \cite{Coornaert1993}.
Moreover, $\m_{a, b}$ is doubly ergodic, i.e., $\m_{a, b} \otimes \m_{a, b}$ is ergodic with respect to the diagonal action of $\G$ on $\partial^2 \G$.
Proving this amounts to showing that the geodesic flow is ergodic if it is properly defined e.g., in the case of manifolds (where we owe this idea to the work of Kaimanovich \cite{KaimanovichInvariant}).
Furman has constructed a framework where machinery concerning geodesic flows works for general hyperbolic groups \cite{FurmanCoarse} (see also \cite{BaderFurman}),
but the space in his setup has only a measurable structure 
and so difficulties arise when we discuss a family of flow-invariant measures (on the same space) and carry out limiting arguments with those measures.
We employ a compact model of geodesic flow defined by Mineyev \cite{MineyevFlow},
and show that there exist associated flow-invariant probability measures $m_{a, b}$ (which is unique) for each $(a, b) \in \Cc_M$ and they are continuous in $a \in \R$ in the weak-star topology (Section \ref{Sec:top-flow}).

Finally we define the \textit{local intersection number} $\t(\x)$ at $\x \in \partial \G$ for a pair $(d, d_\ast)$
as the limit of $d_\ast(\g_\x(0), \g_\x(t))/d(\g_\x(0), \g_\x(t))$ as $t\to\infty$, where $\g_\x$ is a quasi-geodesic such that $\g_\x(t) \to \x$ as $t \to \infty$,
if the limit exists.
We show that the limit exists and is equal to a constant $\t_{a, b}$ for $\m_{a, b}$-almost every $\x \in \partial \G$.
Furthermore $\t_{a, b}$ is continuous in $(a, b) \in \Cc_M$ (Section \ref{Sec:proofofC1}).
In fact, $\th'(a)=-\t_{a, b}$ if $\th$ is differentiable at $a \in \R$.
Note that $\th$ is convex (as seen from the definition) and thus is continuous everywhere and differentiable at all but at most countably many points.
Since we have shown that $\t_{a, b}$ is continuous in $a$, the function $\th$ is $C^1$.
In the last step, where we show that $\t_{a, b}$ coincides with $-\th'(a)$ (if it exists),
we prove that $\m_{a, b}$ assigns a full measure to the set where the local intersection number $\t(\x)$ is defined and is equal to $-\th'(a)$.
This discussion naturally leads us to study the multifractal spectrum of $\t(\x)$, i.e., to determine all the possible values $\a$ of $\t(\x)$ and the size of the level sets for which $\t(\x)=\a$.
This spectrum is actually the multifractal profile of a Patterson-Sullivan measure for $d_\ast$, where the Hausdorff dimension is defined by a quasi-metric associated with $d$ (Theorem \ref{Thm:MFS}).
Furthermore, the profile function is the Legendre transform of the Manhattan curve and is defined on the interval $(\Dil_-, \Dil_+)$.
(It would be interesting if it is defined on $[\Dil_-, \Dil_+]$, including the two extrema, which is indeed the case for some special case, e.g., word metrics. See Remark \ref{Rem:MFS} and our investigation of this point in Section \ref{Sec:extreme}.)
Based on this discussion, we show that the Manhattan curve $\Cc_M$ for a pair $(d, d_\ast)$ is a straight line if and only if
$d$ and $d_\ast$ are roughly similar (Theorem \ref{Thm:rigidity}) and thus conclude Theorem \ref{Thm:thm1}.
The discussion also yields Theorem \ref{Thm:thm2} by identifying $-\th'(0)$ with the mean distortion $\t(d_\ast/d)$ 
(see Theorem \ref{Thm:distortion}).

Let us briefly describe the proof of Theorem \ref{Thm:thm3} as our methods are quite different to those used in the proof of Theorem \ref{Thm:thm1}.
We introduce a subshift of finite type coming from the coding built upon Cannon's automatic structure. 
This allows us to employ techniques from thermodynamic formalism which we
apply within the strengthened thermodynamic framework of Gou\"ezel \cite{GouezelLocalLimit} (Section \ref{Sec:C2}).
This enables us to relate $\Cc_M$ to a family of real analytic (pressure) functions associated to suitable potentials on the subshift.
Unfortunately, in general, we do not know whether the subshift is topologically transitive, i.e., whether Cannon's automatic structure has a single recurrent component which is dominant.
This introduces additional difficulties.
Since the automatic structure may contain various large components,
our pressure functions of interest are not coming from a single component.
In particular, we must compare the corresponding dominating eigenvalues of a collection of transfer operators, each of which depend on one of these components.
To overcome these difficulties we use the following ideas.
First, we introduce a multi-parameter family of Patterson-Sullivan measures developed in Section \ref{Sec:PS}. 
Second, we compare these measures to a collection of pressure functions and show that the first and second order partial derivatives of these functions coincide at certain points; our proof of this latter part is developed upon an argument of Calegari and Fujiwara \cite[Section 4.5]{CalegariFujiwara2010}. 
This allows us to show that a function $\wt \theta(a,b)$, which we obtain from gluing together our analytic pressure functions, is twice continuously differentiable.
Finally, we realize $\Cc_M$ as the solution to $\wt \theta(a, b) =0$ and then apply the implicit function theorem to conclude the proof.

After showing Theorem \ref{Thm:thm3} (Theorem \ref{Thm:2nd-d-sh}), 
we investigate further properties of the Manhattan curve for strongly hyperbolic and word metrics.
For a pair of word metrics,
we obtain a finer version of Theorem \ref{Thm:thm1} and show that two word metrics are not roughly similar if and only if the corresponding Manhattan curve is globally strictly convex (Theorem \ref{thmwords}).
Furthermore we show that there are two lines, the slopes of which we can express explicitly, which are within bounded distance of the Manhattan curve at $\pm \infty$ (Proposition \ref{Prop:extreme}).
As an application we obtain a  precise large deviation principle for $d_\ast(o, x_n)/n$ when $x_n$ uniformly distributes on the sphere $d(o, x)=n$, for word metrics $d, d_\ast$. We identify the effective domain (where the rate function is finite) with $[\Dil_-, \Dil_+]$ (Theorem \ref{Thm:LDT} and Corollary \ref{Cor:ldt}).

\subsection{Organization of the paper}
In Section \ref{Sec:hyperbolic}, we review basic theory on hyperbolic groups, a classical Patterson-Sullivan construction in a generalized form and a topological flow.
In Section \ref{Sec:general}, we show Theorems \ref{Thm:thm1} and \ref{Thm:thm2}.
In Section \ref{Sec:C2}, we discuss thermodynamic formalism.
We show Theorem \ref{Thm:thm3} in Theorem \ref{Thm:2nd-d-sh}, the analogous result for word metrics in Theorem \ref{Thm:2nd-d-words},
a stronger version of Theorem \ref{Thm:thm1} for word metrics in Theorem \ref{thmwords},
and that $\Dil_-$ and $\Dil_+$ are rational for word metrics in Proposition \ref{Prop:extreme}. The proof of this proposition is based on a finer analysis of a transfer operator in Proposition \ref{Prop:Pmax}.
We exhibit an application to a large deviation principle in Theorem \ref{Thm:LDT} and Corollary \ref{Cor:ldt}.
In Section \ref{Sec:example}, 
we compute explicit examples: the free group of rank $2$ and the $(3, 3, 4)$-triangle group.
In particular, in the case of the latter group, we find a pair of word metrics for which the mean distortion is algebraic irrational.
In Appendix \ref{Sec:Lebesgue}, we show Lemma \ref{Lem:diff} which we use in the proof of main rigidity result Theorem \ref{Thm:rigidity} (the second part of Theorem \ref{Thm:thm1}).

\bigskip

{\bf Notation}:
Throughout the article, we denote by $C, C', C'', \dots$ constants
whose explicit values may change from line to line, and by $C_R, C'_R, C''_R, \dots$
constants with subscript $R$ to indicate their dependency on a parameter $R$.
For real-valued functions $f(t)$ and $g(t)$ in $t \in \R$,
we write $f(t) \asymp g(t)$ if there exist constants $C_1, C_2>0$ independent of $t$
such that $C_1 g(t) \le f(t) \le C_2 g(t)$, and $f(t) \asymp_R g(t)$ if those constants $C_1$ and $C_2$ depend only on $R$.
Further we use the big-$O$ and small-$o$ notations: 
$f(t)=O(g(t))$ if there exist constants $C >0$ and $T>0$ such that $|f(t)| \le C|g(t)|$ for all $t \ge T$,
while $f(t)=O_R(g(t))$ if the implied constant is $C_R$,
and
$f(t)=g(t)+o(t)$ as $t \searrow 0$ if $|f(t)-g(t)|/t \to 0$ as $t \to 0$ for $t>0$.
We say that two measures $\m_1$ and $\m_2$ defined on the common measurable spaces are {\it comparable} if there exists constant $C>0$ such that $C^{-1} \m_1 \le \m_2 \le C \m_1$.
We use the notation $\#A$ which stands for the cardinality of a set $A$.


\section{Preliminaries}

\subsection{Hyperbolic groups}\label{Sec:hyperbolic}
We briefly review some fundamental material concerning hyperbolic groups. 
See the original work by Gromov
\cite{GromovHyperbolic} and \cite{GhysdelaHarpe} for background.
Let $(X, d)$ be a metric space.
The Gromov product is defined by
\[
(x|y)_w:=\frac{d(w, x)+d(w, y)-d(x, y)}{2} \quad \text{for $x, y, w \in X$}.
\]
For $\d \ge 0$, a metric space $(X, d)$ is called $\d$-{\it hyperbolic} if
\[
(x|y)_w \ge \min\big\{(x|z)_w, (z|y)_w\big\}-\d \quad \text{for all $x, y, z, w \in X$}.
\]
We say that a metric space is {\it hyperbolic} if it is $\d$-hyperbolic for some $\d \ge 0$.

Let $\G$ be a finitely generated group.
We call a finite set of generators $S$ of $\G$ {\it symmetric}
if $s^{-1} \in S$ whenever $s \in S$.
The word metric associated to a symmetric finite set of generators $S$
 is defined by
\[
d_S(x, y):=|x^{-1}y|_S \quad \text{for $x, y \in \G$},
\]
where
$|x|_S:=\min\{k \ge 0 \ : \ x=s_1\cdots s_k, \ s_i \in S\}$ and $0$ for the identity element.
We say that $\G$ is a {\it hyperbolic} group if the pair $(\G, d_S)$ is hyperbolic for some word metric $d_S$.
If $(\G, d_S)$ is $\d$-hyperbolic, then for every finite, symmetric set of generators $S'$,
the pair $(\G, d_{S'})$ is $\d'$-hyperbolic for some $\d'$.
A hyperbolic group is called {\it non-elementary} if
it is non-amenable, and {\it elementary} otherwise.
Elementary hyperbolic groups are either finite groups or contain $\Z$ as a finite index subgroup.

We say that two metrics $d$ and $d_\ast$ on $\Gamma$ are {\it quasi-isometric} if there exist constants $L > 0$ and $C \ge 0$ such that
\[
L^{-1}d(x, y)-C \le d_\ast(x, y) \le L d(x, y)+C \quad \text{for all $x, y \in \G$},
\] 
and {\it roughly similar} if there exist constants $\lambda>0$ and $C \ge 0$ such that
\[
\lambda d(x, y)-C \le d_\ast(x, y) \le \lambda d(x, y)+C \quad \text{for all $x, y \in \G$}.
\]
Suppose that $(\G, d)$ is $\d$-hyperbolic.
If $d_\ast$ is roughly similar to $d$,
then $(\G, d_\ast)$ is $\d'$-hyperbolic for some (possibly different) $\d'$.
However, if $d_\ast$ is just quasi-isometric to $d$,
then $(\G, d_\ast)$ is not necessarily hyperbolic.
We will discuss a category of metrics which are hyperbolic and quasi-isometric to some hyperbolic metric in $\G$.

Let $\G$ be a non-elementary hyperbolic group.
We define $\Dc_\G$ to be the set of metrics on $\G$ that are left-invariant, i.e., $d(g x, g y)=d(x, y)$ for all $x, y$ and $g \in \G$,
hyperbolic,
and quasi-isometric to some (equivalently, every) word metric.

\begin{example}
Let $\G$ be the fundamental group of a compact negatively curved manifold $(M, d_M)$.
The group $\G$ acts on the universal cover $(\wt M, d_{\wt M})$ isometrically and freely.
For each point $p$ in $\wt M$, if we define $d(x, y):=d_{\wt M}(xp, yp)$ for $x, y \in \G$,
then $d$ yields a metric, which is left-invariant, hyperbolic and quasi-isometric to a word metric by the Milnor-\v{S}varc lemma. 
Such $d$ therefore belongs to 
$\Dc_\G$.
More generally, if $\G$ is a non-elementary hyperbolic group and acts on a ${\rm CAT}(-1)$-space isometrically and freely with a precompact fundamental domain, then as in the same way above the metric of the ${\rm CAT(-1)}$-space yields a metric on $\G$ in $\Dc_\G$.
\end{example}

A particular subclass of metrics that we will be interested in are strongly hyperbolic metrics.

\begin{definition}\label{Def:SH}
A hyperbolic metric $d$ on $\G$ is called {\it strongly hyperbolic} if 
there exist $L \ge 0, c>0$ and $R_0 \ge 0$ such that for all $x, x', y, y' \in \G$ and all $R \ge R_0$,
the condition
\[
d(x, y)-d(x, x')-d(y, y')+d(x', y') \ge R,
\]
implies that
\[
|d(x, y)-d(x', y)-d(x, y')+d(x', y')| \le L e^{-c R}.
\]
\end{definition}

Every hyperbolic group $\G$ admits a {\it strongly hyperbolic metric} in $\Dc_\G$.
This was shown by Mineyev \cite[Theorem 32]{MineyevFlow} (see also Nica-{\v S}pakula \cite{NicaSpakula}).
We use the existence of such a metric in the course of our proofs.\\

Let us consider a metric $d$ on $\G$.
We say that for an interval $I$ in $\R$, a map $\g:I \to (\G, d)$ is an $(L, C)$-{\it quasi-geodesic} for constants $L>0$ and $C \ge 0$
if it holds that
\[
L^{-1}|s-t|-C \le d(\g(s), \g(t)) \le L|s-t|+C \quad \text{for all $s, t \in I$},
\]
and a $C$-{\it rough geodesic} for $C \ge 0$
if it holds that
\[
|s-t|-C \le d(\g(s), \g(t)) \le |s-t|+C \quad \text{for all $s, t \in I$}.
\]
A {\it geodesic} is a $C$-rough geodesic with $C=0$.
A metric $d$ is called $C$-roughly geodesic if for all $x, y \in \G$ there exists a $C$-rough geodesic $\g:[a, b] \to \G$
such that $\g(a)=x$ and $\g(b)=y$, and called \textit{roughly geodesic} if it is $C$-roughly geodesic for some $C \ge 0$.
If $d \in \Dc_\G$, then $(\G, d)$ is not necessarily a geodesic metric space, but it is a roughly geodesic space \cite[Proposition 5.6]{BonkSchramm}.
In many places, we use the following fact which we refer to as the \textit{Morse lemma}:
if $d$ is a proper (i.e., all balls of finite radius consist of finitely many points) $C_0$-roughly geodesic hyperbolic metric in $\G$,
then every $(L, C)$-quasi-geodesic $\g$ in $(\G, d)$,
there exists a $C_0$-rough geodesic $\g'$ such that $\g$ and $\g'$ are within Hausdorff distance $D$
where $D$ depends only on $C_0, L, C$ and the hyperbolic constant of $d$ (cf.\ \cite[Th\'eor\`emes 21 et 25, Chapitre 5]{GhysdelaHarpe} and \cite[the proof of Proposition 5.6]{BonkSchramm}).

\subsection{Boundary at infinity}

Let us define the (geometric) boundary of $\G$.
Let $o$ be the identity element in $\G$.
Fix $d \in \Dc_\G$ and consider the corresponding Gromov product in $\G$.
We say that a sequence $\{x_n\}_{n=0}^\infty$ is {\it divergent}
if $(x_n|x_m)_o \to \infty$ as $n, m \to \infty$,
and define an equivalence relation in the set of divergent sequences by
\[
\{x_n\}_{n=0}^\infty \sim \{x_n'\}_{n=0}^\infty \iff (x_n|x_m')_o \to \infty \quad \text{as $n, m \to \infty$}.
\]
Let us define $\partial \G$ the set of equivalence classes of divergent sequences in $\G$ and call it the {\it boundary} of $(\G, d)$.
For $\x \in \partial \G$, if $\{x_n\}_{n=0}^\infty \in \x$, then we write $x_n \to \x$ as $n \to \infty$.
We extend the Gromov product to $\G \cup \partial \G$ by setting
\[
(\x|\y)_o:=\sup\big\{\liminf_{n \to \infty}(x_n|y_n)_o \ : \ \{x_n\}_{n=0}^\infty \in \x, \{y_n\}_{n=0}^\infty \in \y \big\},
\]
where if $\x$ or $\y$ is in $\G$, then $(\x|\y)_o$ is defined by taking the constant sequences $\x_n=\x$ or $\y_n=\y$.
Note that if divergent sequences $\{x_n\}_{n=0}^\infty$ and $\{y_n\}_{n=0}^\infty$ are equivalent to $\{x_n'\}_{n=0}^\infty$ and $\{y_n'\}_{n=0}^\infty$, respectively, 
then
\[
\liminf_{n \to \infty}(x_n'|y_n')_o \ge \limsup_{n \to \infty}(x_n|y_n)_o-2\d.
\]
This implies that for all $\x, \y, \z \in \G \cup \partial \G$,
\[
(\x|\y)_o \ge \min\big\{(\x|\z)_o, (\z|\y)_o\big\}-3\d.
\]
Let us define a {\it quasi-metric} by
\[
\rho(\x, \y):=e^{-(\x|\y)_o} \quad \text{for $\x, \y \in \partial \G$}.
\]
In general, $\rho$ is not a metric in $\partial \G$, but it satisfies that
$\rho(\x, \y)=0$ if and only if $\x=\y$, 
$\rho(\x, \y)=\rho(\y, \x)$ for all $\x, \y \in \partial \G$, 
and there exists a constant $C>0$ such that
\[
\rho(\x, \y) \le C \max\big\{\rho(\x, \z), \rho(\z, \y)\big\} \quad \text{for all $\x, \y, \z \in \partial \G$}.
\]
The quasi-metric $\rho$ associated to $d \in \Dc_\G$ defines a topology on $\partial \G$ that is compact,
separable and metrizable.
In fact, for arbitrary two metrics $d, d_\ast \in \Dc_\G$ the corresponding boundaries with the topologies constructed above are homeomorphic. 
We refer to $\partial \G$ the underlying topological space.

\subsection{Shadows}

For all $R \ge 0$ and $x \in \G$, we define the {\it shadow} by
\[
O(x, R):=\big\{\x \in \partial \G \ : \ (\x|x)_o \ge d(o, x)-R\big\}.
\]
Let us denote by $B(\x, r)$ the ball of radius $r \ge 0$ centered at $\x$ in $\partial \G$ relative to the quasi-metric
$\rho(\x, \y)=e^{-(\x|\y)_o}$.
The $\d$-hyperbolic inequality yields the following comparison between balls and shadows.

\begin{lemma}\label{Lem:shadow-ball}
Let $(\G, d)$ be $\d$-hyperbolic for $\d \ge 0$.
For each $\t \ge 0$, 
if $R \ge \t+3\d$, then
for all $\x \in \partial \G$ and all $x \in \G$ such that $(o|\x)_x\le \t$, we have
\[
B(\x, e^{-3\d+R-d(o, x)}) \subset O(x, R) \subset B(\x, e^{3\d+R-d(o, x)}).
\]
\end{lemma}

\proof
See e.g., \cite[Proposition 2.1]{BHM11}; we omit the details.
\qed

Note that if $d$ and $d_\ast$ are in $\Dc_\G$ and $(L, C)$-quasi-isometric, 
then for all $R \ge 0$ there exists $R'\ge 0$ depending on $L, C$ and their hyperbolicity constants 
such that
\[
O(x, R) \subset O'(x, R') \quad \text{for all $x \in \G$},
\]
where $O(x, R)$ (resp.\ $O'(x, R')$) are the shadows defined by $d$ (resp.\ $d_\ast$).
This follows from the stability of rough geodesics and the fact that every pair of points in $\G \cup \partial \G$
are connected by a $C$-rough geodesic in $(\G, d)$ for some $C$.
Therefore omitting the dependency on $d$ in the shadow $O(x, R)$
will not cause any confusion, up to changing the thickness parameter $R$.

\subsection{Hausdorff dimension}
For every $d \in \Dc_\G$,
let $\rho(\x, \y)=\exp(-(\x|\y)_o)$ be the corresponding quasi-metric in $\partial \G$.
Although it is not a metric in general,
we may define the Hausdorff dimension of sets and measures in $\partial \G$ relative to $\rho$ as in the case of metrics.
It is known that there exists a constant $\e>0$ such that 
$\rho^\e$ is bi-Lipschitz to a genuine metric $d_\e$ (e.g., \cite[Proposition 14.5]{Heinonen}),
in which case
the Hausdorff dimension relative to $d_\e$ will be $1/\e$ times the Hausdorff dimension relative to $\rho$.

For every subset $E$ in $\partial \G$,
let us denote by $\rho(E):=\sup\{\rho(\x, \y) \ : \ \x, \y \in E\}$.
For all $s \ge 0$ and $\D>0$, we define
\[
\Hc_\D^s(E, \rho):=\inf\Big\{\sum_{i=0}^\infty \rho(E_i)^s \ : \ E \subset \bigcup_{i=0}^\infty E_i \ \text{and} \ \rho(E_i) \le \D\Big\},
\]
and 
\[
\Hc^s(E, \rho):=\sup_{\D>0}\Hc_\D^s(E, \rho)=\lim_{\D \to 0}\Hc_\D^s(E, \rho).
\]
The {\it Hausdorff dimension} of a set $E$ in $(\partial \G, \rho)$ is defined by
\[
\dim_\H(E, \rho):=\inf\{s \ge 0 \ : \ \Hc^s(E, \rho)=0\}=\sup\{s \ge 0 \ : \ \Hc^s(E, \rho)>0\}.
\]
For every Borel measure $\n$ on $\partial \G$,
the (upper) {\it Hausdorff dimension} of $\n$ is defined by
\[
\dim_\H(\n, \rho):=\inf\{\dim_\H(E, \rho) \ : \ \n\(\partial \G \setminus E\)=0 \ \text{and $E$ is Borel}\}.
\]

\begin{lemma}[The Frostman-type lemma]\label{Lem:Frostman}
Let $\n$ be a Borel probability measure on $(\partial \G, \rho)$.
For $s_1, s_2 \ge 0$,
let 
\[
E(s_1, s_2):=\Big\{\x \in \partial \G \ : \ s_1 \le \liminf_{r \to 0}\frac{\log \n\(B(\x, r)\)}{\log r} \le s_2\Big\},
\]
where $B(\x, r)=\{\y \in \partial \G \ : \ \rho(\x, \y)\le r\}$.
If $\n(E(s_1, s_2))=1$, 
then
\[
s_1 \le \dim_\H(E(s_1, s_2), \rho) \le s_2 \quad \text{and} \quad s_1 \le \dim_\H(\n, \rho) \le s_2.
\]
\end{lemma}

\proof
It suffices to show that $\dim_\H(E(s_1, s_2), \rho)\le s_2$ and $s_1 \le \dim_\H(\n, \rho)$.
These follow as in the case when $\rho$ is a metric; see e.g., \cite[Section 8.7]{Heinonen}.
\qed

\subsection{Distance (Busemann) quasi-cocycles}\label{Sec:Busemann}

For $d \in \Dc_\G$, let us define
\[
\b_w(x, \x):=\sup\big\{\limsup_{n \to \infty}\(d(x, \x_n)-d(w, \x_n)\) \ : \ \{\x_n\}_{n=0}^\infty \in \x\big\}
\]
for $w, x \in \G$ and for $\x \in \partial \G$, and call $\b_w: \G \times \partial \G \to \R$ the {\it Busemann function} based at $w$.
We note that 
\[
d(x, z)-d(o, z)=d(o, x)-2(x|z)_o \quad \text{for $x, z \in \G$},
\]
and thus the $\d$-hyperbolicity implies that
\[
\big|\b_o(x, \x)-\(d(o, x)-2(x|\x)_o\)\big| \le 2 \d \quad \text{for $(x, \x) \in \G \times \partial \G$}.
\]
The Busemann function $\b_o$ satisfies the following cocycle identity with an additive error:
\[
\left|\b_o(xy, \x)-\(\b_o(y, x^{-1}\x)+\b_o(x, \x)\)\right| \le 4\d \quad \text{for $x, y \in \G$ and $\x \in \partial \G$}.
\]

Let us consider a strongly hyperbolic metric $\wh d$ in $\Dc_\G$ (Definition \ref{Def:SH}) and denote by $\langle x|y \rangle_o$ the corresponding Gromov product.
Then there exists a constant $\e>0$ such that
\[
\exp(-\e\langle x|y\rangle_w) \le \exp(-\e\langle x|z\rangle_w)+\exp(-\e \langle z|y\rangle_w) \quad \text{for all $x, y, z, w \in \G$},
\]
\cite[Lemma 6.2, Definition 4.1]{NicaSpakula}
(in fact, this property characterizes the strong hyperbolicity).
This shows that the Gromov product based at $o$ for a strongly hyperbolic metric extends to $\G \cup \partial \G$ as genuine limits.
This also shows that the corresponding Busemann function $\wh \b_o$ is defined as limits and satisfies the cocycle identity,
\[
\wh \b_o(xy, \x)=\wh \b_o(y, x^{-1}\x)+\wh \b_o(x, \x) \quad \text{for $x, y \in \G$ and $\x \in \partial \G$}.
\]
We use a strongly hyperbolic metric to construct an analogue of geodesic flow in Section \ref{Sec:top-flow}.

\subsection{Patterson-Sullivan construction}\label{Sec:PS}

For $d \in \Dc_\G$,
let us denote by
\[
B(x, r):=\big\{y \in \G \ : \ d(x, y) \le r\big\} \quad \text{for $x \in \G$ and $r \ge 0$},
\]
the ball of radius $r$ centered at $x$ relative to $d$.
We define the {\it exponential volume growth rate} relative to $d$ as
\[
v:=\limsup_{r \to \infty}\frac{1}{r}\log \# B(o, r).
\]
Since $\G$ is non-amenable, $v$ is finite and non-zero.

We recall the classical construction of Patterson-Sullivan measures for $d \in \Dc_\G$.
Consider the Dirichlet series
\[
\Pc(s):=\sum_{x \in \G}e^{-s d(o, x)},
\]
which has the divergence exponent $v$.
Suppose for a moment that the series diverges at $s=v$.
Then the sequence of probability measures on $\G$,
\[
\m_s:=\frac{1}{\Pc(s)}\sum_{x \in \G}e^{-s d(o, x)}\d_x,
\]
where $\d_x$ is the Dirac measure at $x$, considered as measures on the compactified space $\G \cup \partial \G$,
has a convergent subsequence as $s \searrow v$.
A limit point $\m$ is a probability measure supported on $\partial \G$,
and there exists a constant $C_\d>0$ such that for $x \in \G$ and for $\x \in \partial \G$,
\begin{equation}\label{Eq:PS0}
C_\d^{-1}e^{-v \b_o(x, \x)} \le \frac{d x_\ast \m}{d \m}(\x) \le C_\d e^{-v \b_o(x, \x)}.
\end{equation}
All limit points satisfy the above estimates \eqref{Eq:PS0}.
If the series $\Pc(s)$ does not diverge at $s=v$, 
then a slight modification yields a measure satisfying \eqref{Eq:PS0}.
We call a probability measure satisfying \eqref{Eq:PS0} a {\it Patterson-Sullivan measure} for $d \in \Dc_\G$.
For details, see \cite[Th\'eor\`eme 5.4]{Coornaert1993}.

The above construction applies to the following setting where the distance is replaced by a more general function.
Let us consider a function $\psi: \G \times \G \to \R$ and define
\[
\psi(x|y)_z:=\frac{\psi(x, z)+\psi(z, y)-\psi(x, y)}{2} \quad \text{for $x, y, z \in \G$},
\]
as a generalization of the Gromov product. Note that the order of $x, y, z$ matters since $\psi$ may not satisfy $\psi(x, y)=\psi(y, x)$.
We assume that $\psi(\cdot \ | \ \cdot)_o$ admits a ``quasi-extension" to $\G \times (\G \cup \partial \G)$, i.e.,
there exist a function $\psi(\cdot \ | \ \cdot)_o: \G \times (\G \cup \partial \G) \to \R$ and a constant $C\ge 0$ 
such that
\begin{equation}\label{Eq:QE}\tag{QE}
\limsup_{n \to \infty}\psi(x|\x_n)_o-C \le \psi(x|\x)_o \le \liminf_{n \to \infty}\psi(x|\x_n')_o+C
\end{equation}
for all $(x, \x) \in \G \times (\G \cup \partial \G)$ and for all $\{\x_n\}_{n=0}^\infty, \{\x_n'\}_{n=0}^\infty \in \x$.
This allows us to define the following function analogous to the Busemann function for $(x, \x) \in \G \times \partial \G$,
\[
\b^\psi_o(x, \x):=\sup\big\{\limsup_{n \to \infty}(\psi(x, \x_n)-\psi(o, \x_n)) \ : \ \{\x_n\}_{n=0}^\infty \in \x\big\}.
\]
Furthermore, if $\psi$ is $\G$-invariant, i.e., $\psi(g x, g y)=\psi(x, y)$ for all $g, x, y \in \G$,
then $\b^\psi_o$ satisfies that the quasi-cocycle relation:
\[
|\b^\psi_o(xy, \x)-(\b^\psi_o(y, x^{-1}\x)+\b^\psi_o(x, \x))| \le 4C.
\]

Recall that if $d \in \Dc_\G$, then $(\G, d)$ is a $C$-rough geodesic metric space for some $C \ge 0$.
Let us consider the following ``rough geodesic'' condition:
for all large enough $C, R \ge 0$, there exists $C_0 \ge 0$ such that for all $C$-rough geodesics $\g$ between $x$ and $y$,
and for all $z$ in the $R$-neighborhood of $\g$,
\begin{equation}\label{Eq:RG}\tag{RG}
|\psi(x, y)-\(\psi(x, z)+\psi(z, y)\)| \le C_0.
\end{equation}
If $\psi$ satisfies \eqref{Eq:RG} relative to $d \in \Dc_\G$,
then there exists a constant $C'$ such that for a large enough $R$ and for all $x \in \G$,
\begin{equation}\label{Eq:corRG}
|\b_o^\psi(x, \x)+\psi(o, x)| \le C' \quad \text{for all $\x \in O(x, R)$}.
\end{equation}

\begin{definition}\label{Def:tempered}
We say that a function $\psi: \G\times \G \time \G \to \R$ is a {\it tempered potential} relative to $d \in \Dc_\G$
if $\psi$ satisfies \eqref{Eq:QE} and \eqref{Eq:RG} relative to $d$.
\end{definition}

\begin{example}\label{Ex:tempered}
For all $d, d_\ast \in \Dc_\G$, by the Morse lemma, $d_\ast$ satisfies \eqref{Eq:RG} relative to $d$.
This implies that for every $a \in \R$,
the function $\psi_a=a d_\ast$ satisfies \eqref{Eq:RG} relative to $d$.
Moreover $\psi_a$ satisfies \eqref{Eq:QE} and is $\G$-invariant.
Therefore $\psi_a=a d_\ast$ is a $\G$-invariant tempered potential relative to $d$ for every $a \in \R$.
The same argument applies to an arbitrary triple $d, d_\ast, d_{\ast \ast} \in \Dc_\G$ and every linear combination
\[
\psi_{a, b}:=a d_\ast+b d \quad \text{for $a, b \in \R$}.
\]
For every $a, b \in \R$, the function $\psi_{a, b}$ is a $\G$-invariant tempered potential relative to $d_{\ast \ast}$.
The functions $\psi_a$ and $\psi_{a, b}$ are the main tools in Sections \ref{Sec:proofofC1} and \ref{Sec:proofofC2}
respectively.
\end{example}

For $d \in \Dc_\G$,
let $\psi$ be a $\G$-invariant tempered potential relative to $d$.
We say that a probability measure $\m$ on $\partial \G$ 
satisfies the ``quasi-conformal'' property with exponent $\th \in \R$ relative to $(\psi, d)$
if there exists a constant $C$ depending only on $\psi$ and $d$ such that
\begin{equation}\label{Eq:generalPS}\tag{QC}
C^{-1} \le \exp\(\b_o^\psi(x, \x)+\th \b_o(x, \x)\) \cdot \frac{dx_\ast \m}{d\m}(\x) \le C,
\end{equation}
for all $x \in \G$ and $\m$-almost every $\x$ in $\partial \G$,
where $\b_o$ is the Busemann function associated to $d$.
We simply say that $\m$ satisfies \eqref{Eq:generalPS} if $\th$ and $(\psi, d)$ are fixed and apparent from the context.

\begin{proposition}\label{Prop:generalPS}
For $d \in \Dc_\G$,
let $\psi$ be a $\G$-invariant tempered potential relative to $d$.
Then the abscissa of convergence $\th$ of the series in $s$,
\begin{equation}\label{Eq:gPS}
\sum_{x \in \G}\exp(- \psi(o, x)-s d(o, x)),
\end{equation}
is finite
and there exists a probability measure $\m_\psi$ on $\partial \G$ 
satisfying \eqref{Eq:generalPS} with exponent $\th$ relative to $(\psi, d)$.
Moreover, every finite Borel measure $\m$ satisfying \eqref{Eq:generalPS} has the property:
\begin{equation}\label{Eq:generalGibbs}
C'^{-1} \exp(-\psi(o, x)-\th d(o, x)) \le \m\(O(x, R)\) \le C' \exp(-\psi(o, x)-\th d(o, x)),
\end{equation}
for all $x \in \G$, where $C'$ is a constant depending on $C$, $C_0$ and $R$.
\end{proposition}

\proof
Note that $\th$ is given by
\[
\limsup_{n \to \infty}\frac{1}{n}\log \sum_{x \in S(n, R_0)}e^{-\psi(o, x)} \quad \text{where $S(n, R_0)=\{x \in \G \ : |d(o, x)-n| \le R_0\}$}.
\]
Since $\psi$ satisfies \eqref{Eq:RG} relative to $d$ and is $\G$-invariant, then for all $n, m \ge 0$,
\[
\sum_{x \in S(n+m, R_0)}e^{-\psi(o, x)} \le e^C\sum_{x \in S(n, R_0)}e^{-\psi(o, x)}\cdot \sum_{x \in S(m, R_0)}e^{-\psi(o, x)},
\]
which implies that $\th$ is finite.
Let us define the family of probability measures for $s>\th$ by
\begin{equation*}
\m_{\psi, s}:=\frac{\sum_{x \in \G}\exp(-\psi(o, x)-sd(o, x))\d_x}{\sum_{x \in \G}\exp(-\psi(o, x)-sd(o, x))}.
\end{equation*}
If the series \eqref{Eq:gPS} diverges at $\th$, then letting $s \searrow \th$ yields a weak limit $\m_\psi$ after passing to a subsequence.
The measure $\m_\psi$ is supported on $\partial \G$.
For $x, y \in \G$, we have that
\begin{align*}
x_\ast \m_{\psi, s}(y)	&=\frac{\exp(-\psi(o, x^{-1}y)-sd(o, x^{-1}y))}{\sum_{z \in \G}\exp(-\psi(o, z)-sd(o, z))}\\
									&=\frac{\exp(-\psi(o, x^{-1}y))}{\exp(-\psi(o, y))}e^{-s(d(o, x^{-1}y)-d(o, y))}\m_{\psi, s}(y).
\end{align*}
By the assumption, $\psi(x, y)-\psi(o, y)$ is $\b_o^\psi(x, \x)$ up to a uniform additive constant as $y$ tends to $\x$.
Further $d(x, y) -d(o, y)$ coincides with $\b_o(x, \x)$ up to a constant depending only on the hyperbolicity constant of $d$ uniformly on a neighborhood of $\x$ in $\G \cup \partial \G$.
This yields \eqref{Eq:generalPS}.
If the series \eqref{Eq:gPS} does not diverge at $\th$,
then the argument as in the classical setting provides \eqref{Eq:generalPS} (cf.\ \cite[Th\'eor\`eme 5.4]{Coornaert1993} and \cite[Theorem 3.3]{THaus} for a special case).

Further since $\psi$ satisfies \eqref{Eq:RG} relative to $d$, we have \eqref{Eq:corRG}.
Suppose that a finite measure $\m$ satisfies \eqref{Eq:generalPS},
and $\m$ is a probability measure without loss of generality.
Then for all $x \in \G$,
\[
x_\ast \m\(O(x, R)\) \asymp_{C, \th} \exp(\psi(o, x)+\th d(o, x))\m\(O(x, R)\).
\]
For all small enough $0<\e_0<1$, there exists a large enough $R$ such that 
\[
\m\(x^{-1}O(x, R)\) \ge 1-\e_0 \quad \text{for all $x \in \G$},
\]
(cf.\ \cite[Proposition 6.1]{Coornaert1993}), and thereby we obtain \eqref{Eq:generalGibbs}.
\qed

\begin{lemma}\label{Lem:Coornaert}
For $d \in \Dc_\G$, if $\psi$ is a $\G$-invariant tempered potential relative to $d$,
then there exist constants $\th \in \R$ and $C, R_0>0$ such that
for all $n \ge 0$,
\[
C^{-1}e^{\th n} \le \sum_{x \in S(n, R_0)} e^{-\psi(o, x)}\le C e^{\th n},
\]
where $S(n, R_0):=\big\{x \in \G \ : \ |d(o, x)-n| \le R_0\big\}$.
\end{lemma}

We say that $\th$ is the {\it exponent} of $(\psi, d)$ abusing the notation;
the proof actually shows that if there is a finite Borel measure $\m$ satisfying \eqref{Eq:generalPS} relative to $(\psi, d)$ with some exponent,
then that exponent has to be $\th$.

\proof[Proof of Lemma \ref{Lem:Coornaert}]
For $(\G, d)$, fix large enough constants $R_0, R>0$ so that for every $n \ge 0$ the shadows $O(x, R)$ for $x \in S(n, R_0)$ cover $\partial \G$.
Since $(\G, d)$ is hyperbolic, there exists a constant $M$ such that for every $n$, each $\x \in \partial \G$ is included in at most $M$ shadows
$O(x, R)$ with $x \in S(n, R_0)$. 
By Proposition \ref{Prop:generalPS}, there exists a probability measure $\m_\psi$ which satisfies \eqref{Eq:generalGibbs}.
The first inequality in \eqref{Eq:generalGibbs} shows that for all $n \ge 0$,
\[
e^{- \th (n+R_0)}\sum_{x \in S(n, R_0)} e^{-\psi(o, x)} \le C\sum_{x \in S(n, R_0)}\m_\psi\(O(x, R)\) 
\le C M.
\]
The second inequality in \eqref{Eq:generalGibbs} shows that for all $n \ge 0$,
\[
1=\m_\psi\(\partial \G\)\le \sum_{x \in S(n, R_0)}\m_\psi\(O(x, R)\) \le C e^{-\th (n-R_0)}\sum_{x \in S(n, R_0)}e^{-\psi(o, x)},
\]
hence we obtain the claim.
\qed

We say that a (finite) Borel measure $\m$ on $\partial \G$ is {\it doubling} relative to a quasi-metric $\rho$
if $\m\(B(\x, r)\)>0$ for all $\x \in \partial \G$ and for all $r>0$, and there exists a $C$ such that for every $r \ge 0$ and $\x \in \partial \G$,
\[
\m\(B(\x, 2 r)\) \le C \m\(B(\x, r)\),
\] 
where $B(\x, r)$ is the ball defined by $\rho$ in $\partial \G$.

\begin{lemma}\label{Lem:ergodic}
For $d \in \Dc_\G$, if $\psi$ is a $\G$-invariant tempered potential relative to $d$,
then every finite Borel measure $\m$ on $\partial \G$ satisfying \eqref{Eq:generalPS} is doubling relative to a quasi-metric $\rho$.
Moreover
arbitrary two finite Borel measures on $\partial \G$ satisfying \eqref{Eq:generalPS} with the same exponent and $(\psi, d)$ are mutually absolutely continuous and their densities are uniformly bounded from above and below.
In particular,
every measure $\m_\psi$ is ergodic with respect to the $\G$-action on $\partial \G$,
i.e., every $\G$-invariant Borel set $A$ in $\partial \G$ satisfies that either $\m_\psi(A)=0$ or $\m_\psi(\partial \G \setminus A)=0$.
\end{lemma}
\proof
First by Proposition \ref{Prop:generalPS}, every finite Borel measure $\m$ with \eqref{Eq:generalPS} satisfies \eqref{Eq:generalGibbs}, which shows that $\m(O(x, R))>0$ for all $x \in \G$ and
\[
\m(O(x, 2R)) \asymp_R \m(O(x, R)) \quad \text{for all $x \in \G$},
\]
where $R$ is a large enough fixed constant.
Applying to this estimate for finitely many times if necessary, by Lemma \ref{Lem:shadow-ball}
we find that $\m$ is doubling relative to $\rho$.

Next \eqref{Eq:generalGibbs} implies that for arbitrary two finite Borel measures $\m, \m'$ satisfying \eqref{Eq:generalPS} with the common exponent and $(\psi, d)$, the ratio of the measures of balls relative to $\m$ and $\m'$ 
are uniformly bounded from above and below.
Since both measures are doubling relative to $\rho$, the Vitali covering theorem \cite[Theorem 1.6]{Heinonen} (adapted to a quasi-metric $\rho$) shows that $\m$ and $\m'$ are mutually absolutely continuous and their densities are uniformly bounded from above and below.

Finally for $\m$ satisfying \eqref{Eq:generalPS}, if $A$ is an arbitrary $\G$-invariant Borel set in $\partial \G$ such that $\m(A)>0$,
then the restriction $\m|_A$ also satisfies \eqref{Eq:generalPS} with the same exponent and $(\psi, d)$.
Therefore what we have shown implies that $\m|_A \asymp \m$ and thus $\m(\partial \G \setminus A)=0$.
This in particular applies to $\m_\psi$.
\qed

A central example of the construction is a family of measures $\m_{a, b}$ for $(a, b) \in \Cc_M$ associated to a pair $(d, d_\ast)$.
Let us single out the following corollary which we use in Section \ref{Sec:top-flow}.

\begin{corollary}\label{Cor:ab}
Let us consider a pair $d, d_\ast \in \Dc_\G$.
\begin{itemize}
\item[(1)] 
For each $(a, b) \in \Cc_M$, 
there exists a probability measure $\m_{a, b}$ on $\partial \G$ such that
for all $x \in \G$,
\begin{equation*}
C_{a, b}^{-1}e^{-a\b_{\ast o}(x, \x)-b\b_o(x, \x)} \le \frac{d x_\ast \m_{a, b}}{d \m_{a, b}}(\x) \le C_{a, b} e^{-a\b_{\ast o}(x, \x)-b\b_o(x, \x)},
\end{equation*}
where $\b_{\ast o}$ and $\b_o$ are Busemann functions for $d_\ast$ and $d$, respectively, and $C_{a, b}$ is a constant of the form $C_{a, b}=C_{d_\ast}^{|a|}C_d^{|b|}$.
Moreover, we have that
\[
C'^{-1} \exp(-a d_\ast(o, x)-b d(o, x)) \le \m_{a, b}\(O(x, R)\) \le C' \exp(-a d_\ast(o, x)-b d(o, x)),
\]
for all $x \in \G$, where $C'$ is a constant depending on $C_{a, b}$ and $R$.
\item[(2)]
For every $a \in \R$,
\begin{equation*}
\th(a) =\lim_{n \to \infty}\frac{1}{n}\log \sum_{x \in S(n, R_0)}e^{-ad_\ast(o, x)},
\end{equation*}
where $S(n, R_0):=\big\{x \in \G \ : \ |d(o, x)-n| \le R_0\big\}$ for some constant $R_0$,
and the function $\th$ is convex and continuous on $\R$.
\item[(3)]
For each $(a, b) \in \Cc_M$,
every probability measure $\m_{a, b}$ is ergodic with respect to the $\G$-action on $\partial \G$.
\end{itemize}
\end{corollary}

\proof
For each $a \in \R$,
if we let $\psi(x, y)=a d_\ast(x, y)$, then $\psi$ is a $\G$-invariant tempered potential (Example \ref{Ex:tempered})
and $\th=b$ for $(a, b) \in \Cc_M$.
Therefore Proposition \ref{Prop:generalPS} implies (1), where
the constant $C_{a, b}=C_{d_\ast}^{|a|}C_d^{|b|}$ is obtained from the proof of Proposition \ref{Prop:generalPS}.
Lemma \ref{Lem:Coornaert} and the H\"older inequality imply that $\th(a)$ is finite and convex in $a \in \R$, hence $\th$ is continuous on $\R$, showing (2),
and Lemma \ref{Lem:ergodic} shows (3).
\qed

Note that letting $v$ and $v_\ast$ be the exponential volume growth rates for $d$ and $d_\ast$ respectively,
we have that $(0, v), (v_\ast, 0) \in \Cc_M$, and
$\m_{0, v}$ and $\m_{v_\ast, 0}$ are (classical) Patterson-Sullivan measures for $d$ and $d_\ast$, respectively.

\subsection{Topological flow}\label{Sec:top-flow}
\def\ev{{\bf ev}}

In this section, we follow the discussion in \cite[Section 3]{TopFlows}.
Let $\partial^2 \G:=(\partial \G)^2 \setminus\{\rm diagonal\}$,
where $\G$ acts on $\partial^2 \G$ by $x\cdot(\x, \y):=(x\x, x\y)$ for $x \in \G$ and $(\x, \y) \in \partial^2 \G$.
Consider the space $\partial^2 \G \times \R$ and fix a strongly hyperbolic metric $\wh d \in \Dc_\G$.
There exists a constant $C\ge 0$
such that for each $(\x, \y) \in \partial^2 \G$
there is a $C$-rough geodesic $\g_{\x, \y}: \R \to (\G, \wh d\,)$ satisfying that $\g_{\x, \y}(-t)\to \x$ and $\g_{\x, \y}(t) \to \y$ as $t \to \infty$, respectively \cite[Proposition 5.2 (3)]{BonkSchramm}.
Shifting the parameter $t \mapsto t+T$ by some $T$ if necessary,
we parametrize $\g_{\x, \y}$ in such a way that
\[
\wh d(\g_{\x, \y}(0), o)=\min_{t \in \R}\wh d(\g_{\x, \y}(t), o).
\]
We define
\[
\ev: \partial^2 \G \times \R \to \G \ \ \text{ by } \ \ \ev(\x, \y, t):=\g_{\x, \y}(t).
\]
Note that the map $\ev$ depends on the choice of $C$-rough geodesics,
however, every other choice yields the map whose image lies in a uniformly bounded distance:
if $\g_{\x, \y}$ and $\g'_{\x, \y}$ are two $C$-rough geodesics with the same pair of extreme points,
then
\[
\max_{t \in \R}\wh d(\g_{\x, \y}(t), \g'_{\x, \y}(t)) < C',
\]
for some positive constant $C'$ depending only on the metric.
Let us endow the space of $C$-rough geodesics on $(\G, \wh d\,)$ with the point-wise convergence topology.
We define $\ev: \partial^2 \G \times \R \to \G$ as a measurable map by assigning $\g_{\x, \y}$ to $(\x, \y) \in \partial^2 \G$ in a Borel measurable way:
first fix a set of generators $S$ in $\G$ and an order on it,
second consider $C$-rough geodesics evaluated on the set of integers as sequences of group elements and choose lexicographically minimal ones $\g^0_{\x, \y}$ for each $(\x, \y) \in \partial^2 \G$,
and finally define $\g_{\x, \y}(t):=\g^0_{\x, \y}(\lfloor t\rfloor)$ for $t \in \R$ where $\lfloor t\rfloor$ stands for the largest integer at most $t$.

Letting $\wh \b_o: \G \times \partial \G \to \R$ be the Busemann function based at $o$ associated with $\wh d$,
we define the cocycle 
\[
\k:\G \times \partial^2 \G \to \R, \quad \k(x, \x, \y):=\frac{1}{2}\(\wh \b_o(x^{-1}, \x)-\wh \b_o(x^{-1}, \y)\),
\]
where the cocycle identity for $\k$ follows from that of $\wh \b_o$ (Section \ref{Sec:Busemann}).
Then, $\G$ acts on $\partial^2 \G \times \R$ through $\k$ by
\[
x\cdot(\x, \y, t):=(x \x, x\y, t-\k(\x, \y, t)).
\]
Let us call this $\G$-action the {\it $(\G, \k)$-action} on $\partial^2 \G \times \R$.
It is shown that the $(\G, \k)$-action on $\partial^2 \G \times \R$ is properly discontinuous and cocompact,
namely, the quotient topological space $\G \backslash (\partial^2 \G \times \R)$ is compact \cite[Lemma 3.2]{TopFlows}.
Let
\[
\Fc_\k:=\G \backslash (\partial^2 \G \times \R),
\]
where we define a continuous $\R$-action as in the following.
The $\R$-action $\wt \F$ on $\partial^2 \G \times \R$ is defined by
the translation in the $\R$-component:
\[
\wt \F_s(\x, \y, t):=(\x, \y, t+s).
\]
This action and the $(\G, \k)$-action commute, and thus
the $\R$-action $\wt \F$ descends to the quotient
\[
\F_s[\x, \y, t]:=[\x, \y, t+s] \quad \text{for $[\x, \y, t] \in \Fc_\k$}.
\]
Then $\R$ acts on $\Fc_\k$ via $\F$ continuously.
We call the $\R$-action $\F$ on $\Fc_\k$ the {\it topological flow} (or, simply the {\it flow}) on $\Fc_\k$.

Let us consider finite measures invariant under the flow on $\Fc_\k$.
Let $\L$ be a $\G$-invariant Radon measure on $\partial^2 \G$,
i.e., $x_\ast \L=\L$ for all $x \in \G$ and $\L$ is Borel regular and finite on every compact set.
Then every measure of the form $\L \otimes dt$ where $dt$ is the (normalized) Lebesgue measure on $\R$ yields
a flow invariant finite measure on $\Fc_\k$.
Namely,
for every $\G$-invariant Radon measure $\L$ on $\partial^2 \G$,
there exists a unique finite Radon measure $m$ invariant under the flow on $\Fc_\k$
such that
\begin{equation}\label{Eq:top-flow}
\int_{\partial^2 \G\times \R}f\,d\L\otimes dt=\int_{\Fc_\k}\wbar f\,dm
\end{equation}
for all compactly supported continuous functions $f$ on $\partial^2 \G \times \R$,
where $\wbar f$ is the $\G$-invariant function
\[
\wbar f(\x, \y, t):=\sum_{x \in \G}f(x\cdot(\x, \y, t)),
\]
considered as a function on $\Fc_\k$ \cite[Lemma 3.4]{TopFlows} (and we further note that every continuous function $\f$ on $\Fc_\k$ is of the form $\f=\wbar f$ by invoking Urysohn's lemma).
If we take a Borel fundamental domain $D$ in $\partial^2 \G \times \R$ with respect to the $(\G, \k)$-action 
and a measurable section $\i: \Fc_\k \to D$,
then 
\[
\L \otimes dt=\sum_{x \in \G}x_\ast(\i_\ast m).
\]
Note that it is not necessarily the case that the restriction $\L \otimes dt|_D$ coincides with $\i_\ast m$ unless the $(\G, \k)$-action is free.
We always normalize $\L$ in such a way that the corresponding flow invariant measure $m$ has total measure $1$ (and so is a probability measure on $\Fc_\k$).

For all $d \in \Dc_\G$, an associated Patterson-Sullivan measure $\m$ on $\partial \G$
yields a $\G$-invariant Radon measure $\L_d$ on $\partial^2 \G$ equivalent to
\[
\exp\(2 v(\x|\y)_o\)\m\otimes \m,
\]
with the Radon-Nikodym density uniformly bounded from above and from below by positive constants, 
and the corresponding flow invariant probability measure $m_d$ on $\Fc_\k$
is ergodic with respect to the flow, i.e., for every Borel set $A$ such that $\F_{-t}(A)=A$ for all $t \in \R$,
either $m_d(A)=0$ or $1$ 
\cite[Proposition 2.11 and Theorem 3.6]{TopFlows}.
The same construction applies to measures $\m_{a, b}$ for all $(a, b) \in \Cc_M$.
\begin{proposition}\label{Prop:Hopf}
For each $(a, b) \in \Cc_M$, 
there exists a $\G$-invariant Radon measure $\L_{a, b}$ on $\partial^2 \G$ equivalent to
\begin{equation}\label{Eq:n}
\exp\(2a(\x|\y)_{\ast o}+2b(\x|\y)_o\)\m_{a, b}\otimes \m_{a, b},
\end{equation}
with the Radon-Nikodym density uniformly bounded from above and below by positive constants of the form $C_{d_\ast}^{|a|}C_d^{|b|}$.
Moreover, $\L_{a, b}$ is ergodic with respect to the $\G$-action on $\partial^2 \G$, i.e., for every $\G$-invariant Borel set $A$ in $\partial^2 \G$, either the set $A$ or the complement has zero $\L_{a, b}$-measure,
and the corresponding flow invariant probability measure $m_{a, b}$
is ergodic with respect to the flow on $\Fc_\k$.
\end{proposition}
\proof
If we denote the measure \eqref{Eq:n} by $\n$,
then we have that 
$C^{-1}\le \frac{d x_\ast \n}{d\n}(\x, \y) \le C$ for all $x \in \G$ and for $\n$-almost all $(\x, \y) \in \partial^2 \G$,
where $C$ is a positive constant of the form $C_{d_\ast}^{|a|}C_d^{|b|}$.
If we define 
\[
\f(\x, \y):=\sup_{x \in \G}\frac{d x_\ast \n}{d\n}(\x, \y),
\]
then $\L_{a, b}:=\f(\x, \y)\n$ is a $\G$-invariant Radon measure, which is desired.
The details follow as in Proposition 2.11, Theorem 3.6 and Corollary 3.7 in \cite{TopFlows}.
\qed

\begin{lemma}\label{Lem:pure-type}
If $\L$ and $\L'$ are $\G$-invariant ergodic Radon measures on $\partial^2 \G$, 
then either $\L$ and $\L'$ are mutually singular, or there exists a positive constant $c>0$
such that $\L=c \L'$.
\end{lemma}

\proof
Let us decompose $\L$ as a sum of two measures $\L=\L_{\ac}+\L_{\sing}$
where $\L_{\ac}$ (resp.\ $\L_{\sing}$) is the absolutely continuous (resp.\ singular) part with respect to $\L'$.
Note that $\L_{\ac}$ and $\L_{\sing}$ are $\G$-invariant Radon measures since $\L$ and $\L'$ are so.
Suppose that $\L_{\ac}\neq 0$.
Then the Radon-Nikodym density $d\L_{\ac}/d\L'$ is locally integrable and $\G$-invariant,
and thus constant since $\L'$ is ergodic with respect to the $\G$-action.
Hence there exists a positive constant $c>0$ such that $\L_{\ac}=c \L'$,
and since $\L$ is ergodic with respect to the $\G$-action,
$\L_{\sing}=0$ and $\L=c\L'$, as desired.
\qed

\begin{corollary}\label{Cor:limit}
For each $(a, b) \in \Cc_M$, let $m_{a, b}$ be the flow invariant probability measure on $\Fc_\k$ corresponding to the (normalized) $\G$-invariant Radon measure $\L_{a, b}$ on $\partial^2\G$.
If $(a, b) \to (a_0, b_0)$ in $\Cc_M$,
then $m_{a, b}$ weakly converges to $m_{a_0, b_0}$.
\end{corollary}

\proof
If $(a, b) \to (a_0, b_0)$, then 
up to taking a subsequence,
there exists a normalized $\G$-invariant Radon measure $\L_\ast$ on $\partial^2 \G$ such that
$\int_{\partial^2 \G}f\,d\L_{a, b}$ converges to 
$\int_{\partial^2 \G}f\,d\L_\ast$
for each compactly supported continuous function $f$ on $\partial^2 \G$ (where we use the fact that $\partial^2 \G$ is $\s$-compact).
Let $\L_\ast$ be an arbitrary such limit point.
Taking a further subsequence, we have that $\m_{a, b}$ weakly converges to some probability measure $\m_\ast$,
which is comparable to $\m_{a_0, b_0}$ by Proposition \ref{Prop:generalPS} (in the form of Corollary \ref{Cor:ab}) and Lemma \ref{Lem:ergodic}.
This together with Proposition \ref{Prop:Hopf} shows that $\L_\ast$ is equivalent to $\L_{a_0, b_0}$.
Lemma \ref{Lem:pure-type} implies that $\L_\ast$ coincides with $\L_{a_0, b_0}$ up to a multiplicative constant,
and if they are normalized, then $\L_\ast=\L_{a_0, b_0}$. 
Therefore by \eqref{Eq:top-flow} for every limit point $m_\ast$ of $m_{a, b}$ as $(a, b) \to (a_0, b_0)$, we have that $m_\ast=m_{a_0, b_0}$, hence $m_{a, b}$ weakly converges to $m_{a_0, b_0}$.
\qed


\section{The Manhattan Curve for general hyperbolic metrics}\label{Sec:general}

\subsection{Fundamental properties of the Manhattan curve}

For $d \in \Dc_\G$,
we recall that the stable translation length of $x \in \G$ with respect to $d$ is given by
$\ell[x]=\lim_{n \to \infty}d(o, x^n)/n$,
where $\ell$ defines a function on the set of conjugacy classes $\conj$
and $[x]$ denotes the conjugacy class of $x \in \G$.
For $d_\ast \in \Dc_\G$, we denote the corresponding function by $\ell_\ast$.
For $a, b \in \R$,
let
\[
\Qc(a, b):=\sum_{[x] \in \conj}\exp\(-a\ell_\ast[x]-b\ell[x]\),
\]
and for each fixed $a \in \R$, we define $\Theta(a)$ as the abscissa of convergence of $\Qc(a, b)$ in $b$.
Recall that for $a\in \R$,
we have defined $\th(a)$ as the abscissa of convergence of $\Pc(a, b)$ in $b$, 
where
\[
\Pc(a, b)=\sum_{x \in \G}\exp\(-ad_\ast(o, x)-bd(o, x)\).
\]

\begin{proposition}\label{Prop:conjugacy}
For all $a \in \R$, we have $\th(a)=\Theta(a)$.
\end{proposition}

We will also call the functions $\th$ as well as $\Theta$ the Manhattan curve for the pair $(d, d_\ast)$.
The proof follows the ideas from \cite[Section 5]{CoornaertKnieper} and Knieper \cite[Section II]{Knieper} (the latter is indicated in \cite[Section 4.1]{BurgerManhattan}); we provide the main argument adapted to our setting for the sake of completeness.
We use the following lemma in the proof.

\begin{lemma}\label{Lem:translation}
For $d \in \Dc_\G$,
there exists a constant $C_0$ such that 
for all $x \in \G$,
if $d(o, x)-2(x|x^{-1})_o>C_0$,
then 
\[
\left|\ell[x]-\(d(o,x)-2(x|x^{-1})_o\)\right| \le C_0,
\]
and there exists $p \in \G$ such that
$|\ell[x] - d(p, xp)| \le C_0$.
\end{lemma}

\proof[Sketch of proof]
Recall that if $d \in \Dc_\G$,
then $(\G, d)$ is a $C$-rough geodesic metric space, i.e.,
for all $x, y \in \G$, there exists a $C$-rough geodesic $\g:[a, b] \to \G$ such that $\g(a)=x$ and $\g(b)=y$.
We provide an outline of the proof when $d$ is geodesic for the sake of convenience (a detailed proof is found in \cite[Proposition 5.8]{MT}); the same argument applies to $C$-rough geodesic metrics with slight modifications.
For all $x, y \in \G$, let us denote by $[x, y]$ the image of a geodesic between $x$ and $y$.
On the one hand,
for each $x \in \G$,
let us consider a geodesic triangle on $o, x$ and $x^2$,
and
take $p$ as a midpoint of $[o, x]$.
If $d(o, x)-2(x|x^{-1})_o$ is large enough,
then
$d(p, x) >(x|x^{-1})_o$, and thus
\begin{equation}\label{Eq:R-1}
d(p, xp) \le d(o, x)-2(x|x^{-1})_o+\d.
\end{equation}
On the other hand, 
for an arbitrary positive integer $n > 0$,
let us consider a geodesic $\g:=[o, x^n]$.
It holds that 
if $d(o, x)-2(x|x^{-1})_o$ is large enough,
then
\begin{equation}\label{Eq:R0}
\max_{0 \le k \le n}d(x^k, \g) \le (x|x^{-1})_o+3\d.
\end{equation}
Indeed,
let us write $x_k:=x^k$ for $0 \le k \le n$ and use $p_k$ to denote a nearest point from $x_k$ on $\g$.
Suppose that $x_k$ is one of the furthest points among $x_0, \dots, x_n$ from $\g$.
Consider a geodesic quadrangle on $x_{k-1}$, $p_{k-1}$, $p_{k+1}$ and $x_{k+1}$ in this order.
Let $q$ be a nearest point from $x_k$ on $[x_{k-1}, x_{k+1}]$.
By $\d$-hyperbolicity there is a point $r$ with $d(q, r) \le 2 \d$ on 
$[x_{k-1}, p_{k-1}]\cup [p_{k-1}, p_{k+1}]\cup [p_{k+1}, x_{k+1}]$,
and we see that $r$ is in fact on $[p_{k-1}, p_{k+1}]$; this shows \eqref{Eq:R0}.

Finally, \eqref{Eq:R0} together with the triangle inequality implies that 
for all $n >0$,
\begin{equation*}\label{Eq:R4}
d(o, x_n) \ge d(o, x_{n-1})+d(o, x)-2(x|x^{-1})_o-6\d,
\end{equation*}
which yields
$\ell[x] \ge d(o, x)-2(x|x^{-1})_o-6\d$.
Combining this with \eqref{Eq:R-1} shows the claim since $\ell[x] \le d(p, xp)$.
\qed

\proof[Proof of Proposition \ref{Prop:conjugacy}]
To simplify the notations we write 
$|x|:=d(o, x)$ and $|x|_\ast:=d_\ast(o, x)$.
Note that for all large enough $L \ge 0$ and for each $a \in \R$,
\[
\th(a)=\limsup_{n \to \infty}\frac{1}{n}\log \sum_{||x|-n| \le L}e^{- a|x|_\ast} 
\quad \text{and} \quad 
\Theta(a)=\limsup_{n \to \infty}\frac{1}{n}\log\sum_{[x]: |\ell[x]-n| \le L}e^{-a \ell_\ast[x]}.
\]
Also note that for each $x \in \G$ there exists $p \in \G$ such that
\[
|\ell_\ast[x]-d_\ast(p, x p)| \le C_1 \quad \text{and} \quad |\ell[x]-d(p, x p)| \le C_2,
\]
by the proof of Lemma \ref{Lem:translation},
where $C_1$, $C_2$ are constants depending only on the hyperbolicity constants of $d_\ast$ and $d$.
This yields $\th(a) \ge \Theta(a)$ for each $a \in \R$.

For a large enough $R>0$ and all $z, w \in \G$, 
let
\[
O(z, w, R):=\{g \in \G\cup \partial \G \ : \ (z|g)_w \le R\}.
\]
Let us take a pair of hyperbolic elements $x, y$ such that $n \mapsto x^n$ and $n \mapsto y^n$ for $n \in \Z$ yield quasi-geodesics and their extremes points are distinct; there exists such a pair since $\G$ is non-elementary (cf.\ \cite[37.-Th\'eorème in Section 3, Chapitre 8]{GhysdelaHarpe}).
Taking large enough powers of $x$ if necessary,
we define for a large enough $R>0$,
\[
U:=O(o, x^{-1}, R), \quad V:=O(o, x, R), \quad \wt V:=O(x^{-3}, x^{-2}, R) \quad \text{and} \quad \wt U:=O(x^3, x^2, R),
\]
such that 
\[
U \cap V = \emptyset, \quad (\G \cup\partial\G)\setminus U \subset \wt V \quad \text{and} \quad (\G \cup\partial\G) \setminus V \subset \wt U.
\] 
Further, taking large enough powers of $y$ if necessary, we assume that $U':=yU$, $V':=yV$, $U$ and $V$ are disjoint.
For a fixed positive constant $L>0$ and every positive integer $n$,
let
\[
S_{n, L}:=\{z \in \G \ : \ ||z|-n| \le L\} \quad \text{and} \quad
S_{n, L}(U, V):=\{z \in S_{n, L} \ : \  U \cap z V=\emptyset\},
\]
and similarly, $S_{n, L}(V, U)$ and $S_{n, L}(U', V')$.

First we note that if $z \in S_{n, L}(U, V)$, then
\[
(x^3 z x^3)\wt V \subset V \quad \text{and} \quad (x^3 z x^3)^{-1} \wt U \subset U,
\]
since $(x^3 z x^3)\wt V = (x^3 z) V \subset x^3 (\G\cup\partial \G \setminus U) \subset x^3 \wt V= V$ and the latter is analogous.
Therefore
if we define
\[
U_{n, L}(x, x^{-1}, R):=\{z \in S_{n, L} \ : \ z^{-1} \in O(o, x^{-1}, R), \ z \in O(o, x, R)\},
\]
then
\begin{equation}\label{Eq:UV}
x^3 S_{n, L}(U, V) x^3 \subset U_{n, L+6|x|}(x, x^{-1}, R),
\end{equation}
since $o \in \wt U$ and $o \in \wt V$.
Moreover, we have that
\begin{equation}\label{Eq:VUU'V'}
x^{3}S_{n, L}(V, U)^{-1}x^{3} \subset U_{n, L+6|x|}(x, x^{-1}, R) \quad \text{and} \quad y^{-1}S_{n, L}(U', V')y\subset S_{n, L+2|y|}(U, V),
\end{equation}
where the former follows from $S_{n, L}(V, U)^{-1}=S_{n, L}(U, V)$ and \eqref{Eq:UV}, and the latter holds by the definition of $U'$ and $V'$.
Note that the map $z \mapsto x^3 z x^3$ yields an injection from $S_{n, L}(U, V)$ into $U_{n, L+6|x|}(x, x^{-1}, R)$.
Similarly the map $z \mapsto x^{3}z^{-1} x^{3}$ yields an injection from $S_{n, L}(V, U)$ into $U_{n, L+6|x|}(x, x^{-1}, R)$, and the map $z \mapsto y^{-1}z y$ yields an injection from $S_{n, L}(U', V')$ into $U_{n, L+6|x|+2|y|}(x, x^{-1}, R)$.

Second let us show that there exists a finite set $F_{U, V}$ in $\G$ independent of $n$ such that
\begin{equation}\label{Eq:triple}
S_{n, L}\setminus F_{U, V} \subset S_{n, L}(U, V) \cup S_{n, L}(V, U) \cup S_{n, L}(U', V').
\end{equation}
Indeed, if $z \in S_{n, L}$ and $z$ is not included in any one of $S_{n, L}(U, V)$, $S_{n, L}(V, U)$, or $S_{n, L}(U', V')$,
then one has 
\[
U \cap z V \neq \emptyset, \quad V \cap z U \neq \emptyset, \quad \text{and} \quad U'\cap z V' \neq \emptyset.
\]
Note that those elements $z$ for which $U \times V \times U'$ and $z (V \times U \times V')$ intersect are finite;
this follows since $U \times V \times U'$ and $V \times U \times V'$ are in $(\G \cup \partial \G)^{(3)}$,
where 
\[
(\G \cup \partial \G)^{(3)}:=\{(\x, \y, \z) \in (\G \cup\partial \G)^3 \ : \ \text{$\x$, $\y$ and $\z$ are distinct}\},
\]
and the diagonal action of $\G$ on $(\G \cup \partial \G)^{(3)}$ is properly discontinuous.
(Note that it is more standard to state that the diagonal action of $\G$ on the space of distinct ordered triples of points in the boundary $\partial \G$ is properly discontinuous; the same proof works for the case of $(\G \cup \partial \G)^{(3)}$ where we endow $\G \cup \partial \G$ with the compactified topology, cf.\ \cite[8.2.M]{GromovHyperbolic} and \cite[Lemma 1.2 and Proposition 1.12]{BowditchConvergence}.)
Hence \eqref{Eq:triple} holds for some finite set $F_{U, V}$ in $\G$ independent of $n$.

\begin{figure}
\centering
\includegraphics[width=150mm, bb=0 0 842 595]{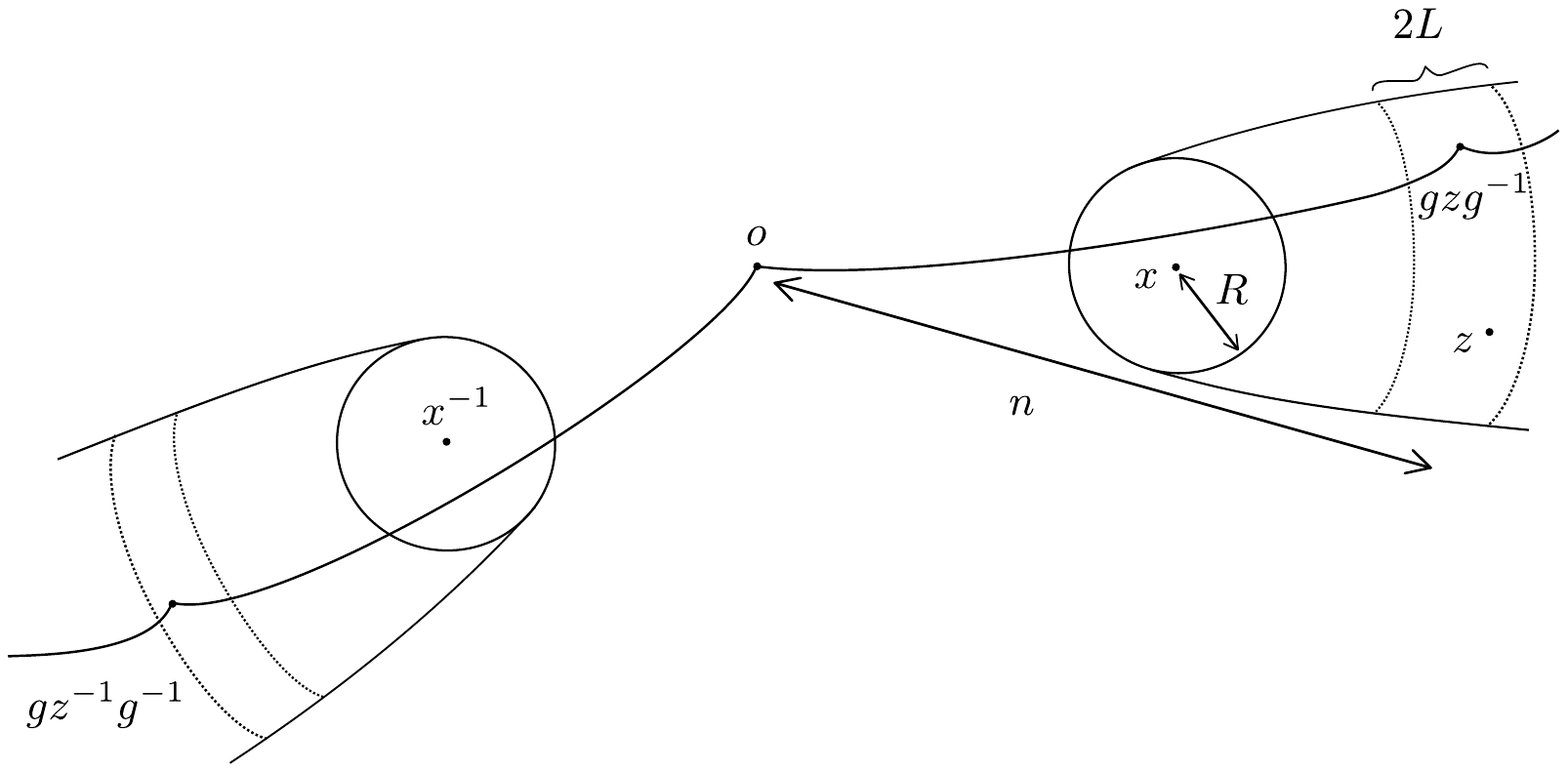}
\caption{}
\label{Fig:Knieper}
\end{figure}

Finally if $z \in U_{n, L}(x, x^{-1}, R)$,
then since $z \in V=O(o, x, R)$ and $z^{-1} \in U=O(o, x^{-1}, R)$,
we have that by the $\d$-hyperbolicity, 
\[
|(z|z^{-1})_o-(x|x^{-1})_o| \le C_{R,\d}.
\]
Lemma \ref{Lem:translation} implies that for all such $z$,
\[
\ell[z]=|z|-2(z|z^{-1})_o+O_{R, \d}(1)=|z|+O_{x, R, \d}(1).
\]
The analogous relations hold for $\ell_\ast[z]$.
Given $z \in U_{n, L}(x, x^{-1}, R)$,
let us count the number of elements in the set
\[
C_n(z; x, L, R):=\{g\langle z\rangle  \in \G/\langle z\rangle \ : \ g z g^{-1} \in U_{n, L}(x, x^{-1}, R)\},
\]
i.e., the number of $g \in \G$ modulo powers of $z$ such that $g z g^{-1} \in U_{n, L}(x, x^{-1}, R)$ (see Figure \ref{Fig:Knieper}).
Let $[o, z]$ denote a $C$-rough geodesic segment between $o$ and $z$ (with respect to $d$),
and define $\g(z):=\bigcup_{k \in \Z} z^k[o, z]$,
which is (the image of) a $(A, B)$-quasi geodesic line invariant under $z$ for some $A, B>0$.
Similarly, $\g(gzg^{-1})$ is a $(A, B)$-quasi geodesic line invariant under $g z g^{-1}$.
Note that $g \g(z)$ is also a $(A, B)$-quasi geodesic line invariant under $g z g^{-1}$,
and thus $g \g(z)$ and $\g(g z g^{-1})$ lie within a bounded Hausdorff distance.
This shows that if $g\langle z\rangle \in C_n(z; x, L, R)$,
then $g\g(z)$ passes through near $o$ within a bounded distance $C_{A, B, \d}$ of $o$.
Crucially, $A$ and $B$ depend only on the hyperbolicity constant and the $x, R$ used in $C_n(z; x, L, R)$.
Since for all such $g$ the inverse $g^{-1}$ lies in a neighborhood of $[o, z]$ up to translation by $z$,
and $\g(z)$ is $z$-invariant, counting such $g^{-1}$ modulo the $\langle z\rangle$-action on $\gamma(z)$ yields $\# C_n(z; x, L, R) \le C_{A, B, \d}' n$.
Now we obtain
\[
\sum_{z \in U_{n, L}(x, x^{-1}, R)}e^{-a |z|_\ast} \le C'n \sum_{[z]: |\ell[z]-n|\le L'}e^{-a\ell_\ast[z]},
\]
where $C'$ and $L'$ are constants depending only on $a, x, R, L$ and the hyperbolicity constants of $d$ and $d_\ast$.
Therefore,
noting that
\[
||x^{3}zx^{3}|_\ast - |z|_\ast| \le 6|x|_\ast, \quad ||x^{3} z^{-1} x^{3}|_\ast - |z|_\ast| \le 6|x|_\ast \quad \text{and} \quad ||y^{-1} z y|_\ast-|z|_\ast| \le 2|y|_\ast,
\]
we obtain by \eqref{Eq:UV} and \eqref{Eq:VUU'V'} together with \eqref{Eq:triple}, for all large enough $n$,
\[
\sum_{z \in S_{n, L}\setminus F_{U, V}}e^{-a |z|_\ast} \le Cn \sum_{[z]: |\ell[z]-n| \le L'} e^{-a \ell_\ast[z]},
\]
where $L$ and $L'$ are large enough fixed constants depending only on $x$ and $y$.
Since $F_{U, V}$ is a finite set of elements independent of $n$,
we have that $\th(a) \le \Theta(a)$, as required.
\qed

Let 
\[
\a_{\min}:=-\lim_{a \to \infty}\frac{\th(a)}{a} \quad \text{and} \quad \a_{\max}:=-\lim_{a \to -\infty}\frac{\th(a)}{a},
\]
where $\a_{\min}$ and $\a_{\max}$ are positive and finite since $d$ and $d_\ast$ are quasi-isometric.
Recall that
\[
\Dil_-:=\inf_{[x] \in \conj_{>0}}\frac{\ell_\ast[x]}{\ell[x]} \quad \text{and} \quad \Dil_+:=\sup_{[x] \in \conj_{>0}}\frac{\ell_\ast[x]}{\ell[x]}
\]
where $\conj_{>0}$ is the set of elements $[x] \in \conj$ such that $\ell[x]$ (and hence $\ell_\ast[x]$) is non-zero.
\begin{corollary}\label{Cor:Dil}
For all $d, d_\ast \in \Dc_\G$, we have
\[
\Dil_-=\a_{\min} \quad \text{and} \quad \Dil_+=\a_{\max}.
\]
\end{corollary}

\proof
Fix a large enough $L>0$.
For all $a>0$ and all integers $n \ge 0$, we have that
\[
\sum_{[x]:\ |\ell[x]-n| \le L}e^{-a\ell_\ast[x]} \le \sum_{[x]:\ |\ell[x]-n|\le L}e^{-a \Dil_-\, \ell[x]},
\]
which together with Proposition \ref{Prop:conjugacy} implies that $\th(a) \le -a \Dil_-+O(1)$ and thus $\a_{\min} \ge \Dil_-$.
Further for all $\e>0$, there exist infinitely many $[x] \in \conj_{>0}$ such that
\begin{equation*}\label{Eq:dil-}
\frac{\ell_\ast[x]}{\ell[x]} \le \Dil_-+\e.
\end{equation*}
Hence for all $a>0$ and infinitely many integers $n \ge 0$, 
\[
\sum_{[x]:\ |\ell[x]-n|\le L}e^{-a\ell_\ast[x]} \ge c e^{-a(\Dil_-+\e)n},
\]
where $c$ is a positive constant independent of $n$,
and thus by Proposition \ref{Prop:conjugacy},
\[
\th(a) \ge -a(\Dil_-+\e).
\]
Therefore $\a_{\min} \le \Dil_-$.
We obtain $\a_{\min}=\Dil_-$.
Showing that $\a_{\max}=\Dil_+$ is analogous.
\qed

\begin{lemma} \label{Lem:asymprigidity}
Let $\Gamma$ be a non-elementary hyperbolic group and $d, d_\ast \in \Dc_\Gamma$. 
Then the following are equivalent:
\begin{enumerate}
\item[\upshape(1)] The Manhattan curve $\Cc_M$ for the pair $d$ and $d_\ast$ is a straight line, and
\item[\upshape(2)] there exists a constant $c>0$ such that $\ell[x] = c \ell_\ast[x]$ for all $[x] \in \conj$.
\end{enumerate}
\end{lemma}

\begin{proof}
If the Manhattan curve $\Cc_M$ is a straight line, then ${\rm Dil}_- = {\rm Dil}_+$. Furthermore $-{\rm Dil}_-$ and $ -{\rm Dil}_+$ are equal to the gradient of the Manhattan curve. Since $\Cc_M$ goes through the points $(0,v)$ and $(v_\ast,0)$ this gradient is $-v/v_\ast$, where $v, v_\ast>0$ if $\G$ is non-elementary. Hence for all $[x] \in \conj_{>0}$
\[
\frac{v}{v_\ast} = {\rm Dil}_- = \inf_{[g] \in \conj_{>0}} \frac{\ell_\ast[g]}{\ell[g]} \le \frac{\ell_\ast[x]}{\ell[x]} \le \sup_{[g] \in \conj_{>0}} \frac{\ell_\ast[g]}{\ell[g]} = {\rm Dil}_+ = \frac{v}{v_\ast},
\]
implying that $v \ell[x]=v_\ast \ell_\ast[x]$ for all $[x] \in \conj$.
The converse follows from the definition of the Manhattan curve.
\end{proof}

\subsection{Proof of the $C^1$-regularity}\label{Sec:proofofC1}

Fix a pair of metrics $d, d_\ast \in \Dc_\G$.
For each $\x \in \partial \G$
and quasi-geodesic $\g_\x:[0, \infty) \to (\G, d)$ with $\g_\x(t) \to \x$ as $t \to \infty$,
we define
\[
\t_{\inf}(\x):=\liminf_{t \to \infty}\frac{d_\ast(\g_\x(0), \g_\x(t))}{d(\g_\x(0), \g_\x(t))}
\quad \text{and} \quad
\t_{\sup}(\x):=\limsup_{t \to \infty}\frac{d_\ast(\g_\x(0), \g_\x(t))}{d(\g_\x(0), \g_\x(t))}.
\]
The Morse Lemma (applied to $(\G, d)$) implies that $\t_{\inf}(\x)$ and $\t_{\sup}(\x)$ are independent of the choice of quasi-geodesics converging to $\x$,
or of their starting points, respectively.
If $\t_{\inf}(\x)=\t_{\sup}(\x)$, then we denote the common value by $\t(\x)$ and call it the {\it local intersection number} at $\x$ for the pair $(d, d_\ast)$.

\begin{lemma}\label{Lem:tau}
For each $(a, b) \in \Cc_M$, we have that $\t_{\inf}(\x)=\t_{\sup}(\x)$ for $\m_{a, b}$-almost every $\x \in \partial \G$ and further there exists a constant $\t_{a, b}$ such that
\[
\t(\x)=\t_{a, b} \quad \text{for $\m_{a, b}$-almost every $\x \in \partial \G$}.
\]
Moreover, $\t_{a, b}$ is continuous in $(a, b) \in \Cc_M$.
\end{lemma}

\proof
Let $\wh d$ be a strongly hyperbolic metric in $\Dc_\G$.
We consider the map $\ev: \partial^2 \G \times \R \to (\G, \wh d\,)$ and the flow space $\Fc_\k$ where $\k$ is the cocycle associated with $\wh d$ defined in Section \ref{Sec:top-flow}.
We write $w_t:=\F_t(w)$ for $w \in \Fc_\k$.
Taking a measurable section $\i: \Fc_\k \to D$ for a Borel fundamental domain $D$ in $\partial^2 \G \times \R$ for the $(\G, \k)$-action,
we define $\wt w:=\i(w)$ and set
$\wt w_t:=\wt \F_t(\wt w)$.
Let
\[
c_\ast(w_s, w_t):=d_\ast\(\ev(\wt w_s), \ev(\wt w_t)\).
\]
Then $c_\ast(w_s, w_t)$
is subadditive,
i.e., for all $t, s \in [0, \infty)$,
\begin{equation}\label{Eq:subadditive}
c_\ast(w_0, w_{s+t}) \le c_\ast(w_0, w_s)+c_\ast(w_s, w_{s+t}),
\end{equation}
and superadditive up to an additive constant, i.e.,
there exists a $C \ge 0$ such that for all $t, s \in [0, \infty)$,
\begin{equation}\label{Eq:superadditive}
c_\ast(w_0, w_{s+t}) \ge c_\ast(w_0, w_s)+c_\ast(w_s, w_{s+t})-C,
\end{equation}
by the Morse lemma on $(\G, d_\ast)$.
Since $m_{a, b}$ is ergodic with respect to the flow on $\Fc_\k$,
the Kingman subadditive ergodic theorem implies that there exists a constant $\chi_\ast(a, b)$ such that
\[
\lim_{t\to \infty}\frac{1}{t} c_\ast(w_0, w_t) = \chi_\ast(a, b) \quad \text{for $m_{a, b}$-almost every $w$ in $\Fc_\k$},
\]
and 
\[
\lim_{t \to \infty}\frac{1}{t}\int_{\Fc_\k}c_\ast(w_0, w_t)\,dm_{a, b}=\chi_\ast(a, b).
\]
Let us show that $\chi_\ast(a, b)$ is continuous in $(a, b) \in \Cc_M$.
For each $(a_0, b_0) \in \Cc_M$,
if $(a, b) \to (a_0, b_0)$,
then $m_{a, b}$ weakly converges to $m_0:=m_{a_0, b_0}$
by Corollary \ref{Cor:limit}.
Then the subadditivity \eqref{Eq:subadditive} yields for all $(a, b) \in \Cc_M$ and all $t > 0$,
\[
\frac{1}{t}\int_{\Fc_\k}c_\ast(w_0, w_t)\,dm_{a, b} \ge \inf_{t>0}\frac{1}{t}\int_{\Fc_\k}c_\ast(w_0, w_t)\,dm_{a, b}=\chi_\ast(a, b),
\]
we have that for each $t \ge 0$,
\[
\frac{1}{t}\int_{\Fc_\k}c_\ast(w_0, w_t)\,dm_0 \ge \limsup_{a \to a_0}\chi_\ast(a, b),
\]
and similarly, \eqref{Eq:superadditive} implies that for each $t > 0$
\[
\frac{1}{t}\int_{\Fc_\k}\(c_\ast(w_0, w_t)-C\)\,dm_0 \le \liminf_{a \to a_0}\chi_\ast(a, b).
\]
Therefore letting $t\to \infty$, we obtain $\lim_{a \to a_0}\chi_\ast(a, b)=\chi_\ast(a_0, b_0)$, i.e., $\chi_\ast(a, b)$ is continuous in $(a, b) \in \Cc_M$.

We apply the same discussion to $d$:
letting 
\[
c(w_s, w_t):=d\(\ev(\wt w_s), \ev(\wt w_t)\),
\] 
we have that there exists a constant $\chi(a, b)$ such that
\[
\lim_{t\to \infty}\frac{1}{t} c(w_0, w_t) = \chi(a, b) \quad \text{for $m_{a, b}$-almost every $w$ in $\Fc_\k$},
\]
and $\chi(a, b)$ is continuous in $(a, b) \in \Cc_M$.

Therefore for $m_{a, b}$-almost every $w=[\x_-, \x_+, t_0] \in \Fc_\k$,
\[
\lim_{t \to \infty}\frac{c_\ast(w_0, w_t)}{c(w_0, w_t)}=\frac{\chi_\ast(a, b)}{\chi(a, b)}.
\]
Recall that $\L_{a, b}\otimes dt=\sum_{x \in \G}x_\ast(\i_\ast m_{a, b})$, $\L_{a, b}$ is equivalent to $\m_{a, b}\otimes \m_{a, b}$
and that $d_\ast$ and $d$ are left-invariant.
Hence if we define $\t_{a, b}:=\chi_\ast(a, b)/\chi(a, b)$, 
then $\t_{\inf}(\x)=\t_{\sup}(\x)$ for $\m_{a, b}$-almost every $\x \in \partial \G$ and 
\[
\t(\x)=\t_{a, b} \quad \text{for $\m_{a, b}$-almost every $\x \in \partial \G$}.
\]
Since $\chi(a, b)$ and $\chi_\ast(a, b)$ are positive and continuous, $\t_{a, b}$ is continuous in $(a, b) \in \Cc_M$, as required.
\qed

For every real value $r \in \R$,
let
\[
E_r:=\big\{\x \in \partial \G \ : \ \t_{\inf}(\x)=\t_{\sup}(\x)=r\big\}.
\]
The set $E_r$ is possibly empty for some $r$.
Note that a point $\x$ is in $E_r$ if and only if for some (equivalently, every) quasi-geodesic $\g_\x$ converging to $\x$,
we have
\[
\lim_{t \to \infty}\frac{d_\ast(\g_\x(0), \g_\x(t))}{d(\g_\x(0), \g_\x(t))}=r.
\]
Recall that the Manhattan curve $\Cc_M$ is the graph of the function $\th$, i.e., $(a, b) \in \Cc_M$ if and only if $b=\th(a)$,
and since $\th$ is convex, it is differentiable except for at most countably many points.

\begin{lemma}\label{Lem:full}
Fix a pair $d, d_\ast \in \Dc_\G$.
For each $(a, b) \in \Cc_M$, if $\th$ is differentiable at $a$ and $r=-\th'(a)$, then 
$\m_{a, b}(E_r)=1$.
\end{lemma}
\proof
Fix a large enough constant $C \ge 0$.
Let us endow the space of $C$-rough geodesic rays from $o$ in $(\G, d)$ with the point-wise convergence topology.
For each $\x \in \partial \G$, we associate a $C$-rough geodesic $\g_\x$ from $o$ to $\x$,
and define this correspondence in a Borel measurable way as in Section \ref{Sec:top-flow} where we have done the same but for rough geodesics.
For each non-negative integer $n$ and $\x \in \partial \G$, 
we abbreviate notation by writing $\x_n:=\g_\x(n)$ and $|\x_n|:=d(o, \x_n)$ (resp.\ $|\x_n|_\ast:=d_\ast(o, \x_n)$).
We use $O(x, R)$ to denote the shadows associated to the metric $d$, for a large enough thickness parameter $R$.

For $(a, b) \in \Cc_M$, let us suppose that $r=-\th'(a)$.
For a Patterson-Sullivan measure $\m_\ast$ for $d_\ast$,
we show that
\begin{equation}\label{Eq:lemfullinf}
\liminf_{n \to \infty}-\frac{1}{v_\ast|\x_n|}\log \m_\ast\(O(\x_n, R)\) \ge r \quad \text{for $\m_{a, b}$-almost every $\x \in \partial \G$}.
\end{equation}
For every $\e>0$, the Markov inequality shows the following: For every $s>0$, integrating
\[
 \1_{\big\{\x \in \partial \G \ : \ \m_\ast\(O(\x_n, R)\) \ge e^{-(r-\e)v_\ast|\x_n|}\big\}} \le \m_\ast\(O(\x_n, R)\)^s \cdot e^{s(r-\e)v_\ast|\x_n|},
\]
over $\x \in \partial \G$ with respect to $\m_{a, b}$ yields
\begin{align*}
&\m_{a, b}\(\big\{\x \in \partial \G \ : \ \m_\ast\(O(\x_n, R)\) \ge e^{-(r-\e)v_\ast|\x_n|}\big\}\)\\
&\qquad \qquad \qquad \qquad \qquad \qquad \qquad \le \int_{\partial \G}\m_\ast\(O(\x_n, R)\)^s e^{s(r-\e)v_\ast|\x_n|}\,d\m_{a, b}(\x).
\end{align*}
Since
\[
\m_\ast\(O(\x_n, R)\) \asymp_R \exp\(-v_\ast|\x_n|_\ast\) \quad \text{and} \quad \m_{a, b}\(O(\x_n, R)\) \asymp_R \exp\(-a|\x_n|_\ast-b|\x_n|\),
\]
the integral in the right hand side is at most
\begin{equation}\label{Eq:lemfull}
C_R\sum_{x \in S(n, R)}\exp\(-sv_\ast|x|_\ast+s(r-\e)v_\ast|x|-a|x|_\ast-b|x|\),
\end{equation}
where we recall that $S(n, R)=\{x \in \G \ : \ |d(o, x)-n| \le R\}$, 
up to a multiplicative constant depending only on $R$ and the $\d$-hyperbolicity constant of $d$.
Moreover, \eqref{Eq:lemfull} is at most
\begin{align*}
&C'_R \exp\(s(r-\e)v_\ast n-b n\) \sum_{x \in S(n, R)}\exp\(-sv_\ast|x|_\ast-a|x|_\ast\)\\
&\qquad \qquad \qquad \qquad \le C''_R \exp\(s(r-\e)v_\ast n-b n+\th(sv_\ast+a)n\),
\end{align*}
where we have used Lemma \ref{Lem:Coornaert}:
\[
\sum_{x \in S(n, R)}\exp\(-sv_\ast|x|_\ast-a|x|_\ast\) \asymp \exp\(\th(s v_\ast+a) n\).
\]
Since $b=\th(a)$ and $r=-\th'(a)$,
\[
\th(sv_\ast+a)-\th(a)=-r s v_\ast+o(s) \quad \text{as $s \searrow 0$},
\]
we obtain
\begin{align*}
&\m_{a, b}\(\big\{\x \in \partial \G \ : \ \m_\ast\(O(\x_n, R)\) \ge e^{-(r-\e)v_\ast|\x_n|}\big\}\)\\
&\qquad \qquad \qquad \qquad \le C_R'' \exp\(s(r-\e)v_\ast n+(\th(sv_\ast+a)-\th(a))n\) \\
&\qquad \qquad \qquad \qquad \le C_R'' \exp\(-s\e v_\ast n+o(s)n\)
\le C_R'' \exp\(-c(\e, s)n\),
\end{align*}
for some constant $c(\e, s)>0$ for all $n \ge 0$.
Hence the Borel-Cantelli lemma shows that 
\[
\liminf_{n \to \infty}-\frac{1}{v_\ast|\x_n|}\log \m_\ast\(O(\x_n, R)\) \ge r-\e \quad \text{for $\m_{a, b}$-almost every $\x \in \partial \G$},
\]
and since this holds for every $\e>0$, we obtain \eqref{Eq:lemfullinf}.

Similarly, it holds that
\begin{equation}\label{Eq:lemfullsup}
\limsup_{n \to \infty}-\frac{1}{v_\ast|\x_n|}\log \m_\ast\(O(\x_n, R)\) \le r \quad \text{for $\m_{a, b}$-almost every $\x \in \partial \G$}.
\end{equation}
Indeed, 
for every $\e>0$ and every $s>0$,
we have that
\begin{align*}
&\m_{a, b}\(\big\{\x \in \partial \G \ : \ \m_\ast\(O(\x_n, R)\) \le e^{-(r+\e)v_\ast|\x_n|}\big\}\)\\
&\qquad \qquad \qquad \qquad \qquad \qquad \qquad \le \int_{\partial \G}\m_\ast\(O(\x_n, R)\)^{-s} e^{-s(r+\e)v_\ast|\x_n|}\,d\m_{a, b}(\x),
\end{align*}
and the rest follows as in the same way above; we omit the details.

Combining \eqref{Eq:lemfullinf} and \eqref{Eq:lemfullsup},
we obtain
\[
\lim_{n \to \infty}\frac{\ |\x_n|_\ast}{|\x_n|}=\lim_{n \to \infty}-\frac{1}{v_\ast|\x_n|}\log \m_\ast\(O(\x_n, R)\)=r,
\]
for $\m_{a, b}$-almost every $\x \in \partial \G$.
Therefore we have that $\t_{\inf}(\x)=\t_{\sup}(\x)$ for $\m_{a, b}$-almost every $\x \in \partial \G$,
and
$\m_{a, b}\(E_r\)=1$
if $b=\th(a)$ and $r=-\th'(a)$, as required.
\qed

\begin{theorem}\label{Thm:C1}
For every pair $d, d_\ast \in \Dc_\G$, the Manhattan curve $\Cc_M$ is $C^1$, i.e.,
the function $\th$ is continuously differentiable on $\R$.
Moreover, $\th'(a)=-\t_{a, b}$ for all $(a, b) \in \Cc_M$.
\end{theorem}

\proof
Recall that since $\th$ is convex, $\th$ is differentiable except for at most countably many points.
For each $(a, b) \in \Cc_M$,
if $r=-\th'(a)$, then Lemma \ref{Lem:full} implies that
$\t(\x)=r$ for $\m_{a, b}$-almost every $\x \in \partial \G$;
on the other hand, Lemma \ref{Lem:tau} implies that
$\t(\x)=\t_{a, b}$ for $\m_{a, b}$-almost every $\x \in \partial \G$.
Therefore if $b=\th(a)$ and $r=-\th'(a)$, then
\[
\th'(a)=-\t_{a, b}.
\]
Since this holds for Lebesgue almost every $a$ in $\R$ and $\t_{a, b}$ is continuous in $(a, b) \in \Cc_M$ by Lemma \ref{Lem:tau},
$\th$ is differentiable everywhere and the derivative coincides with $-\t_{a, b}$ which is continuous.
\qed

The above proof yields the multifractal spectrum of every Patterson-Sullivan measure $\m_\ast$ with respect to $\rho(\x,\y)=\exp(-(\x|\y)_o)$ in $\partial \G$ and the profile is the Legendre transform of the Manhattan curve.

\begin{theorem}[The multifractal spectrum]\label{Thm:MFS}
For every pair $d, d_\ast \in \Dc_\G$,
let $\m_\ast$ be an arbitrary Patterson-Sullivan measure relative to $d_\ast$
and $\rho(\x, \y)=\exp(-(\x|\y)_o)$ be the quasi-metric relative to $d$ on $\partial \G$.
For $\a \in \R$ we define
\[
E(\a):=\left\{\x \in \partial \G \ : \ \lim_{r \to 0}\frac{\log \m_\ast\(B(\x, r)\)}{\log r}=\a\right\},
\]
where $B(\x, r)=\{\y \in \partial \G \ : \ \rho(\x, \y)<r\}$.
Then we have 
\begin{equation}\label{Eq:MFS}
\dim_\H \(E(v_\ast \a), \rho\)=\inf_{a \in \R}\big\{a \a+\theta(a)\big\} \quad \text{for $\a \in (\a_{\min}, \a_{\max})$},
\end{equation}
where 
\[
\a_{\min}=-\lim_{a \to \infty}\frac{\theta(a)}{a} \quad \text{and} \quad \a_{\max}=-\lim_{a \to -\infty}\frac{\theta(a)}{a}.
\]
\end{theorem}

\proof
Note that the function $\th$ is $C^1$
and $\th'(a)=-\t_{a, b}$ for all $(a, b) \in \Cc_M$ by Theorem \ref{Thm:C1}.
Hence Lemmas \ref{Lem:full} and \ref{Lem:tau} together with Lemma \ref{Lem:shadow-ball} imply that
for all $(a, b) \in \Cc_M$,
\[
\lim_{r \to 0}\frac{\log \m_\ast\(B(\x, r)\)}{\log r}=v_\ast \t_{a, b} \quad \text{for $\m_{a, b}$-almost every $\x \in \partial \G$},
\]
and
\[
\lim_{r \to 0}\frac{\log \m_{a, b}\(B(\x, r)\)}{\log r}=a \t_{a, b}+b \quad \text{for $\m_{a, b}$-almost every $\x \in \partial \G$},
\]
where we have used $\m_{a, b}(O(x, R))\asymp_R \exp(-ad_\ast(o, x)-bd(o, x))$ for $x \in \G$.
The Frostman-type lemma (Lemma \ref{Lem:Frostman}) shows that
\[
\dim_\H(E(v_\ast\t_{a, b}), \rho)=a \t_{a, b}+b.
\]
Since $\th$ is continuously differentiable and convex, for each $\a \in (\a_{\min}, \a_{\max})$ there exists $a \in \R$ such that $\a=-\th'(a)$, and
\[
\dim_\H\(E(v_\ast \a), \rho\)=-a \th'(a)+\th(a),
\]
where the right hand side is the Legendre transform of $\theta$.
Therefore we conclude the claim.
\qed

\begin{remark}\label{Rem:MFS}
If we have
\[
\th(a)=-a \a_{\min}+O(1) \quad \text{as $a \to \infty$} 
\quad \text{and} \quad 
\th(a)=-a \a_{\max}+O(1) \quad \text{as $a \to -\infty$},
\]
then the formula \eqref{Eq:MFS} is valid for all $\a \in [\a_{\min}, \a_{\max}]$ including two extreme points $\a_{\min}$ and $\a_{\max}$.
This is the case for example when both $d_\ast$ and $d$ are word metrics, see Proposition \ref{Prop:extreme} in the following Section \ref{Sec:C2}.
\end{remark}

\subsection{Rough similarity rigidity}
In this section we prove Theorem \ref{Thm:thm2} and the rigidity statement in Theorem \ref{Thm:thm1}. We begin with the following.
\begin{theorem}\label{Thm:rigidity}
For every pair $d, d_\ast \in \Dc_\G$, the following are equivalent:
\begin{itemize}
\item[(i)] The Manhattan curve $\Cc_M$ is a straight line between $(0, v)$ and $(v_\ast, 0)$ where $v$ and $v_\ast$ are the exponential volume growth rates of $(\G, d)$ and $(\G, d_\ast)$ respectively; and,
\item[(ii)] $d_\ast$ and $d$ are roughly similar.

\end{itemize}
\end{theorem}

We use the following lemma in the proof.
Recall that $d \in \Dc_\G$ is a roughly geodesic metric and there exists a constant $C \ge 0$ such that for all $\x \in \partial \G$,
one may take a $C$-rough geodesic ray $\g_\x$ from $o$ converging to $\x$ on $(\G, d)$.

\begin{lemma}\label{Lem:diff}
Let $\n$ be a finite Borel regular measure on $\partial \G$ and $\m$ be a doubling measure on $\partial \G$ relative to a quasi-metric $\rho$ for $d \in \Dc_\G$.
If we decompose $\n=\n_\ac+\n_\sing$ where $\n_\ac$ is the absolutely continuous part of $\n$ and $\n_\sing$ is the singular part of $\n$ relative to $\m$, respectively,
then for a large enough $R>0$,
\[
\limsup_{n \to \infty}\frac{\n_\ac\(O(\x_n, R)\)}{\m\(O(\x_n, R)\)} <\infty 
\quad \text{and} \quad
\limsup_{n \to \infty}\frac{\n_\sing\(O(\x_n, R)\)}{\m\(O(\x_n, R)\)}=0
\]
for $\m$-almost every $\x \in \partial \G$
where $\x_n:=\g_\x(n)$ for $n \ge 1$. 
\end{lemma}

The proof of Lemma \ref{Lem:diff} follows from the classical Lebesgue differentiation theorem and the weak maximal inequality---we include a proof for the sake of completeness in Appendix \ref{Sec:Lebesgue}.

\proof[Proof of Theorem \ref{Thm:rigidity}]
If (ii) holds, then the Manhattan curve $\Cc_M$ is actually a straight line on $\R$ since $\t_{a, b}=\t$ for a constant $\t>0$ for all $(a, b) \in \Cc_M$
by Lemma \ref{Lem:tau} and $\th'(a)=-\t$ for all $a \in \R$ by Theorem \ref{Thm:C1} (or by Lemma \ref{Lem:asymprigidity}).

Suppose that (i) holds.
Then $(a, b):=(v_\ast/2, v/2) \in \Cc_M$.
By Corollary \ref{Cor:ab} (1), we have that for all $x \in \G$,
\[
\m_{a, b}\(O(x, R)\) \asymp \exp\(-\frac{v_\ast}{2}|x|_\ast-\frac{v}{2}|x|\),
\]
\begin{equation}\label{Eq:m-shadow}
\m_\ast(O(x, R))\asymp \exp\(-v_\ast |x|_\ast\) \quad \text{and} \quad \m(O(x, R)) \asymp \exp\(-v|x|\).
\end{equation}
This implies that
\begin{equation}\label{Eq:compare}
\frac{\m_\ast(O(x, R))}{\m_{a, b}(O(x, R))}\cdot \frac{\m(O(x, R))}{\m_{a, b}(O(x, R))} \asymp 1 \quad \text{for all $x \in \G$}.
\end{equation}
Fix a large enough $R>0$.
Letting $\x_n:=\g_\x(n)$ for integers $n \ge 0$,
we have that
\[
\limsup_{n \to \infty}\frac{\m_\ast(O(\x_n, R))}{\m_{a, b}(O(\x_n, R))}<\infty
\quad \text{and} \quad
\limsup_{n \to \infty}\frac{\m(O(\x_n, R))}{\m_{a, b}(O(\x_n, R))}<\infty,
\]
for $\m_{a, b}$-almost every $\x \in \partial \G$ by Lemma \ref{Lem:diff}.
Here we are using that $\m_\ast$ and $\m$ are finite Borel regular measures and that $\m_{a, b}$ is doubling relative to a quasi-metric $\rho$ in $\partial \G$.
Hence if either $\m_\ast$ and $\m_{a, b}$, or $\m$ and $\m_{a, b}$ are mutually singular,
then Lemma \ref{Lem:diff} together with \eqref{Eq:compare} leads to a contradiction. 
Therefore both $\m_\ast$ and $\m$ have non-zero absolutely continuous parts relative to $\m_{a, b}$, and thus
for the corresponding $\G$-invariant Radon measures $\L_\ast$, $\L$ and $\L_{a, b}$ for $\m_\ast$, $\m$ and $\m_{a, b}$, respectively, 
both $\L_\ast$ and $\L$ have non-zero absolutely continuous parts relative to $\L_{a, b}$.
By Lemma \ref{Lem:pure-type}, there exist positive constants $c, c'>0$ such that
$\L_\ast=c\L_{a, b}$ and $\L=c'\L_{a, b}$.
In particular, $\L_\ast=(c/c')\L$ and this implies that $\m_\ast$ and $\m$ are mutually absolutely continuous.
Letting $\f:=d\m_\ast/d\m$, we shall show that $\f$ is uniformly bounded away from $0$ and from above.
We have that
\[
\f(\x)\f(\y)e^{2 v_\ast(\x|\y)_{\ast o}} \asymp e^{2 v(\x|\y)_o} \quad \text{for $(\x, \y) \in \partial^2 \G$}.
\]
If $\f$ is unbounded on $B(\x, \e)$ for all $\e>0$, then for a fixed $\y \neq \x$ such that $\f(\y)>0$,
it is possible that  for $\x' \in B(\x, \e)$, the value $\f(\x')\f(\y)$ is arbitrarily large,
however, $(\x'|\y)_{\ast o}$ and $(\x'|\y)_o$ are uniformly bounded;
this is a contradiction.
This shows that $\m_\ast \asymp \m$ and thus by the above estimates \eqref{Eq:m-shadow},
there exists a constant $C \ge 0$ such that
\[
|v_\ast |x|_\ast-v |x||  \le C \quad \text{for all 
$x \in \G$},
\]
i.e., $d_\ast$ and $d$ are roughly similar; we conclude the claim.
\qed

We can now conclude the proof of our first main result.

\begin{proof}[Proof of Theorem \ref{Thm:thm1}]
Combining Theorems \ref{Thm:C1} and \ref{Thm:rigidity} concludes the proof of Theorem \ref{Thm:thm1}. 
\end{proof}
\bigskip

Let us now move on to the proof of Theorem \ref{Thm:thm2}. We will break the proof into two parts.
For every pair $d, d_\ast \in \Dc_\G$ define
\[
\t(d_\ast/d):=\limsup_{r \to \infty}\frac{1}{\#B(o, r)}\sum_{x \in B(o, r)}\frac{d_\ast(o, x)}{r},
\]
where $B(o, r):=\{x \in \G \ : \ d(o, x) \le r\}$ for a real $r>0$. 
We begin by proving the following.

\begin{theorem}\label{Thm:distortion}
For every pair $d, d_\ast \in \Dc_\G$,
the following limit exists:
\[
\t(d_\ast/d)=\lim_{r \to \infty}\frac{1}{\#B(o, r)}\sum_{x \in B(o, r)}\frac{d_\ast(o, x)}{r},
\]
and $\t_{0, v}=\t(d_\ast/d)$.
Moreover, we have that
\[
\t(d_\ast/d) \ge \frac{v}{v_\ast},
\]
where $v$ and $v_\ast$ are the exponential volume growth rates of $(\G, d)$ and $(\G, d_\ast)$ respectively.
\end{theorem}

\proof
Fix a pair $d$, $d_\ast \in \Dc_\G$ and consider the point $(0, v)$ on the associated Manhattan curve $\Cc_M$.
By Lemma \ref{Lem:tau}, there exists a constant $\t_{0, v}$ such that 
\[
\t(\x)=\t_{0, v} \quad \text{for $\m_{0, v}$-almost every $\x \in \partial \G$},
\]
where we note that $\m_{0, v}$ is a Patterson-Sullivan measure for the metric $d$.
In particular, for $\m_{0, v}$-almost every $\x \in \partial \G$,
$d_\ast(o, \g_\x(n))/n \to \t_{0, v}$ as $n \to \infty$  where $\g_\x$ is an arbitrary rough geodesic ray (with respect to $d$) starting from $o$. 
Let us define
\[
A_{n, \e}:=\left\{x \in \G \ : \ \frac{|d_\ast(o, x)-n \t_{0, v}|}{n}>\e\right\} \quad \text{for $n \ge 0$ and $\e>0$}.
\]
Consider $S(n, R):=\{x \in \G \ : \ |d(o, x)-n| \le R\}$ and fix a sufficiently large $R_0>0$.
Since the shadows $O(x, R_0)$ for $x \in S(n, R)$ cover the boundary $\partial \G$ with a bounded multiplicity,
we have 
\[
\frac{\# \(A_{n, \e}\cap S(n, R)\)}{\# S(n, R)} \le C \sum_{x \in A_{n, \e}\cap S(n, R)}\m_{0, v}(O(x, R_0))\le C'\m_{0, v}\big(\bigcup_{x \in A_{n, \e}\cap S(n, R)}O(x, R_0)\big).
\]
Note that the last term tends to $0$ as $n \to \infty$ since
if $\x$ belongs to $O(x, R_0)$ for some $x \in A_{n, \e}\cap S(n, R)$,
then $|d_\ast(o, \g_\x(n))-n \t_{0, v}| \ge \e n-R_0 L_\ast$, 
where 
\[
L_\ast:=\sup\{d_\ast(o, x) \ : \ d(o, x)\le R_0\}.
\]
This shows that if $x$ is sampled uniformly at random from $S(n, R)$,
then for all $\e>0$ and for all large enough $n$, we have
$|d_\ast(o, x) -n \t_{0, v}| \le \e n$ with probability at least $1-\e$,
implying that
\[
\t_{0, v}=\lim_{n \to \infty}\frac{1}{\#S(n, R)}\sum_{x \in S(n, R)}\frac{d_\ast(o, x)}{n},
\]
for all large enough $R$.
For all real $r>0$, let us take $n:=\lfloor r\rfloor$ the largest integer at most $r$.
Note that if $x_n$ is sampled uniformly at random from the ball $B(o, r)$,
then we have $x_n \in A_{n, \e}$ with probability at most $O(e^{-v R})$ for all large enough $n$
since the probability that $x$ is not in $S(n, R)$ is at most $O(e^{-vR})$ (following from Lemma \ref{Lem:Coornaert}: $\# S(n, R) \asymp_{R} e^{v n}$).
Therefore first letting $r \to \infty$ and then $R \to \infty$, we obtain
\[
\t_{0, v}=\lim_{r \to \infty}\frac{1}{\# B(o, r)}\sum_{x \in B(o, r)}\frac{d_\ast(o, x)}{r},
\]
and thus $\t(d_\ast/d)=\t_{0, v}$.
Furthermore this reasoning shows that for each fixed, sufficiently large $R$,
\[
\#B_\ast(o, (\t_{0, v}+\e) r) \ge (1-O(e^{-v R}))\cdot \# B(o, r) \quad \text{as $r \to \infty$},
\]
where $B_\ast(o, R)$ stands for the ball of radius $R$ centered at $o$ with respect to $d_\ast$.
Therefore $\t(d_\ast/d) \ge v/v_\ast$ where $v$ and $v_\ast$ are exponential volume growth rate relative to $d$ and $d_\ast$ respectively.
\qed

We can now conclude the proof of Theorem \ref{Thm:thm2}.

\begin{proof}[Proof of Theorem \ref{Thm:thm2}]
We have already proven the first part of the theorem in Theorem \ref{Thm:distortion}. 
Let us show the equivalence of statements (1), (2) and (3). 
Note that the equivalence of (2) and (3) is a consequence of Lemma \ref{Lem:asymprigidity} and Theorem \ref{Thm:rigidity}. 
(If two metrics $d, d_\ast$ are roughly similar, then the corresponding Manhattan curve is a straight line on the entire part, not just on the part connecting $(0, v)$ and $(v_\ast, 0)$.)
We therefore just need to prove the equivalence of (1) and (3) which we prove below:

Consider the Manhattan curve $\Cc_M$ and the function $\theta(a)$ for the pair $(d, d_\ast)$ and recall that, by Theorem \ref{Thm:C1},
we have that 
$\theta'(0) = - \tau(d_\ast/d)$. 
It follows, since the curve $\Cc_M$ passes through $(0, v)$ and $(v_\ast, 0)$,
that $\t(d_\ast/d)=v/v_\ast$ if and only if $\theta$ is a straight line on $[0,v_\ast]$. 
By Theorem \ref{Thm:rigidity} this is the case if and only if $d$ and $d_\ast$ are roughly similar. This concludes the proof.
\end{proof}

Let us record the following result on the asymptotics of a typical ratio between two stable translation lengths as it is of interest in its own right.

\begin{corollary}
For all $d, d_\ast \in \Dc_\G$, we have that
\[
\frac{\ell_\ast[\g_\x(t)]}{\ell[\g_\x(t)]} \to \t(d_\ast/d) \quad \text{as $t \to \infty$ for $\m$-almost every $\x \in \partial \G$},
\]
where $\m$ is a Patterson-Sullivan measure relative to $d$ 
and $\g_\x$ is a quasi-geodesic ray $\g_\x$ converging to $\x$.
\end{corollary}

\proof
This follows from Lemmas \ref{Lem:translation} and \ref{Lem:tau} since
$\t_{0, v}=\t(d_\ast/d)$ by Theorem \ref{Thm:distortion}.
\qed


\section{The $C^2$ regularity for strongly hyperbolic metrics}\label{Sec:C2}

The aim of this section is to deduce better regularity (i.e. higher order differentiability) for the Manhattan curve under the additional assumption that $d,d_\ast$ are strongly hyperbolic metrics 
(see Definition \ref{Def:SH} in Section \ref{Sec:hyperbolic}). 
The method we use also applies to word metrics,
in which case it is (in principle) possible to compute explicit examples; we provide some in the subsequent section (Section \ref{Sec:example}).
We will use automatic structures to introduce a symbolic coding for our group $\Gamma$. This will allow us to use techniques from thermodynamic formalism. We begin with some introductory material on these techniques.
For the thermodynamic formalism on non-topologically transitive systems, we follow \cite[Sections 3.2 and 3.3]{GouezelLocalLimit}.

\subsection{Automatic structures}
Fix a finite (symmetric) set of generators $S$ for $\G$.
An automaton $\Ac=(\Gc, \pi, S)$ is a triple consisting of a finite directed graph $\Gc=(V, E, s_\ast)$ where $s_\ast$ is a distinguished vertex called the initial state, a labeling $\pi:E \to S$ on edges by $S$ and a finite (symmetric) set of generators $S$.
Associated to every directed path $\o=(e_0, e_1, \dots, e_{n-1})$ in the graph $\Gc$
where the terminus of $e_i$ is the origin of $e_{i+1}$,
there is a path $\pi(\o)$ in the Cayley graph $\Cay(\G, S)$ issuing from the identity $\id$, $\pi(e_0)$, $\pi(e_0)\pi(e_1)$, $\dots$, $\pi(e_0)\cdots \pi(e_{n-1})$.
Let us denote by $\pi_\ast(\o)$ the terminus of the path $\pi(\o)$, i.e., $\pi_\ast(\o):=\pi(e_0)\cdots \pi(e_{n-1})$.

\begin{definition}
An automaton $\Ac=(\Gc, \pi, S)$ where $\Gc=(V, E, s_\ast)$ and a labeling $\pi: E \to S$
is called a {\it strongly Markov automatic structure} if
\begin{itemize}
\item[(1)] for every vertex $v \in V$ there is a directed path from the initial state $s_\ast$ to $v$,
\item[(2)] for every directed path $\o$ in $\Gc$ the associated path $\pi(\o)$ is a geodesic in the Cayley graph $\Cay(\G, S)$, and
\item[(3)] the map $\pi_\ast$ evaluating the terminus of a path yields a bijection from the set of directed paths from $s_\ast$ in $\Gc$ to $\G$.
\end{itemize}
\end{definition}

We sometimes abuse notation by identifying $\Ac$ with the underlying finite directed graph $\Gc$.
By a theorem of Cannon \cite{Cannon} every hyperbolic group admits a strongly Markov automatic structure for every finite symmetric set of generators $S$ (cf.\ \cite{Calegari}). 
Given an automatic structure $\Ac=(\Gc, \pi, S)$ for $(\G, S)$,
we write $\SS^\ast$ for the set of finite directed paths in $\Gc$ (not necessarily starting from $s_\ast$)
and $\SS$ for the set of semi-infinite directed paths $\o=(\o_i)_{i=0, 1, \dots}$ in $\Gc$.
Let $\wbar \SS:=\SS^\ast \cup \SS$. The function
$\pi_\ast: \SS^\ast \to \G$ naturally extends to
\[
\pi_\ast: \wbar \SS \to \G\cup \partial \G, \quad \o \mapsto \pi_\ast(\o),
\]
by mapping a sequence to the terminus of the geodesic segment or ray $\pi(\o)$ starting at $\id$ in $\Cay(\G, S)$.
We define a metric $d_{\wbar \SS}$ on $\wbar \SS$ by $d_{\wbar \SS}(\o, \o')=2^{-n}$ if $\o \neq \o'$ and $\o$ and $\o'$ coincide up to the $n$-th entry, and $d_{\wbar \SS}(\o,\o')=0$ if $\o=\o'$.

\subsection{Thermodynamic formalism}

\def\Pr{{\rm Pr}}

The shift map $\s:\wbar \SS \to \wbar \SS$ takes a (possibly finite) sequence $\o = (\o_o)_{i=0,1,\ldots}$ and maps it to $\s(\o) = (\o_{i+1})_{i=0,1,\ldots}$.  To ensure that $\s$ is well-defined, we include the empty path in $\wbar \SS$.
For every real-valued H\"older continuous function $\f: \wbar \SS \to \R$ (which we call a {\it potential}),
the transfer operator $\Lc_\f$ acting on the space of continuous functions $f$ on $\wbar \SS$ is defined by
\[
\Lc_\f f(\o)=\sum_{\s(\o')=\o}e^{\f(\o')}f(\o'),
\]
where for the empty path $\o=\emptyset$ the preimages of $\s$ are defined only by nonempty paths.
We say that the directed graph $\Gc$ is {\it recurrent} if there is a directed path between arbitrary two vertices.
We say that $\Gc$ is {\it topologically mixing} if there exists $n$  such that every pair of vertices is connected by a path of length $n$.
If $\Gc$ is recurrent but not topologically mixing, then there is an integer $p>1$ such that every loop (i.e. path starting and ending at the same vertex) in $\Gc$ has length divisible by $p$. Furthermore
 the set of vertices of $\Gc$ decomposes into $p$ subsets $V=\bigsqcup_{j \in \Z/p\Z}V_j$
where every edge with the origin in $V_j$ has the terminus in $V_{j+1}$.
We call this decomposition a {\it cyclic decomposition} of $V$.
Restricting $\s^p$ to $V_j$, we obtain a topological mixing subshift of finite type.
If $\Gc$ is not recurrent, then we decompose $\Gc$ into components --- these are the maximal induced subgraphs which are recurrent.
For each component $\Cc$, we define the transfer operator $\Lc_\Cc$ by restricting $\f$ to the paths staying in $\Cc$.
The spectral radius of $\Lc_\Cc$ is given by $e^{\Pr_\Cc(\f)}$ for some real value $\Pr_\Cc(\f)$.
This constant is obtained from the limit
\begin{equation}\label{Eq:Pr}
\Pr_\Cc(\f, \s)=\lim_{n \to \infty}\frac{1}{n}\log \sum_{[\o_0, \dots, \o_{n-1}]}\exp\(S_{[\o_0, \dots, \o_{n-1}]}(\f)\),
\end{equation}
where the summation is taken over all the cylinder sets of length $n$,
\[
S_{[\o_0, \dots, \o_{n-1}]}(\f):=\sup\{S_n \f(\o) \ : \ \o \in [\o_0, \dots, \o_{n-1}]\}
\quad \text{and} \quad S_n \f:=\sum_{i=0}^{n-1}\f\circ \s^i,
\]
(see \cite[Theorem 2.2]{ParryPollicott}; this follows from the Gibbs property of an eigenmeasure for each component in the cyclic decomposition).

Let
\[
\Pr(\f):=\max_\Cc \Pr_\Cc(\f),
\]
where the maximum is taken over all components $\Cc$ of $\Gc$.
We call a component $\Cc$ {\it maximal} if $\Pr_\Cc(\f)=\Pr(\f)$.
Note that the set of maximal components depends on $\f$.
We are interested in potentials that satisfy the following condition.

\begin{definition}
A potential $\f$ is called {\it semisimple} if there are no directed paths from any maximal component to any other maximal components.
\end{definition}

We denote by $\Hc$ the space of H\"older continuous functions on $\wbar \SS$ with some fixed exponent, whose explicit value is not used, and by $\|\cdot\|_\Hc$ the corresponding H\"older norm.

\begin{theorem}[Theorem 3.8 in \cite{GouezelLocalLimit}]\label{Thm:Gouezel}
Suppose that $\f$ is a semisimple potential and $\Cc_1, \dots, \Cc_I$ are the corresponding maximal components, each with period $p_i$ and cyclic decomposition $\Cc_i=\bigsqcup_{j \in \Z/p_i\Z}\Cc_{i, j}$.
Then there exist H\"older continuous functions $h_{i, j}$ and measures $\lambda_{i, j}$ with $\int_{\wbar \SS}h_{i, j}d\lambda_{i, j}=1$ such that 
\[
\left\|\Lc_\f^n f -e^{\Pr(\f) n}\sum_{i=1}^I\sum_{j \in \Z/p_i \Z}\(\int_{\wbar \SS}f\,d\lambda_{i, (j-n\, {\rm mod\,} p_i)}\)h_{i, j}\right\|_\Hc
\le C \|f\|_\Hc e^{(\Pr(\f)-\e_0)n},
\]
for every H\"older continuous function $f$,
where positive constants $C$ and $\e_0>0$ are independent of $f$,
and the probability measures
$\m_i=(1/p_i)\sum_{j \in \Z/p_i \Z}h_{i, j}\lambda_{i, j}$ 
are invariant under the shift $\s$. The measures $\mu_i$ are also ergodic.
\end{theorem}

\begin{remark}\label{Rem:Thm:Gouezel}
In the statement of Theorem \ref{Thm:Gouezel},
if we define $\Cc_{i, j, \to}$ to be the set of edges which can be reached by a path from $\Cc_{i, j}$ of length divisible by $p_i$,
and $\Cc_{\to, i, j}$ to be the set of edges which we can reach $\Cc_{i, j}$ with a path of length divisible by $p_i$,
then the function $h_{i, j}$ is bounded from below on the paths starting with edges in $\Cc_{i, j, \to}$ and the empty path,
and takes $0$ elsewhere. Furthermore the measure $\lambda_{i, j}$ is supported on the set of infinite paths starting with edges in $\Cc_{\to, i, j}$ and eventually staying in $\Cc_i$.
\end{remark}

We will use both of the  measures $\m_i$ and $\lambda_i:=\sum_{j=0}^{p_i-1}\lambda_{i, j}$ for each $i=1, \dots, I$.
They have different supports on the space of paths $\wbar \SS$, and $\m_i$ is $\s$-invariant while $\lambda_i$ is not.

\begin{lemma}[Lemma 3.9 in 
\cite{GouezelLocalLimit}]\label{Lem:Gouezel_lambda+}
In the notation in Theorem \ref{Thm:Gouezel}, let $\lambda_i:=\sum_{j=0}^{p_i-1}\lambda_{i, j}$.
Then $\s_\ast\lambda_i$ is absolutely continuous with respect to $\lambda_i$.
\end{lemma}

\begin{proposition}[Proposition 3.10 in \cite{GouezelLocalLimit}]\label{Prop:perturbation}
Suppose that $\f$ is a semisimple potential in $\Hc$ and $\Cc_1, \dots, \Cc_I$ are the corresponding maximal components.
Then there exist positive constants $C, \e_0>0$ such that
for all small enough $\p \in \Hc$,
there exist H\"older continuous functions $h_{i, j}^\p$ and measures $\lambda_{i, j}^\p$ with the same support as $h_{i, j}$ and $\lambda_{i, j}$ respectively, such that 
\[
\left\|\Lc_{\f+\p}^n f-\sum_{i=1}^I e^{\Pr_{\Cc_i}(\f+\p)n}\sum_{j \in \Z/p_i\Z}\(\int_{\wbar \SS}f\,d\lambda^\p_{i, (j-n\, {\rm mod\,} p_i)}\)h_{i, j}^\p\right\|_\Hc \le C\|f\|_\Hc e^{(\Pr(\f)-\e_0)n},
\]
for all $f \in \Hc$.
Moreover, the maps $\p \mapsto \Pr_{\Cc_i}(\f+\p)$, $\p \mapsto h_{i, j}^\p$ and $\p \mapsto \lambda_{i, j}^\p$ (from $\Hc$ to $\R$, $\Hc$ and the dual of $\Hc$ respectively)
are each real analytic in a small neighborhood of $0$ in $\Hc$.
\end{proposition}

Let $[E_\ast]$ denote the set of paths in $\wbar \SS$ starting at $s_\ast$. If $\1_{[E_\ast]}$ denotes the corresponding indicator function then
\[
\Lc_\f^n \1_{[E_\ast]}(\emptyset)=\sum e^{S_n \f(\o)}, \quad \text{where $S_n \f(\o)=\sum_{k=0}^{n-1} \f(\s^{k}(\o))$},
\]
and the summation is taken over all paths $\o$ of length $n$ starting from $s_\ast$.

\begin{lemma}\label{Lem:semisimple}
For every H\"older continuous potential $\f$ and for every integer $k \ge 1$,
if there exists a path from $s_\ast$ in $\Ac$ containing edges successively from $k$ different maximal components for $\f$,
then there exists a constant $C>0$ such that for all $n \ge 1$,
\[
\Lc_\f^n \1_{[E_\ast]}(\emptyset) \ge C n^{k-1}e^{n \Pr(\f)}.
\]
On the other hand, if there are $L$ components in $\Ac$, 
then there exists a constant $C>0$ such that for all $n \ge 1$,
\[
\Lc_\f^n \1_{[E_\ast]}(\emptyset) \le C n^L e^{n \Pr(\f)}.
\]
\end{lemma}

\proof
This lemma is a special case of \cite[Lemma 3.7]{GouezelLocalLimit} (the proof of the second part is found in \cite[Lemma 4.7]{THaus}).
\qed

\subsection{Semisimple potentials}
In this section we use thermodynamic formalism to link the geometric measures constructed in Section \ref{Sec:PS} to certain measures on $\wbar \SS$.
The key result that allows us to do this
is the following.

\begin{lemma}\label{Lem:pr-ss}
Let $\psi$ be a $\G$-invariant tempered potential relative to $d_S$ on $\G$ with exponent $\th$ (see Definition \ref{Def:tempered}).
If for a strongly Markov automatic structure $\Ac=(\Gc, \pi, S)$ 
the corresponding shift space $(\wbar \SS, \s)$ admits a H\"older continuous potential $\Psi$
such that
\begin{equation}\label{Eq:combable}
S_n \Psi(\o)=\sum_{i=0}^{n-1}\Psi(\s^i(\o))=-\psi(o, \pi_\ast (\o)) \quad \text{for all $\o=(\o_0, \dots, \o_{n-1}) \in \SS^\ast$},
\end{equation}
then $\Psi$ is semisimple and $\Pr(\Psi)=\th$.
Moreover, for each $a \in \R$, the potential $a \Psi$ is semisimple.
\end{lemma}

\proof
Note that for the potential $\Psi$, we have that for all $n$,
\[
\Lc_{\Psi}^n\1_{[E_\ast]}(\emptyset)=\sum_{|x|_S=n}e^{-\psi(o,x)}.
\]
Hence Lemmas \ref{Lem:Coornaert} and \ref{Lem:semisimple} show that $\Pr(\Psi)=\th$.
If the potential $\Psi$ is not semisimple,
then there is a directed path in the automatic structure $\Ac=(\Gc, \pi, S)$ starting from $s_\ast$ passing through $k$ distinct maximal components for $\Psi$ for $k>1$.
The first part of Lemma \ref{Lem:semisimple} implies that 
$\Lc_{\Psi}^n\1_{[E_\ast]}(\emptyset) \ge C n^{k-1}e^{n \Pr(\Psi)}$ for all $n \ge 0$.
This however contradicts Lemma \ref{Lem:Coornaert}.
Therefore $\Psi$ is semisimple.
Furthermore for every $a \in \R$, the same proof applies to $a \psi$, and the potential $a \Psi$ is semisimple.
\qed

For each semisimple potential $\Psi$ on $\wbar \SS$,
let $\Cc_i$ for $i=1, \dots, I$ be the maximal components for $\Psi$.
Let $\lambda_{i, j}$ for $i =1, \dots, I$ and $j=0, \dots, p_i$ be the measures obtained in Theorem \ref{Thm:Gouezel} applied to the potential $\Psi$.
We define $\lambda_i:=\sum_{j=0}^{p_i-1}\lambda_{i, j}$ and $\lambda_\Psi:=\sum_{i=1}^I\lambda_i$.
Let us denote by $\m_\psi$ a finite Borel measure on $\partial \G$ satisfying \eqref{Eq:generalPS} with exponent $\th$ relative to $(\psi, d)$ (which has been constructed in Proposition \ref{Prop:generalPS}).

\begin{lemma}\label{Lem:lambda+}
Assume that $\psi$ and $\Psi$ are as in Lemma \ref{Lem:pr-ss}.
Then, the push-forward of $\lambda_\Psi(\,\cdot\, \cap [E_\ast])$ by $\pi_\ast$ is comparable to $\m_\psi$.
\end{lemma}
\proof
For all $n$, let $\wt m_n$ be the finite measure on $\wbar \SS$ defined by the positive linear functional $f \mapsto e^{-n \Pr(\Psi)}\cdot \Lc_{\Psi}^n f(\emptyset)$.
If the maximal components for the potential $\Psi$ have periods $p_i$ for $i=1, \dots, I$,
then let $p$ be the least common multiple of these periods.
Theorem \ref{Thm:Gouezel} shows that for every H\"older continuous function $f$ on $\wbar \SS$,
we have that for each $q=0, \dots, p-1$,
\[
e^{-(np+q) \Pr(\Psi)}\cdot \Lc_{\Psi}^{np+q}f(\emptyset) \to \sum_{i=1}^I \sum_{j=0}^{p_i}\int_{\wbar \SS}f\,d\lambda_{i, (j-q \ {\rm mod\,} p_i)}h_{i, j}(\emptyset) 
\quad \text{as $n \to \infty$}.
\]
This convergence holds for all continuous functions $f$ on $\wbar \SS$;
indeed, we approximate $f$ by H\"older continuous functions and use 
$|e^{-n \Pr(\Psi)}\cdot \Lc_{\Psi}^n f(\emptyset)| \le C\|f\|_\infty$ for all $n$, where $\|\cdot\|_\infty$ stands for the supremum norm.
This shows that $\wt m_{np+q}$ weakly converges to a measure $\wt m_q$ for each $q=0, \dots, p-1$.
Since $c_1 \le h_{i, j}(\emptyset) \le c_2$ for some $c_1, c_2>0$ (see Theorem \ref{Thm:Gouezel} and Remark \ref{Rem:Thm:Gouezel}),
all $\wt m_q$ are comparable with $\sum_{i, j}\lambda_{i, j}$.
If we denote by $\wt m_\infty$ the weak limit of $\sum_{k=0}^n\wt m_k(\,\cdot\, \cap[E_\ast])/\sum_{k=0}^n\wt m_k([E_\ast])$,
then the measure $\pi_\ast \wt m_\infty$ is actually $\m_\psi$.
Indeed, for every continuous function $f$ on $\G \cup \partial \G$,
we have 
\begin{align*}
e^{-n\Pr(\Psi)}\Lc_{\Psi}^n (\1_{[E_\ast]}\cdot f \circ \pi_\ast)(\emptyset)
&=e^{-n\Pr(\Psi)}\sum_{\o \ \text{of length $n$ from $s_\ast$}} e^{S_n \Psi(\o)}f(\pi_\ast(\o))\\
&=e^{-n\Pr(\Psi)}\sum_{|x|_S=n}e^{-\psi(o, x)}f(x),
\end{align*}
where the last line follows since the map $\pi_\ast$ induces a bijection from the set of paths of length $n$ starting at $s_\ast$ to the set of $x \in \G$ with $|x|_{S}=n$ and \eqref{Eq:combable}.
Since $\Pr(\Psi)=\th$ by Lemma \ref{Lem:pr-ss},
this shows that the measure $\pi_\ast \wt m_\infty$ is comparable with $\m_\psi$ obtained by the Patterson-Sullivan procedure,
and
for $\lambda_\Psi:=\sum_{i, j}\lambda_{i, j}$, the measure $\pi_\ast\lambda(\,\cdot \,\cap [E_\ast])$ is comparable with $\m_\psi$.
\qed

\begin{example}\label{Ex:combable1}
For every pair of finite symmetric sets of generators $S$ and $S_\ast$,
there exist a strongly Markov automatic structure $\Ac=(\Gc, \pi, S)$ and a function $d\phi_{S_\ast}: E(\Gc) \to \Z$
such that
\[
|\pi_\ast(\o)|_{S_\ast}=\sum_{i=0}^{n-1}d\phi_{S_\ast}(\o_i) \quad \text{for every path $\o=(\o_0, \dots, \o_{n-1})$ from $s_\ast$ on $\Gc$},
\]
where $\Gc=(V(\Gc), E(\Gc))$ is the underlying directed graph of $\Ac$.
This is proved in \cite[Lemma 3.8]{CalegariFujiwara2010} (see also \cite[Theorem 6.39]{Calegariscl}).
Let us define a function $\Psi_{S_\ast}: \overline{\Sigma} \to \mathbb{R}$ by setting 
\[
\Psi_{S_\ast}(\omega) = -d\phi_{S_\ast}(\omega_0) \quad \text{for $\omega \in \wbar \SS$}. 
\]
This function depends only on the first coordinate of $\o$ and is H\"older continuous with respect to the metric $d_{\wbar \SS}$. Further, by construction, for $\omega = (\omega_0, \ldots, \omega_{n-1})  \in \Sigma^\ast$ we have that
\[
S_n \Psi_{S_\ast}(\omega) = -\sum_{i=0}^{n-1}d\phi_{S_\ast}(\o_i) = -d_{S_\ast}(o, \pi_\ast(\o)).
\]
This shows that, on $\Sigma^\ast$, the Birkhoff sums of $\Psi_{S_\ast}$ encode information about the metric $d_{S_\ast}$. 
\end{example}

\begin{example}\label{Ex:combable2}
Let $d \in \Dc_\G$ be a strongly hyperbolic metric.
For every finite (symmetric) set of generators $S$, 
we consider a subshift $\overline{\Sigma}$ arising from a strongly Markov structure $\Ac=(\Gc, \pi, S)$. 
Since the Busemann function for a strongly hyperbolic metric is defined as limits (Section \ref{Sec:PS}),
if we define
\[
\b_o(x, y):=d(x, y)-d(o, y) \quad \text{for $(x, y) \in \G \times (\G \cup \partial \G)$},
\]
then its restriction on $\G \times \partial \G$ is the original Busemann function (based at $o$) for $d$.
Let 
\[
\Psi(\o):=\b_o(\pi_\ast(\o_0), \pi_\ast(\o)) \quad \text{for $\o \in \wbar \SS$}.
\]
Note that $\Psi$ is H\"older continuous with respect to $d_{\wbar \SS}$ by the definition of strong hyperbolicity (see Section \ref{Sec:PS}).
Furthermore for $\omega \in \Sigma^\ast$,
if $n = |\pi_\ast(\o)|_S$,
then we have that
\[
S_n \Psi(\omega) = -d(o,\pi_\ast(\o)).
\]
\end{example}

\begin{remark}
It is important to note that for \textit{any} strongly Markov structure $\Ac=(\Gc, \pi, S)$ and \textit{any} strongly hyperbolic metric $d \in \Dc_\Gamma$ we can find a function $\Psi$ on $\wbar \SS$ encoding $d$. 
It is not clear if we can do the same when $d$ is a word metric; in which case we only know the existence of \textit{some} strongly Markov structure and a function on $\wbar \SS$ encoding $d$ (see Example \ref{Ex:combable1}).
We will exploit this freedom of choice for strongly hyperbolic metrics in our proof of Theorem \ref{Thm:thm3}.
\end{remark}

\subsection{Proof of the $C^2$-regularity}\label{Sec:proofofC2}

In general,
restricting to each component $\Cc_i$, 
the pressure function $\Pr_{\Cc_i}(\Psi)$ is real analytic in $\Psi$. Furthermore, for every $\Psi_0, \Psi \in \Hc$,
\[
\Pr_{\Cc_i}(\Psi_0+s\Psi)=\Pr_{\Cc_i}(\Psi_0)+s \t_i+\frac{1}{2}s^2\s_i^2+O(s^3) \quad \text{as $s \to 0$},
\]
where
\[
\t_i:=\int_{\SS}\Psi\,d\m_i \quad \text{and} \quad \s_i^2:=\lim_{n \to \infty}\frac{1}{n}\int_{\SS}\(S_n \Psi-n \t_i\)^2\,d\m_i.
\]
We will prove that $\t_i$ respectively $\s_i^2$ coincide on all maximal components for $\Psi_0$.

\begin{proposition}\label{Prop:2nd-d}
Let $\psi$ be a $\G$-invariant tempered potential relative to $d_S$ on $\G$ with exponent $\th$, and
\[
\th(a):=\lim_{n \to \infty}\frac{1}{n}\log \sum_{|x|_S=n}e^{-a \psi(o, x)} \quad \text{for $a \in \R$}.
\]
Suppose that the shift space $(\wbar \SS, \s)$ corresponding to a strongly Markov automatic structure $\Ac=(\Gc, \pi, S)$  admits a H\"older continuous potential $\Psi$
satisfying \eqref{Eq:combable}.
Then the function $\th(a)$ is twice continuously differentiable in $a \in \R$.
\end{proposition}

\proof
By Lemma \ref{Lem:pr-ss}, we have that $\th(a)=\max_\Cc \Pr_\Cc(a \Psi)$ for every $a \in \R$.
Proposition \ref{Prop:perturbation} shows that for each maximal component $\Cc_i$, $i=1, \dots, I$ for $a \Psi$,
we have $\Pr'_{\Cc_i}(a \Psi)=\int_{\SS_i} \Psi\,d\m_i$,
where $\SS_i$ is the set of paths staying in $\Cc_i$ for all time.
Let us prove that $\int_{\SS_i}\Psi\,d\m_i=\int_{\SS_j}\Psi\,d\m_j$ for all $i, j \in \{1, \dots, I\}$.
For every $\t \in \R$, let $A(\t)$ be the set of boundary points $\x$ in $\partial \G$ for which 
there exists a unit speed geodesic ray $\x_n$ in $\Cay(\G, S)$ converging to $\x$ such that
\[
\lim_{n \to \infty}\frac{1}{n}\psi(o, \x_n)=-\t.
\]
Note that if this convergence holds for some geodesic ray toward $\x$, then in fact this holds for every geodesic ray toward $\x$ since arbitrary two geodesic rays converging to the same extreme point are eventually within bounded distance up to shifting the parameterizations
(where all geodesic rays are parameterized with unit speed). 
This shows that the set $A(\t)$ is $\G$-invariant.
Let $\t_i:=\int_{\SS_i}\Psi\,d\m_i$, where $\m_i=(1/p_i)\sum_{j=0}^{p_i-1}h_{i, j}\lambda_{i, j}$.
If we define 
\[
U_i:=\Big\{\o \in \SS_i \ : \ \lim_{n \to \infty}\frac{1}{n}S_n \Psi(\o)=\t_i\Big\},
\]
then $\pi_\ast(U_i)\subset A(\t_i)$ and the Birkhoff ergodic theorem implies that $\m_i\(U_i\)=1$ since $\m_i$ is ergodic by Theorem \ref{Thm:Gouezel}.
We shall show that $\t_i=\t_j$ for all $i, j \in \{1, \dots, I\}$.
Let $U_i^c:=\wbar \SS\setminus U_i$.
Since $\lambda_i$ and $\m_i$ are equivalent on $\SS_i$,
we have $\lambda_i\(\SS_i\cap U_i^c\)=0$.
This implies that $\lambda_i(U_i^c)=0$.
Indeed, note that $\s^k U_i^c \subset U_i^c$, and
$\s^{-k}\SS_i \cap U_i^c=\s^{-k}(\SS_i\cap \s^k U_i^c) \subset \s^{-k}(\SS_i \cap U_i^c)$.
Since $\lambda_i\circ \s^{-1}$ is absolutely continuous with respect to $\lambda_i$ by Lemma \ref{Lem:Gouezel_lambda+},
we have $\lambda_i(\s^{-k}\SS_i \cap U_i^c)=0$, 
and
since $\wbar \SS=\bigcup_{k=0}^\infty \s^{-k}\SS_i$ modulo $\lambda_i$-null sets,
we obtain $\lambda_i(U_i^c)=0$.

We then have that $\lambda_i(U_i\cap [E_\ast])>0$ since $U_i$ has full $\lambda_i$-measure and $\lambda_i$ assigns positive measure to $[E_\ast]$ by Theorem \ref{Thm:Gouezel}.
Therefore 
\[
\lambda_i(\pi_\ast^{-1}A(\t_i) \cap [E_\ast]) \ge \lambda_i(U_i\cap [E_\ast])>0
\]
and Lemma \ref{Lem:lambda+} implies that 
$\m_\psi(A(\t_i))>0$.
As we have noted, $A(\t_i)$ is $\G$-invariant, and since $\m_\psi$ is ergodic with respect to the $\G$-action on $\partial \G$ by Lemma \ref{Lem:ergodic},
the set $A(\t_i)$ has the full $\m_\psi$ measure.
Since this is true for all $i=1, \dots, I$, all $\t_i$ and thus $\Pr_{\Cc_i}'(a \Psi)$ coincide.
This shows that $\th(a)$ is differentiable at every $a \in \R$.

Let $\t$ be the common value of all of the $\t_i$ at $a \in \R$.
For each $i=1, \dots, I$, we have by \cite[Proposition 4.11]{ParryPollicott},
$\Pr_{\Cc_i}''(a \Psi)=a^2 \s_i^2$ where 
\[
\s_i^2=\lim_{n \to \infty}\frac{1}{n}\int_{\SS_i}\(S_n \Psi-n \t\)^2\,d\m_i.
\]
We define the set $B(\s_i)$ of points $\x$ in $\partial \G$ 
for which there exists a (unit speed) geodesic ray $\x$ in $\Cay(\G, S)$ converging to $\x$
such that the following double limits hold:
\[
\s_i^2=\lim_{n \to \infty}\frac{1}{n}\s_i^2(n) \quad \text{where $\s_i^2(n):=\lim_{m \to \infty}\frac{1}{m}\sum_{k=0}^{m-1}(-\psi(\x_k, \x_{n+k})-n \t)^2$}.
\]
Note that $B(\s_i)$ is $\G$-invariant since $\p$ is $\G$-invariant.
If we define $\m:=\sum_{i=1}^I \m_i$,
then applying to
the Birkhoff ergodic theorem countably many times on a dense subset of the space of continuous functions on $\SS$, 
we have that for $\m$-almost every $\o \in \SS$,
the measures $(1/n)\sum_{k=0}^{n-1}\d_{\s^k \o}$
weakly converge to a measure $\m_\o$ on $\SS$ and for $\m_i$-almost every $\o \in \SS_i$ one has $\m_\o=\m_i$ for each $i =1, \dots, I$.
Note that 
\[
\m_\o \circ \s^{-1}=\m_{\s \o}=\m_\o, \quad \text{$\m$-almost everywhere}.
\]
Let us define
\[
V_i:=\Big\{\o \in \SS_i \ : \ \s_i^2=\lim_{n \to \infty}\frac{1}{n}\int_{\SS}(S_n \Psi-n \t)^2\,d\m_{\o}\Big\}.
\]
We then have that $\pi_\ast (V_i) \subset B(\s_i)$ since $\psi$ is a $\G$-invariant tempered potential relative to $d_S$ and $\m_i(V_i)=1$.
Since $V_i$ is $\s$-invariant modulo $\m_i$-null sets,
the same argument as above implies that $\m_\psi(B(\s_i))>0$ and all $\s_i^2$ coincide.
This shows that $\Pr_{\Cc_i}''(a\Psi)$ coincide for all $i=1, \dots, I$.
Since each $\Pr_{\Cc_i}(a)$ is real analytic and $\th(a)$ coincides with the maximum of finitely many $\Pr_{\Cc_i}(a)$ on a neighborhood of $a$ by Proposition \ref{Prop:perturbation},
the function $\th(a)$ is twice continuously differentiable in $a \in \R$.
\qed

\begin{theorem}\label{Thm:2nd-d-words}
For every pair of finite symmetric sets of generators $S$ and $S_\ast$,
if
\[
\th_{S_\ast/S}(a):=\lim_{n \to \infty}\frac{1}{n}\log \sum_{|x|_S=n}e^{-a |x|_{S_\ast}} \quad \text{for $a \in \R$},
\]
then $\th_{S_\ast/S}$ is twice continuously differentiable in $a \in \R$.
\end{theorem}
\proof
This follows from Proposition \ref{Prop:2nd-d} and Example \ref{Ex:combable1}.
\qed

We now consider the case when $d, d_\ast \in \Dc_\G$ are both strongly hyperbolic and we want to show that the associated Manhattan curve is $C^2$.
In the previous part we exploited the fact that the subshift $\overline{\Sigma}$ encoded one of the metrics $d$ that we were considering. 
It is not clear how to exploit this fact when both metrics are strongly hyperbolic. To get around this issue we take a word metric $d_S$ on $\Gamma$ associated to a finite symmetric set of generators $S$ and use this to introduce a subshift $\overline{\Sigma}$ on which we are able to encode information about the metrics $d,d_\ast$.

For the rest of this section assume that we have two strongly hyperbolic metrics $d,d_\ast \in \Dc_\G$ and that we have arbitrarily chosen a finite symmetric set of generators $S$ for $\Gamma$. 
We begin by constructing a useful two parameter family of measures on $\partial \Gamma$. 
For each $(a,b) \in \mathbb{R}^2$, let $\wt \th(a, b)$ be the abscissa of convergence of 
$$\sum_{x\in\Gamma} \exp( - a d_\ast(o,x)-b d(o,x)-sd_S(o, x))$$
as $s$ varies.

To understand this summation we use the measures we constructed in Section \ref{Sec:PS}.
Since $a d_\ast+b d$ is a $\G$-invariant tempered potential relative to $d_S$ for each $(a, b) \in \R^2$ (Example \ref{Ex:tempered}),
Proposition \ref{Prop:generalPS} implies that for each $(a, b) \in \mathbb{R}^2$,
there exists a measure $\mu_{a, b, S}$ on $\partial \Gamma$ such that for $x\in\Gamma$,
\[
C_{a, b}^{-1}\le  \exp\(a\beta_{\ast o}(x,\xi) + b\beta_o(x,\xi)+\wt \th(a, b) \beta_{S o}(x,\xi)\) \cdot \frac{dx_\ast \mu_{a,b,S}}{d\mu_{a,b,S}}(\xi) \le C_{a, b},
\]
where $\b_{\ast o}$, $\b_o$ and $\b_{S o}$ are Busemann functions (based at $o$) for $d_\ast$, $d$ and $d_S$ respectively, and $C_{a,b}$ is a positive constant. 
By Lemma \ref{Lem:Coornaert}, we have that
\begin{equation}\label{Eq:counting2}
\sum_{|x|_S = n} e^{-ad_\ast(o,x) -bd(o,x)} \asymp_{a, b} \exp\(\wt \th(a, b) n\) \quad \text{for all integers $n \ge 0$}.
\end{equation}
For each fixed $a \in \R$, the function $b \mapsto \wt \th(a, b)$ is continuous by the H\"older inequality.
Note that for each fixed $a \in \R$,
\[
b>\th(a) \implies \wt \th (a, b)<0 \quad \text{and} \quad
b<\th(a) \implies \wt \th (a, b)>0.
\]
Therefore, combining with the continuity of $\wt \th(a, b)$ in $b$, we have that $\wt \th(a, b)=0$ if and only if $b=\th(a)$.

\def\Cov{{\rm Cov}}

Now consider the subshift of finite type $\overline{\Sigma}$ arising from a coding corresponding to $S$. 
Let $\Psi, \Psi_\ast : \Sigma \to \mathbb{R}$ be H\"older continuous potentials that encode $d$, $d_\ast$ respectively as in Example \ref{Ex:combable2}. 
For each $(a,b) \in \R^2$, 
by Lemma \ref{Lem:pr-ss} (adapted to \eqref{Eq:counting2})
the potential $\Psi_{a,b} =a \Psi_\ast + b \Psi$ is semisimple and $\Pr(\Psi_{a, b})=\wt \th(a, b)$.

Consider a point $(a, b) \in \mathbb{R}^2$ and take a maximal component $\Cc$ for $\Psi_{a,b}$. Let $\mu_\Cc$ denote the measure corresponding to $\Psi_{a, b}$ on $\Cc$ from applying Theorem \ref{Thm:Gouezel}.
Since for each $\Cc$, the function $\Psi_{a, b}$ is real analytic in $(a, b) \in \R^2$,
it admits a Taylor expansion
with Jacobian $J_\Cc(a, b)$ and a symmetric Hessian $\Cov_\Cc(a, b)$. 
More precisely,
letting 
$\Psi_1:=\Psi_\ast$ and $\Psi_2:=\Psi$, 
we have $J_\Cc(a, b) = \left( J_1(a, b), J_2(a, b)\right)$
where
\begin{align*}
J_i(a, b) :=\int_{\SS_\Cc}\Psi_i\,d\m_\Cc
= \frac{\partial}{\partial s_i}\Big|_{(a,b)} \Pr_\Cc(\Psi_{s_1, s_2}) \quad \text{for $i=1, 2$},
\end{align*}
and $\Cov_\Cc(a, b)$ has entries
\begin{align*}
\Cov_\Cc(a, b)_{i,j} 
&= \frac{\partial^2}{\partial s_i \partial s_j}\Big|_{(a, b)} \Pr_\Cc(\Psi_{s_1,s_2})\\
&= \lim_{n\to\infty} \frac{1}{n} \int_{\Sigma_{\Cc}} \ (S_n\Psi_i(\o) - n J_i(a,b))(S_n \Psi_j(\o) - n J_j(a,b))\,d\mu_\Cc(\o),
\end{align*}
for $i, j=1, 2$; this follows from the one-parameter case by considering
$s \mapsto \Psi_{a+s, b+s}$ and differentiating at $s=0$
\cite[Proposition 4.11]{ParryPollicott}. 
We know that the potentials $\Psi_1$ and $\Psi_2$ satisfy a (possibly degenerate) multi-dimensional central limit theorem with respect to $\mu_\Cc$ on $\Sigma_\Cc$. 
That is, the distribution of 
\[
\frac{(S_n \Psi_1(\o),  S_n \Psi_2(\o))- n J_\Cc(a,b)}{\sqrt{n}} 
\]
under $\m_\Cc$ weakly
converges to a two dimensional Gaussian distribution with covariance matrix $\Cov_\Cc(a,b)$ as $n\to\infty$. 
It is useful to keep this in mind throughout the following.
Furthermore we note that $\Psi$ and $\Psi_\ast$ may vanish at certain points, however
there exists $n$ such that $S_n \Psi$ and $S_n \Psi_\ast$ are strictly negative functions.
Therefore for $\Psi_{a, b}$ and for the corresponding on maximal components $\Cc$,
we have that $J_\Cc(a, b)\neq (0, 0)$ whenever $b=\th(a)$ for all $a \in \R$.
This fact is crucial when we appeal to the implicit function theorem later in our discussion.

\begin{proposition}\label{Prop:multi}
For every $(a, b) \in \R^2$, the Jacobian $J_\Cc(a,b)$ and the Hessian matrix $\Cov_\Cc(a,b)$ do not depend on the choice of maximal component $\Cc$.
\end{proposition}

\begin{proof}
Showing that $J_\Cc(a, b)$ is independent of the choice of maximal component $\Cc$ for $\Psi_{a, b}$
is analogous to the first part in the proof of Proposition \ref{Prop:2nd-d};
we omit the details.
We will show that $\Cov_\Cc(a,b)$ is independent of the maximal component $\Cc$. 
The proof follows the same lines as in the proof of Proposition \ref{Prop:2nd-d} but we need to adapt the arguments to the multidimensional setting.

To simplify the following exposition, let us fix $a$ and $b$ and suppress their dependence in the notations. We also
write $d_1:=d_\ast$ and $d_2:=d$, and denote the corresponding potentials by $\Psi_1=\Psi_\ast$ and $\Psi_2=\Psi$.
For $\Cov_\Cc=(\s_{i, j})_{i, j=1,2}$,
we define the set $B(\Cov_\Cc)$
of points in $\partial \Gamma$ for which there exists a geodesic ray $\xi \in \text{Cay}(\Gamma, S)$ converging to $\xi$ such that
for all $i, j=1, 2$, we have 
\[
\s_{i,j} = \lim_{n\to\infty} \frac{1}{n} \sigma_{i,j}(n),
\]
where
\[
\sigma_{i,j}(n) = \lim_{m\to\infty} \frac{1}{m} \sum_{k=0}^{m-1} (-d_i(\xi_k,\xi_{n+k}) - nJ_i)(-d_j(\xi_k, \xi_{n+k}) -nJ_j).
\]
For each $\Psi_{a, b}$ we have a measure $\lambda_{a, b}$ by Lemma \ref{Lem:lambda+}
such that the push-forward of $\lambda_{a, b}$ restricted on $[E_\ast]$
by $\pi_\ast$ is comparable to $\mu_{a, b, S}$,
which is ergodic with respect to the $\G$-action on $\partial \G$ by Lemma \ref{Lem:ergodic}.
By comparing the set $B(\Cov_\Cc)$ to the set
\[
V_\Cc := \bigcap_{i, j=1,2}\left\{ \omega \in \Sigma_\Cc \ : \ \sigma_{i,j} = \lim_{n\to\infty} \frac{1}{n} \int_\Sigma (S_n \Psi_i - nJ_i) (S_n \Psi_j - nJ_j) d \mu_\o \right\},
\]
where $\m_\o$ is as in the proof of Proposition \ref{Prop:2nd-d}, 
we see that the matrix $\Cov_\Cc$ does not depend on the component $\Cc$. 
This concludes the proof.
\end{proof}

\begin{theorem}\label{Thm:2nd-d-sh}
For every pair of strongly hyperbolic metrics $d$ and $d_\ast$ on $\G$,
the corresponding function $\th(a)$ is twice continuously differentiable in $a \in \R$.
\end{theorem}

\begin{proof}
 For each $(a,b) \in \R^2$, let $\Cc$ be a maximal component for the potential $\Psi_{a, b}$.
 Proposition \ref{Prop:multi} shows that $\Pr_\Cc(\Psi_{s_1, s_2})$ admits the Taylor expansion whose terms up to the second order are independent of the choice of $\Cc$.
 This implies that since $\Pr(\Psi_{a, b})$ is given by the maximum over finitely many functions $\Pr_\Cc(\Psi_{a, b})$ and $\Pr(\Psi_{a, b})=\wt \th(a, b)$,
 the function $\wt \th(a, b)$ is twice continuously differentiable in $(a, b) \in \R^2$.
 Note that $\wt \th(a, b)=0$ if and only if $b=\th(a)$ for all $(a, b) \in \R^2$,
 and for every $(a,b) \in \R^2$ with $\wt \th(a, b)=0$ and for every maximal component $\Cc$, we have that $J_\Cc(a,b)\neq (0, 0)$ (see the discussion just before Proposition \ref{Prop:multi}).
 Therefore the implicit function theorem implies that $\theta$ is twice continuously differentiable.
\end{proof}

Now Theorem \ref{Thm:thm3} follows from Theorem \ref{Thm:2nd-d-sh}.

Note that this result and the arguments we applied to prove it are independent of the choice of $S$. 
However in the case when $\Gamma$ admits a finite symmetric set of generators $S$ such that the underlying directed graph of an automaton has only one recurrent component, 
then the Manhattan curve associated to two strongly hyperbolic metrics in $\Dc_\Gamma$ is real analytic. 
This is because, in this case, $\Pr(\Psi_{a, b})$ is real analytic in $(a, b) \in \R$ as there is only one maximal component which is recurrent and all the others are transient (not recurrent). 
For example, the fundamental groups of closed orientable surface of genus at least $2$ admit such automata with the standard set of generators since they have a single relator.
For more general cocompact Fuchsian groups, see \cite{SeriesTheInfiniteWord}.

\subsection{Pairs of word metrics}
In this section we deduce further rigidity results for word metrics. Recall that the Manhattan curve $\theta_{S_\ast/S}: \mathbb{R} \to \mathbb{R}$ associated to every pair of word metrics $d_S$, $d_{S_\ast}$ is twice continuously differentiable.

\begin{theorem} \label{thmwords} 
Let $\Gamma$ be a non-elementary hyperbolic group and $d_S$ and $d_{S_\ast}$ be word metrics associated to finite symmetric sets of generators $S$ and $S_\ast$ respectively. 
If we denote the Manhattan curve for the pair $(d_S, d_{S_\ast})$ by
\[
\Cc_M=\{(a, b) \in \R^2 \ : \ b=\th_{S_\ast/S}(a)\},
\]
then the following are equivalent:
\begin{enumerate}
\item[\upshape(1)] $d_S$ and $d_{S_\ast}$ are not roughly similar,
\item[\upshape(2)] the Manhattan curve $\Cc_M$ is strictly convex at $0$, i.e., $\th_{S_\ast/S}''(0)>0$, and
\item[\upshape(3)] the Manhattan curve $\Cc_M$ is strictly convex at every point, i.e., $\th_{S_\ast/S}''(a)>0$ for every $a \in \R$.
\end{enumerate}
\end{theorem}

\def\bS{{\bf S}}

\begin{remark}
Let $\bS_n = \{x \in \Gamma: |x|_S =n\}$. If one of the equivalence statements in Theorem \ref{thmwords} holds
then the law of
\begin{equation}\label{Eq:CLT}
\frac{d_{S_\ast}(o, x_n)-n \t(S_\ast/S)}{\sqrt{n}}
\end{equation}
where $x_n$ is chosen uniformly at random from $\bS_n$, converges, as $n \to \infty$, to the normal distribution with mean $0$ and variance  $\th_{S_\ast/S}''(0)>0$ .
This follows from \cite[Theorems 1.1 and 1.2]{stats} and \cite[Theorem 1.1]{GTT} (where Gekhtman et al.\ have established their result in a more general setting).
Observe that if $d_S$ and $d_{S_\ast}$ are roughly similar, then the limiting distribution of \eqref{Eq:CLT} is the Dirac mass at $0$.
\end{remark}

Let $\overline{\Sigma}$ denote a subshift of finite type associated to a strongly Markov automatic structure $\Ac = (\Gc, \pi, S)$.
\begin{lemma}\label{Lem:cohom}
Let $\Psi_\ast: \overline{\Sigma} \to \mathbb{R}$ be a H\"older continuous function 
such that 
\[
S_n \Psi_\ast(\o)=-d_{S_\ast}(o, \pi(\o_0)\cdots \pi(\o_{n-1})) \quad \text{for $\o=(\o_0, \dots, \o_{n-1}) \in \SS^\ast$},
\]
as in Example \ref{Ex:combable1}.
The function $\Psi_\ast: \Sigma_\Cc \to \mathbb{R}$ is cohomologous to a constant on a maximal component $\Cc$, i.e.,
there exist a constant $c_0 \in \R$ and a H\"older continuous function $u: \SS_\Cc \to \R$ such that
\[
\Psi_\ast=c_0+u-u\circ \s,
\]
 if and only if $d_S$ and $d_{S_\ast}$ are roughly similar.
\end{lemma}

\begin{proof}
  For a maximal component $\Cc$, let $\G_\Cc$ be the set of group elements that are realized as the image of a word corresponding to a finite path in $\Cc$.
Recall that $\Psi_\ast$ is cohomologous to a constant on $\Sigma_\Cc$ if and only if there exists $\t \in \mathbb{R}$ such that the set
\[
\{ S_n\Psi_\ast(\omega) + n\t : \omega \in \Sigma_\Cc \text{ and $n \ge 0$, } n \in \mathbb{Z}\},
\]
 is a bounded subset of $\mathbb{R}$. 
 By the definition of $\Psi_\ast$,
 this holds if and only if 
 \[
 \{d_{S_\ast}(o,x) - \t d_S(o,x) : x \in \Gamma_\Cc\}
 \]
  is a bounded subset of $\R$. 
We show that this is equivalent to the fact that $d_S$ and $d_{S_\ast}$ are roughly similar.

  Let us prove that for each maximal component $\Cc$ there exists a finite set of group elements $B \subset \Gamma$ such that $B \Gamma_\Cc B = \Gamma$.
  This claim, which concludes the proof, is essentially observed in Lemma 4.6 of \cite{gmt}; we provide a proof below for the sake of completeness.
  
  For an element $w \in \G$ and a real number $\D \ge 0$,
  we say that a $S$-geodesic word \textit{$\D$-contains $w$} 
  if it contains a subword $h$ such that $h = h_1wh_2$
  for some $h_1, h_2 \in \Gamma$ with $|h_1|_S, |h_2|_S \le \D$.
  Let $Y_{w,\D}$ be the set of group elements $x \in \G$ such that $x$ is represented by some $S$-geodesic word which does not $\D$-contains $w$.  
  It is known that
  there exists $\D_0> 0$ such that for all $w \in \Gamma$,
  $$\lim_{n\to \infty} \frac{\#(Y_{w,\D_0} \cap \bS_n)}{\#\bS_n}=0,$$ 
  \cite[Theorem 3]{arzlys}.
  Since $\Cc$ is a maximal component in the underlying directed graph $\Gc$ and thus the spectral radius is the exponential volume growth rate relative to $S$,
  the upper density of $\Gamma_\Cc$ is strictly positive, i.e.,
  \[
  \limsup_{n\to\infty} \frac{\#(\Gamma_\Cc \cap \bS_n)}{\#\bS_n} >0.
  \]
  Fix $w\in \Gamma$.
  Since $\Gamma_\Cc$ has positive upper density and $Y_{w, \D_0}$ has vanishing density, 
  $\Gamma_\Cc \backslash Y_{w, \D_0} \neq \emptyset$,
  i.e., there is an element $x$ of $\Gamma_\Cc$ whose every $S$-geodesic representation $\D_0$-contains $w$, in particular there exists $y \in \G_\Cc$ (corresponding to a subword) such that $y=h_1 w h_2$ with $|h_1|_S, |h_2|_S \le \D_0$.
  Hence if we let $B$ denote the ball of radius $\D_0$ with respect to $d_S$ centered at the identity in $\Gamma$, then $\Gamma = B \Gamma_\Cc B$.
\end{proof}

\begin{proof}[Proof of Theorem \ref{thmwords}]
First let us show that $(1) \iff (2)$.
By Theorem \ref{Thm:2nd-d-words} (see the proof of Proposition \ref{Prop:2nd-d}), the second derivative $\th_{S_\ast/S}''(0)$ coincides with the second derivative at $t=0$ of the function $\Pr_\Cc(t\Psi_\ast)$ for each fixed maximal component $\Cc$ (at $t=0$). By Proposition 4.12 of \cite{ParryPollicott} this second derivative is strictly positive if and only if $\Psi_\ast: \Sigma_\Cc \to \mathbb{R}$ is not cohomologous to a constant. Lemma \ref{Lem:cohom} implies that this is true if and only if $d_S$ and $d_{S_\ast}$ are not roughly similar.

Next let us show that $(2) \iff (3)$.
We shall in fact show that if $\th_{S_\ast/S}''(t)>0$ for some $t \in \R$, then $\th_{S_\ast/S}''(t)>0$ for all $t \in \R$.
Let us fix $t_0 \in \R$ and let $\Cc_1, \ldots, \Cc_I$ be the maximal components of $t\Psi_\ast$ at $t=t_0$. 
By Proposition \ref{Prop:perturbation} there exists $\e >0$ such that
\[   
\theta_{S_\ast/S}(t) = \max_{i=1,\ldots,I} \text{Pr}_{\Cc_i}(t\Psi_\ast)
\]
for all $|t-t_0|<\e$. 
Since each $\Pr_{\Cc_i}(t \Psi_\ast)$ is real analytic in $t \in \R$,
changing $\e>0$ if necessary we find at most two components $\Cc_1$ and $\Cc_2$ (possibly $\Cc_1=\Cc_2$) such that
\begin{equation}\label{Eq:discrepancy}
\theta_{S_\ast/S}(t) = 
     \begin{cases}
       \text{Pr}_{\Cc_1}(t\Psi_\ast) &\text{for $t_0 \le t < t_0+\e$}, \\
       \text{Pr}_{\Cc_2}(t\Psi_\ast) &\text{for $t_0-\e < t \le t_0$}. \\
     \end{cases}
\end{equation}
Moreover, $\Pr_{\Cc_1}(t \Psi_\ast)$ (resp.\ $\Pr_{\Cc_2}(t \Psi_\ast)$) is strictly convex at all points if and only if it is strictly convex at some point since this is equivalent to the fact that $\Psi_\ast$ is not cohomologous to a constant function on $\SS_{\Cc_1}$ (resp.\ $\SS_{\Cc_2}$) \cite[Proposition 4.12]{ParryPollicott}.
It follows that if $\th_{S_\ast/S}''(t_0)>0$, then both $\text{Pr}_{\Cc_1}(t\Psi_\ast)$ and $\text{Pr}_{\Cc_2}(t\Psi_\ast)$ are strictly convex at all points since
\[
\theta_{S_\ast/S}''(t_0) = \text{Pr}_{\Cc_1}''(t\Psi_\ast)\Big|_{t=t_0} = \text{Pr}_{\Cc_2}''(t\Psi_\ast)\Big|_{t=t_0} >0,
\]
by the proof of Proposition \ref{Prop:2nd-d}.
Note that since there are only finitely many components in the underlying directed graph,
the set of $t_0 \in \R$ where \eqref{Eq:discrepancy} holds for two distinct $\Cc_1$ and $\Cc_2$ is at most countable and discrete in $\R$.
Applying the same argument to such $t_0$ at most countably many times,
we see that if $\th_{S_\ast/S}$ is strictly convex at some point, then it is strictly convex at all points, as desired.
\end{proof}

\subsection{Tightness of the tangent lines at infinity for the Manhattan curve}\label{Sec:extreme}

In this section we prove an inequality for pressure curves that will be a useful tool in understanding the asymptotic properties of $\Cc_M$. This inequality will have subsequent applications to a large deviation principle (Theorem \ref{Thm:LDT}), our results on the multifractal spectrum (Theorem \ref{Thm:MFS}) and to proving the rationality of the dilation constants associated to word metrics (Proposition \ref{Prop:extreme}).
\begin{proposition}\label{Prop:Pmax}
Let $(\Sigma,\sigma)$ be a transitive subshift of finite type and $\Psi : \Sigma \to \mathbb{R}$ be a H\"older continuous function. 
If we define
\[
P_\infty(\Psi):=\sup\left\{ \exp \left( \frac{S_n \Psi(\o)}{n}\right) \ : \ \s^n \o=\o \quad \text{for $\o \in \SS$ and $n \ge 1$}\right\},
\]
then we have that
\[
 (1/\rho_A)e^{\Pr(t\Psi)}\le P_\infty(t\Psi) \le e^{\Pr(t\Psi)},
\]
for all $t \in \mathbb{R}$, 
where $\Pr(t \Psi)$ is the pressure for $t \Psi$ and $\rho_A$ is the spectral radius of the adjacency matrix for $(\SS, \s)$.

In particular, it holds that
\[
\Pr(t\Psi) = t\log P_\infty(\Psi)  +O(1) \quad \text{and} \quad \Pr(-t\Psi) = t\log P_\infty(-\Psi) +O(1) \quad \text{as $t \to \infty$}.
\]
\end{proposition}

\begin{proof}
Let $E$ be a finite set of alphabets and $A=(A_{e, e'})_{e, e'\in E}$ be the adjacency matrix which defines the transitive subshift of finite type $(\SS, \s)$ on $E$.
We consider the associated finite directed graph $\Gc$ whose set of edges is $E$.
Let us denote by $\O_E$ the set of cycles in $\Gc$, i.e., the set of finite paths $w=(\o_0, \dots, \o_{n-1})$ whose terminus coincides with the origin, and by $|w|=n$ the length of $w$.
Recall that the transfer operator
\[
\Lc_\Psi f(\o)=\sum_{\s(\o')=\o}e^{\Psi(\o')}f(\o') \quad \text{for $f: \SS \to \R$},
\]
has spectral radius $e^{\Pr(\Psi)}$.

First we consider a special case (from which the general case will be reduced); $\Psi:\SS \to \R$ depends only on the first coordinate, i.e.,
$\Psi(\o)=\psi(\o_0)$ for some function $\psi: E \to \R$.
Let
\begin{equation}\label{Eq:Pmax}
P_{\max}(\psi):=\max\left\{\prod_{e \in w} e^{\psi(e)/|w|} \ : \ w \in \O_E \right\},
\end{equation}
where we note that the maximum is attained by some simple cycle (i.e., a cycle consisting of pairwise distinct vertices).
Then 
we have that
\[
P_{\max}(\psi) \le e^{\Pr(\Psi)},
\]
by the definition of spectral radius, and moreover
\begin{equation*}
e^{\Pr(\Psi)} \le \rho_A \cdot  P_{\max}(\psi)
\end{equation*}
where $\rho_A$ stands for the spectral radius of the adjacency matrix $A$
by \cite[Theorem 2]{matrix}---we have applied to the non-negative matrix $\(e^{\psi(e)}A_{e, e'}\)_{e, e' \in E}$
(note that we have $\rho_A=e^{\Pr(0)}$).
Therefore we obtain
\begin{equation}\label{eq:1}
P_{\max}(t \psi) \le e^{\Pr(t\Psi)} \le \rho_A \cdot P_{\max}(t \psi) \quad \text{for all $t \in \R$}.
\end{equation}

Next we consider the general case.
Since $\Psi$ is H\"older continuous, for all $\e >0$ sufficiently small 
there exists a function $\Psi_0 : \Sigma \to \mathbb{R}$ such that $\Psi_0$ depends on finitely many coordinates and satisfies that $\|\Psi_0-\Psi\|_\infty \le \e$. 
Let us re-code the subshift $\Sigma$ where an induced potential depends only on the first coordinate: 
if $\Psi_0$ depends only on at most the first $K$-coordinates for some $K>0$,
then we replace all $K$-length strings allowed by $A$ with a new symbol. 
Thereby we obtain a new transitive subshift $\Sigma_B$ with an adjacency matrix $B$. 
Note that the natural bijective map $\SS \to \SS_B$, $\o=(\o_i)_{i=0}^\infty \mapsto \wt \o=(\o_i \o_{i+1} \cdots \o_{i+K-1})_{i=0}^\infty$ defined by concatenation of subsequent $K$-alphabets yields an isomorphism between the shifts.
If we define $\Psi_0': \Sigma_B \to \mathbb{R}$ by $\Psi_0'(\wt \o)=\Psi_0(\o)$,
then $\Psi_0'(\wt \o)=\psi(\wt \o_0)$ for some function $\psi$ on $\SS_B$,
and the spectral radii of transfer operators for $\Psi_0$ and $\Psi_0'$ coincide so do those of the adjacency matrices $A$ and $B$.
Applying \eqref{eq:1}, we obtain
\begin{align*}
e^{\Pr(t\Psi_0)} \le \rho_B \cdot P_{\max}(t \psi).
\end{align*}
Further since $\rho_A=\rho_B$ and the inequality 
\begin{equation}\label{Eq:Pr-approx}
|\Pr(t\Psi)-\Pr(t\Psi_0)|\le \|t\Psi-t\Psi_0\|_\infty \quad \text{for all $t \in \R$},
\end{equation}
which follows from \eqref{Eq:Pr}, 
we have that
\begin{align*}
e^{\Pr(t\Psi)} \le  e^{\e |t|}\rho_A \cdot P_{\max}(t \psi).
\end{align*}
Combining with
\[
P_{\max}(t \psi) \le e^{\e|t|}\sup\left\{ \exp \left( \frac{t S_n \Psi(\o)}{n}\right) \ : \ \s^n \o=\o \quad \text{for $\o \in \SS$ and $n \ge 1$}\right\},
\]
which follows from the definition \eqref{Eq:Pmax} and \eqref{Eq:Pr-approx}, 
we obtain $e^{\Pr(t \Psi)} \le e^{2 \e |t|}\rho_A \cdot P_\infty(t \Psi)$.
Noting that this estimate is uniform in $t \in \R$ as $\e \to 0$, we have that $e^{\Pr(t \Psi)} \le \rho_A \cdot P_\infty(t \Psi)$ for all $t \in \R$.
Similarly, by using \eqref{eq:1} we have $P_\infty(t \Psi) \le e^{\Pr(t \Psi)}$ for all $t \in \R$, concluding the first claim.
Further noting that $P_\infty(t \Psi)=P_\infty(\Psi)^t$ for all $t \ge 0$, we obtain the second claim.
\end{proof}

\begin{remark}\label{Rem:Sharp}
Richard Sharp has suggested that the above proof can be simplified by appealing to the fact that 
the set of uniform measures on periodic orbits is dense in the set of all $\s$-invariant Borel probability measures in the weak topology  \cite[Theorem 1]{Sigmund}.
We have provided a more elementary approach which gives a clearer insight into the following proposition.
\end{remark}

\begin{proposition}\label{Prop:extreme}
For every pair of word metrics $d_S, d_{S_\ast}$ associated to finite symmetric sets of generators $S, S_\ast$,
the corresponding Manhattan curve satisfies that
\[
\th_{S_\ast/S}(t)=-\a_{\min}t+O(1) \quad \text{as $t \to \infty$} \quad
\text{and}
\quad
\th_{S_\ast/S}(t)=-\a_{\max}t+O(1) \quad \text{as $t \to -\infty$}.
\]
Moreover $\a_{\min}$ and $\a_{\max}$ are rational.
\end{proposition}

\proof
We apply Proposition \ref{Prop:Pmax} to each transitive component in the subshift in Example \ref{Ex:combable1};
for each $t \in \R$, we have that
\[
\th_{S_\ast/S}(t)=\max_{\Cc}\Pr_\Cc(t \Psi_{S_\ast}),
\]
where $\Cc$ is a component in the underlying directed graph by Lemma \ref{Lem:pr-ss}.
Note that although the components which attain the maximum can depend on the $t$,
since we have $\a_{\min}=-\lim_{t \to \infty}\th_{S_\ast/S}(t)/t$ and there are only finitely many components,
there exists $\Cc$ such that $\a_{\min}=-\log P_{\infty}(\Psi_{S_\ast|\Cc})$ where $\Psi_{S_\ast|\Cc}$ is the restriction of $\Psi_{S_\ast}$ to $\SS_\Cc$.
Therefore $\th_{S_\ast/S}(t)=-\a_{\min}t+O(1)$ as $t \to \infty$.
Similarly we have
$\a_{\max}=\log P_{\infty}(-\Psi_{S_\ast|\Cc})$ for a possibly different $\Cc$, 
implying that $\th_{S_\ast/S}(t)=-\a_{\max}t+O(1)$ as $t \to -\infty$.

Furthermore since $\Psi_{S_\ast}(\o)=-d\phi_{S_\ast}(\o_0)$ for $\o \in \SS$ and $d\phi_{S_\ast}$ takes the value in integers  (Example \ref{Ex:combable1}),
by \eqref{Eq:Pmax} in the proof of Proposition \ref{Prop:Pmax},
we have that
\[
\a_{\min}=\sum_{e \in w}\frac{d\phi_{S_\ast}(e)}{|w|} \quad \text{for a cycle $w=(e_0, \dots, e_{k-1})$ and $k=|w|$},
\]
and $\a_{\max}$ has a similar form. 
Hence $\a_{\min}$ and $\a_{\max}$ are rational.
\qed


\subsection{An application to large deviations}
In Section \ref{Sec:general} we compared the typical growth rates of two metrics $d,d_\ast \in \Dc_\Gamma$ by studying two related limits.
In Lemma \ref{Lem:tau} we studied the limiting ratio of the metrics as we travel along ``typical'' quasi-geodesic rays. We then, in Theorem \ref{Thm:distortion}, 
considered 
the limiting average of the ratio of two metrics where we average over $n$-balls in one of the metrics.
In this section we investigate a finer statistical result that compares a metric $d \in \Dc_\G$ with  a word metric $d_S$.
More precisely we study the distribution of $d(o, x)/n$ when $x$ is sampled uniformly at random from the set of all words of length $n$ in $d_S$.

It has been shown that if $d\in \Dc_\Gamma$ is a word metric or is strongly hyperbolic,
then there exists a positive real number $\t$ such that
\begin{equation}\label{eq:LDP0}
\frac{1}{\#\bS_n} \sum_{|x|_S = n} \frac{d(o,x)}{n}  \to  \tau \quad \text{as $n \to \infty$},
\end{equation}
where $\bS_n = \{x \in\Gamma \ : \ |x|_S =n\}$ \cite[Theorem 1.1]{stats}.
(In fact, a stronger result was shown: the left hand side of \eqref{eq:LDP0} is $\tau+O(1/n)$ as $n \to \infty$.) 
Furthermore, the values $d(o, x)/n$ concentrate exponentially near $\t$,
i.e., for all $\epsilon >0$,
\begin{equation*}\label{eq:LDP1}
\limsup_{n\to\infty} \frac{1}{n} \log \left( \frac{1}{\#\bS_n} \#\left\{x\in \bS_n: \left|\frac{d(o,x)}{n} - \tau\right|>\epsilon  \right\}\right) < 0.
\end{equation*}

A detailed analysis of the Manhattan curve allows us to establish a precise global large deviation result that is valid for every metric $d \in \Dc_\G$.

\begin{theorem}\label{Thm:LDT}
Let $\Gamma$ be a non-elementary hyperbolic group equipped with a finite symmetric set of generators $S$. 
Let $\bS_n := \{x \in \Gamma : |x|_S = n\}$ for non-negative integers $n \ge 0$. 
If $d \in \Dc_\Gamma$,
then for every open set $U$ and every closed set $V$ in $\mathbb{R}$ such that $U \subset V$, we have that
\begin{align*}
-\inf_{s\in U} I(s) &\le \liminf_{n\to\infty} \frac{1}{n} \log \left( \frac{1}{\#\bS_n} \#\left\{x\in \bS_n: \frac{d(o,x)}{n} \in U  \right\}\right)\\
&\le  \limsup_{n\to\infty} \frac{1}{n} \log \left( \frac{1}{\#\bS_n} \#\left\{x\in \bS_n: \frac{d(o,x)}{n} \in V  \right\}\right) \le - \inf_{s\in V} I(s),
\end{align*}
where
\[
I(s) = \theta_{d/d_S}(0) + \sup_{t \in \mathbb{R}} \{t s - \theta_{d/d_S}(-t)\}
\]
and $\theta_{d/d_S}$ is the Manhattan curve for the pair of metrics $d$ and $d_S$.
Furthermore, 
$I$ has a unique zero at the mean distortion $\t(d/d_S)$,
is finite on $(\alpha_{\min},\alpha_{\max})$ and infinite outside of $[\alpha_{\min},\alpha_{\max}]$,
where
\[
\a_{\min}=-\lim_{a \to \infty}\frac{\th_{d/d_S}(a)}{a} \quad \text{and} \quad
\a_{\max}=-\lim_{a \to -\infty}\frac{\th_{d/d_S}(a)}{a}.
\]
\end{theorem}

\begin{remark}
 In a recent work, Gekhtman, Taylor and Tiozzo have established the corresponding central limit theorem in the case when $d$ is a (not necessarily proper) hyperbolic metric \cite[Theorem 1.1]{GTT}; 
 we are not aware of the corresponding large deviation principle in the non-proper setting, e.g., see \cite[Theorem 7.3]{GTT2} (for a recent analogous result on random walks, see \cite{BMSS}). 
\end{remark}

\begin{proof}[Proof of Theorem \ref{Thm:LDT}]
For every $t\in\mathbb{R}$, we have that
\[
\lim_{n\to\infty} \frac{1}{n} \log \left( \frac{1}{\#\bS_n} \sum_{|x|_S=n} e^{t d(o,x)} \right)=\theta_{d/d_S}(-t) - \theta_{d/d_S}(0),
\]
by Lemma \ref{Lem:Coornaert},
and furthermore the right hand side is continuously differentiable by Theorem \ref{Thm:C1}.
The theorem then follows from the G\"artner-Ellis theorem (e.g., \cite[Theorem 2.3.6]{ldb})
and the definitions of $I$ and $\th_{d/d_S}$.
\end{proof}

\begin{corollary}\label{Cor:ldt}
Suppose $d,d_S \in \Dc_\Gamma$ are as in Theorem \ref{Thm:LDT}. Further assume that $d$ is a word metric associated to a finite symmetric set of generators $S_\ast$, i.e. $d=d_{S_\ast}$.  
Then,
$I(\a_{\min}), I(\a_{\max})<\infty$ and further $\alpha_{\min}$ and $\alpha_{\max}$ are rational.
\end{corollary}

\begin{proof}
This follows from Proposition \ref{Prop:extreme}.
\end{proof}

\begin{remark}
If we take a strongly hyperbolic metric $d$ in Theorem \ref{Thm:LDT},
then it holds that $I(\a_{\min}), I(\a_{\max})<\infty$ by Proposition \ref{Prop:Pmax} and by the first part in the proof of Proposition \ref{Prop:extreme}. 
We are unsure whether both $I(\alpha_{\min})$ and $I(\alpha_{\max})$ are finite for all $d \in \Dc_\G$,
which we leave open.
\end{remark}


\section{Examples}\label{Sec:example}
In this section we compute the Manhattan curve for two examples, focusing on a pair of word metrics. 
In the first we provide an exact formula for a Manhattan curve associated to a free group. 
In the second we analyze a hyperbolic triangle group, which is another explicit class of hyperbolic groups other than free groups.
We obtain a Manhattan curve as an implicit function for some pair of word metrics in the $(3,3,4)$-triangle group.
For both cases, we use the GAP package \cite{GAP4} to produce explicit forms of automatic structures.

\subsection{The free group of rank $2$}\label{Sec:free}
The following example is considered in \cite[Example 6.5.5]{Calegariscl}.
We extend the discussion found in \cite{Calegariscl} by commenting on the Manhattan curve. 
Let $F = \langle a, b \mid\rangle$ be the free group of rank $2$. 
We consider the standard set of generators $S=\{a, b, a^{-1}, b^{-1}\}$ and another symmetric set of generators 
\[
S_\ast:=\{a,b,c,a^{-1},b^{-1},c^{-1}\} \quad \text{where $c:=ab$}. 
\]
Our aim is to compute the Manhattan curve $\th_{S/S_\ast}$.
Note that a word on $S_\ast$ is reduced if and only if it contains no subword of the form $a^{-1}c, cb^{-1}, c^{-1}a, bc^{-1}$,
and furthermore each element in $F$ has a unique reduced word representative on $S_\ast$. 
Henceforth we use $A, B, C$ to denote $a^{-1}, b^{-1}, c^{-1}$ respectively.
An automatic structure for $(F, S)$ (resp.\ for $(F, S_\ast)$) is given in Figure \ref{Fig:auto_standard_free} (resp.\ Figure \ref{Fig:auto_another_free}).
(In the figures the arrows are solid if they are in the strongly connected component, and dotted otherwise. 
The initial state is denoted by ``1''.)

\begin{figure}
\begin{center}
\begin{tabular}{c}

\begin{minipage}{1\hsize}
\begin{center}
\includegraphics[clip, width=80mm, bb=0 0 528 295]{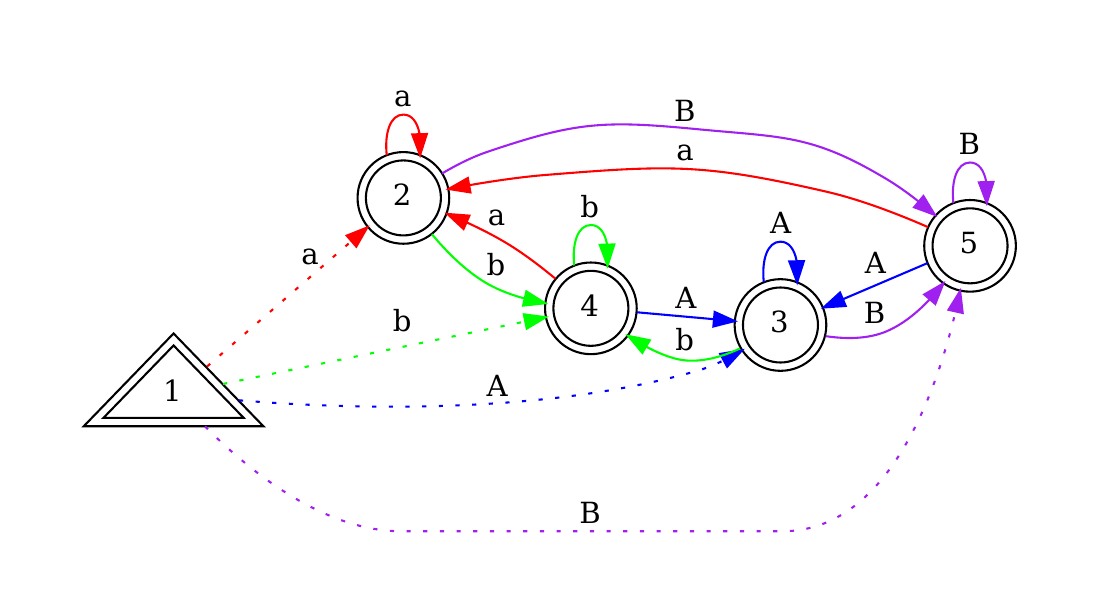}
\hspace{1cm}
\caption{An automatic structure for $(\G, S)$ in Section \ref{Sec:free}.}
\label{Fig:auto_standard_free}
\end{center}
\end{minipage}\\

\begin{minipage}{1\hsize}
\begin{center}\includegraphics[clip, width=70mm, bb=0 0 439 384]{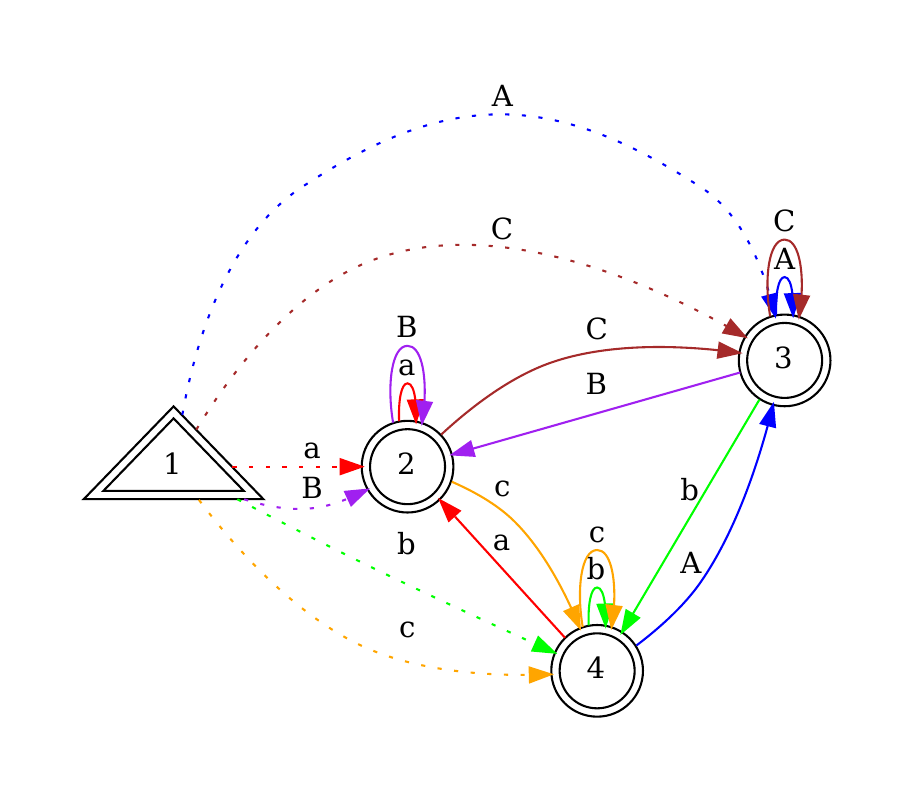}
\hspace{1cm}
\caption{An automatic structure for $(\G, S_\ast)$ in Section \ref{Sec:free}.}
\label{Fig:auto_another_free}
\end{center}
\end{minipage}
\end{tabular}
\end{center}
\end{figure}

\begin{figure}
\centering
\includegraphics[width=100mm, bb=0 0 977 605]{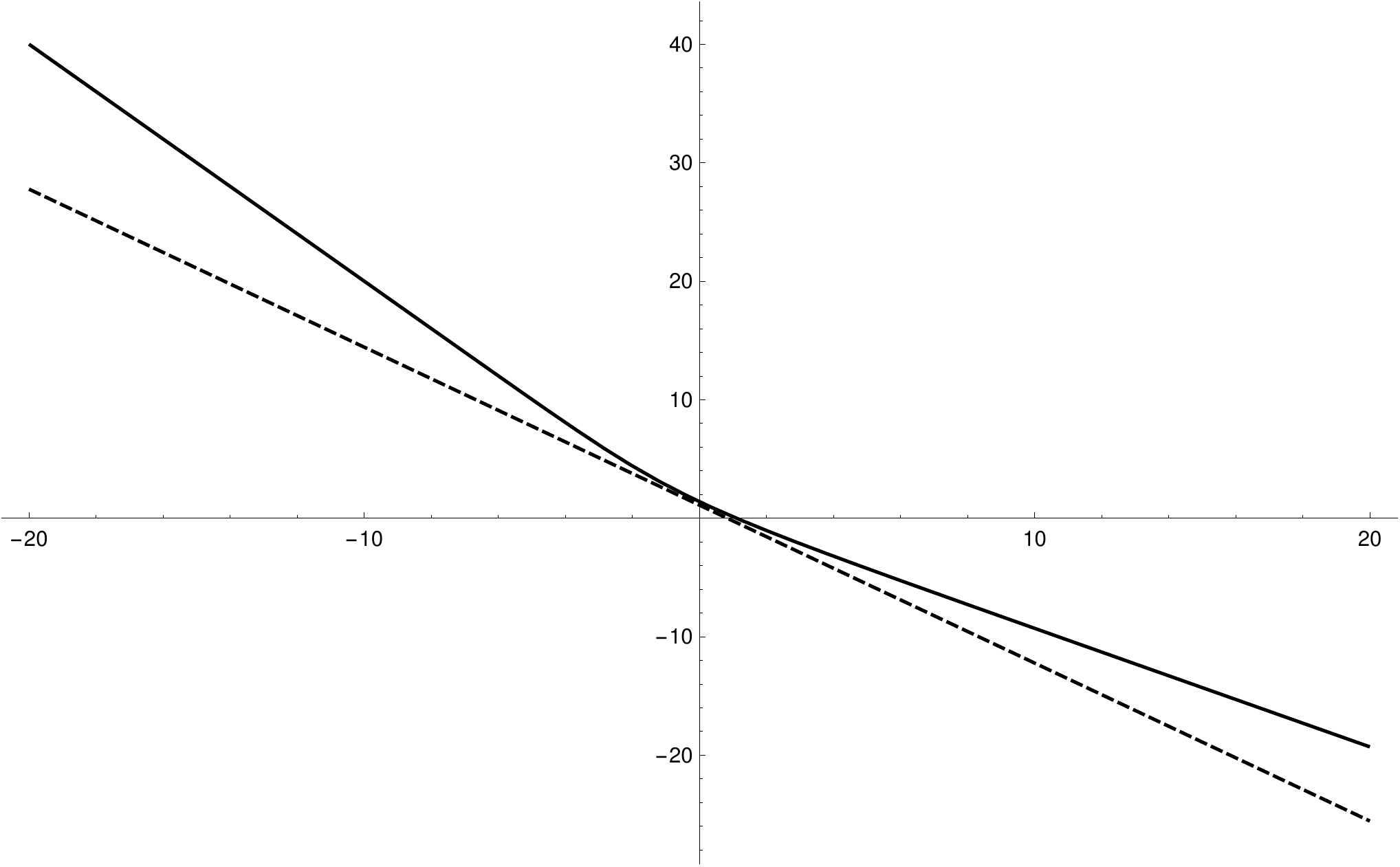}
\hspace{1cm}
\caption{The Manhattan curve (solid) for the example in Section \ref{Sec:free} and the tangent line at $0$ (dotted) for comparison.}
\label{Fig:Manhattan_free}
\end{figure}

The word length with respect to $S$ is computed by setting
\[
d\phi_S(e)
=
\begin{cases}
1 & \text{if $e$ has the label $a, b, A, B$},\\
2 & \text{if $e$ has the label $c, C$},
\end{cases}
\]
where $d\phi_S$ is defined as in Example \ref{Ex:combable1}.
The adjacency matrix of this directed graph is of size $12$ (because we use all the edges in the strongly connected component as indices),
but it is enough to deal with a smaller one:
we observe that the `flip' of the labels by $c \leftrightarrow C$, $a \leftrightarrow B$ and $b \leftrightarrow A$
keeps the directed graph structure with labeling.
This allows us to consider the following matrix of size $6$ where indices correspond to the set of labels $a,b,c,A,B,C$:
\[
\begin{blockarray}{ccccccc}
 {} & a & b & c & A & B & C \\
\begin{block}{c(cccccc)}
  a & 1 & 0 & 1 & 0 & 1 & 1 \\
  b & 1 & 1 & 1 & 1 & 0 & 0\\
  c & 1 & 1 & 1 & 1 & 0 & 0 \\
  A & 0 & 1 & 0 & 1 & 1 & 1 \\
  B & 1 & 0 & 1 & 0 & 1 & 1 \\
  C & 0 & 1 & 0 &1  & 1 & 1\\
\end{block}
\end{blockarray}.
 \]
 Then the Manhattan curve is computed as $\theta_{S/S_\ast}(t) = \log P(e^{-t})$ where $P(s)$ is the spectral radius of the matrix
\[
\begin{blockarray}{cccccc}\\
\begin{block}{(cccccc)}
  s & 0 & s^2 & 0 & s & s^2 \\
  s & s & s^2  & s & 0 & 0\\
  s & s & s^2  & s & 0 & 0 \\
  0 & s & 0 & s & s & s^2  \\
  s & 0 & s^2  & 0 & s & s^2  \\
  0 & s & 0 &s  & s & s^2 \\
\end{block}
\end{blockarray},
 \] 
 (see Lemma \ref{Lem:pr-ss}).
 Hence we obtain
 \[
 \th_{S/S_\ast}(t) = \log \left( \frac{1}{2} e^{-t}(e^{-t}+\sqrt{e^{-t}(e^{-t}+8)} +4)\right),
 \]
 and
 \[
 \alpha_{\max} = \lim_{t \to -\infty}- \frac{1}{t} \th_{S/S_\ast}(t)= 2 \quad \text{and} \quad \alpha_{\min} =  \lim_{t \to \infty} -\frac{1}{t}\th_{S/S_\ast}(t) = 1,
 \]
see Figure \ref{Fig:Manhattan_free}.
Moreover the mean distortion $\t(S/S_\ast)$ and $v_\ast/v$ where $v$ (resp.\ $v_\ast$) is the exponential volume growth rate for $(\G, S)$ (resp.\ $(\G, S_\ast)$) are given by
\[
\t(S/S_\ast)=-\th_{S/S_\ast}'(0)=\frac{4}{3}=1.33333\dots, \quad \text{and} \quad \frac{v_\ast}{v}=\frac{\log 4}{\log 3}=1.26186\dots.
\]

\subsection{The $(3,3,4)$-triangle group}\label{Sec:triangle}
We now turn our attention to computing a Manhattan curve for a pair of word metrics in the case of the $(3,3,4)$-triangle group.
Let 
\[
\G:=\langle a, b, c \, | \, a^3, b^3, c^4, abc\rangle,
\]
where we denote the standard set of generators by
\[
S:=\{a, b, c, a^{-1}, b^{-1}, c^{-1}\},
\]
and another symmetric set of generators by
\[
S_\ast:=\{a, b, c, d, a^{-1}, b^{-1}, c^{-1}, d^{-1}\} \quad \text{where $d:=c^2$}.
\]
Note that $\G$ has a presentation with respect to $S_\ast$,
\[
\langle a, b, c, d \, | \, a^3, b^3, c^4, abc, d^{-1}c^2\rangle.
\]
Based on these presentations for $\G$, 
we compute the exponential growth rate $v$ (resp.\ $v_\ast$) for $(\G, S)$ (resp.\ $(\G, S_\ast)$)
\[
v=0.674756\dots, \quad \text{and} \quad v_\ast=0.732858\dots,
\]
and
we have produced automatic structures for $(\G, S)$ and for $(\G, S_\ast)$
in Figures \ref{Fig:auto_standard_(3,3,4)} and \ref{Fig:auto_another_(3,3,4)}.

\begin{figure}
\begin{center}
\begin{tabular}{c}

\begin{minipage}{1\hsize}
\begin{center}
\includegraphics[clip, width=130mm, bb=0 0 1389 565]{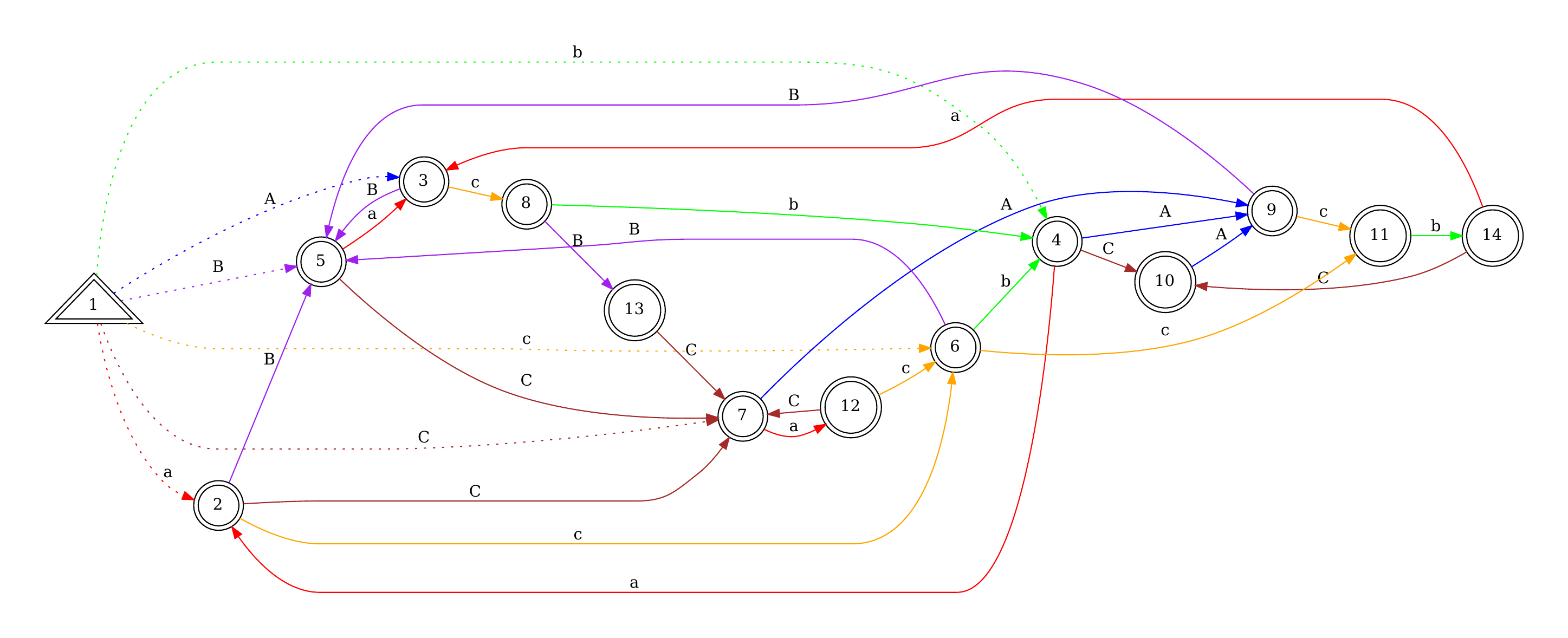}
\hspace{1cm}
\caption{An automatic structure for $(\G, S)$ in Section \ref{Sec:triangle}.}
\label{Fig:auto_standard_(3,3,4)}
\end{center}
\end{minipage}\\

\begin{minipage}{1\hsize}
\begin{center}
\includegraphics[clip, width=100mm, bb=0 0 952 635]{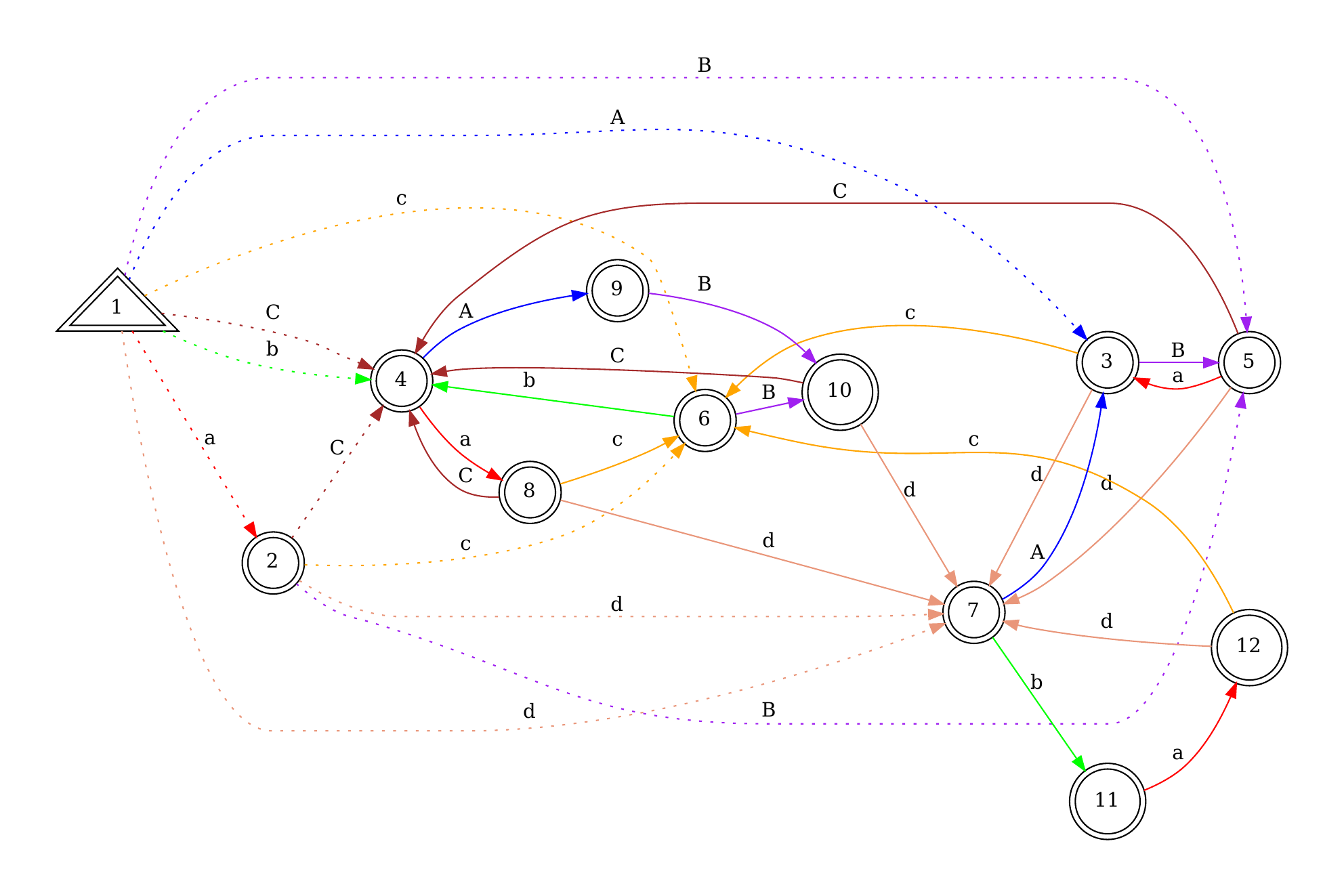}
\hspace{1cm}
\caption{An automatic structure for $(\G, S_\ast)$ in Section \ref{Sec:triangle}.}
\label{Fig:auto_another_(3,3,4)}
\end{center}
\end{minipage}
\end{tabular}
\end{center}
\end{figure}

\begin{figure}
\centering
\includegraphics[width=100mm, bb=0 0 729 384]{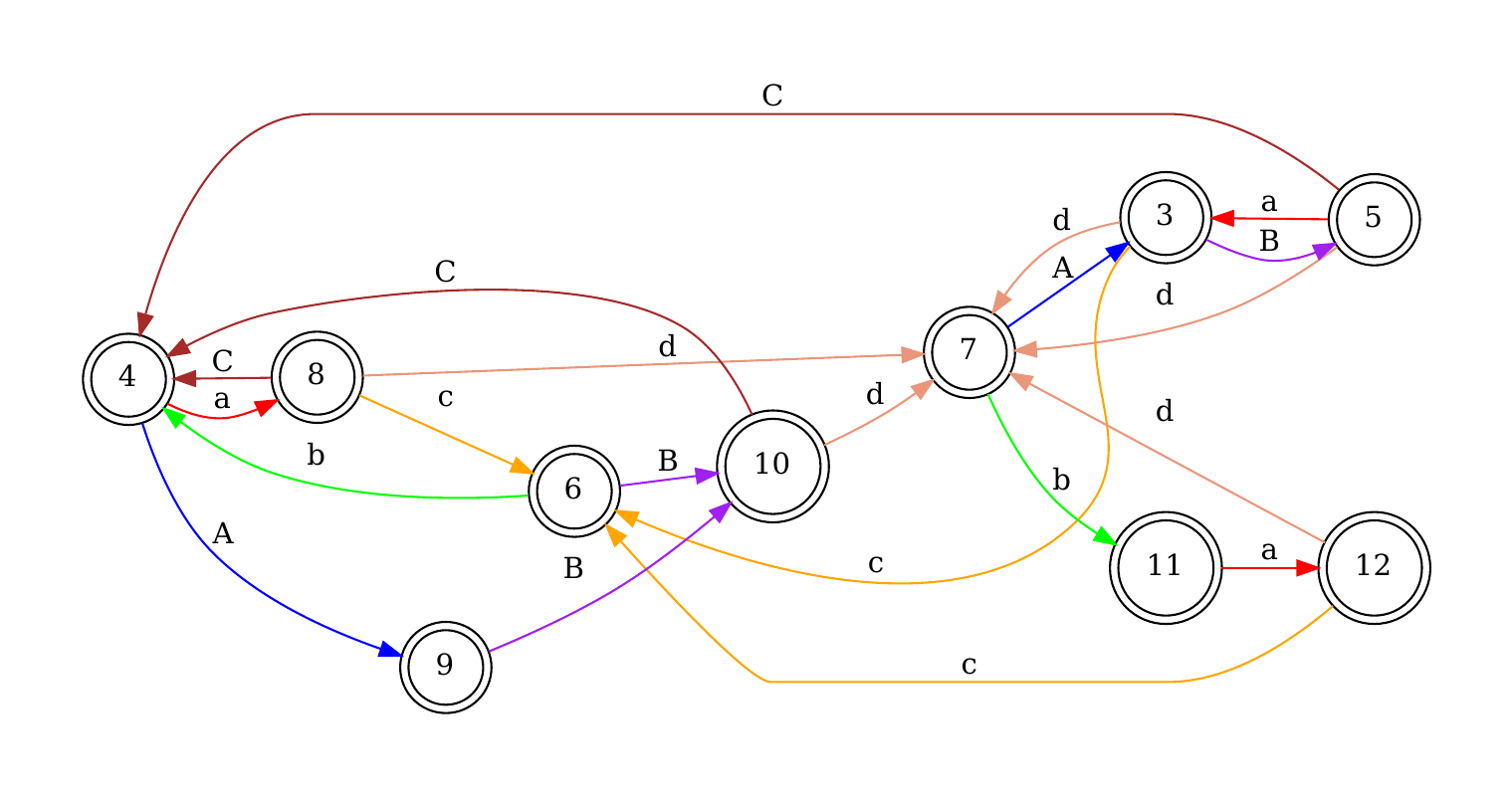}
\hspace{1cm}
\caption{The strongly connected component of the automatic structure for $(\G, S_\ast)$ in Figure \ref{Fig:auto_another_(3,3,4)}.
The weight $1$ is assigned on the directed edges with labels $a, A, b, B, c, C$, and the weight $2$ is assigned on the directed edge with label $d$.}
\label{Fig:auto_another_(3,3,4)SCC}
\end{figure}

\begin{figure}
\centering
\includegraphics[width=110mm, bb=0 0 1040 643]{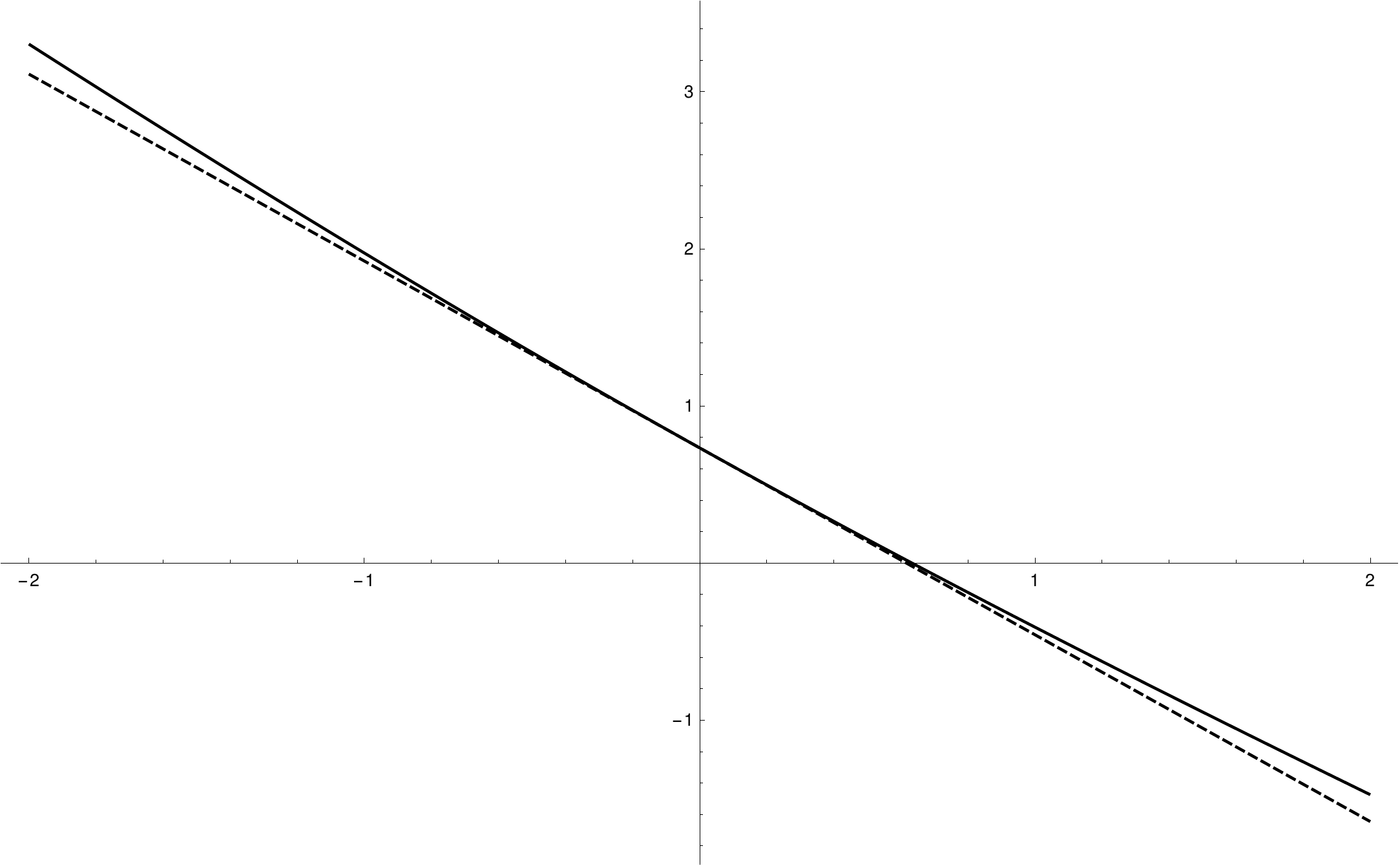}
\hspace{1cm}
\caption{The Manhattan curve (solid) for the example in Section \ref{Sec:triangle} and the tangent line at $0$ (dotted) for comparison.}
\label{Fig:Manhattan_graph_(3,3,4)}
\end{figure}

Following the method in Section \ref{Sec:free},
we find a matrix representation of the transfer operator, which is read off from the strongly connected component with appropriately defined weights in Figure \ref{Fig:auto_another_(3,3,4)SCC}.
The following is the characteristic polynomial of the transfer matrix of size $21$:
\[
-e^{-9s}x^{13}(-1+e^s x)(1+e^s x)(1+e^s x+e^{2s} x^2)
(1-e^s x-e^{2s}x^2-e^{3s}x^2-e^{4s}x^3+e^{5s}x^4).
\]
Let $s \mapsto r(s)$ be the branch given as a root of 
\[
e^{3s}-e^{4s}x-e^{5s}x^2-x^{6s}x^2-e^{7s}x^3+e^{8s}x^4,
\]
such that $r(s)$ coincides with the spectral radius of the transfer operator at $s=0$.
Then the Manhattan curve is obtained by $\th_{S/S_\ast}(s)=\log r(s)$,
see Figure \ref{Fig:Manhattan_graph_(3,3,4)}.
We find that $\a_{\max}=3/2$ and $\a_{\min}=1$.
Moreover the expansion of $\th_{S/S_\ast}(s)$ at $0$ has the following form:
\[
\th_{S/S_\ast}(s)=0.732858\dots -1.18937\dots s+0.0515301\dots s^2+O(s^3),
\]
and 
the mean distortion $\t(S/S_\ast)$ and $v_\ast/v$ are given by
\[
\t(S/S_\ast)=\frac{1}{68}(85-\sqrt{17})=1.18937\dots, \quad \text{and} \quad \frac{v_\ast}{v}=1.08611\dots.
\]
Note that $\t(S/S_\ast)$ is quadratic irrational although it is \textit{a priori} a root of a higher degree polynomial with rational coefficients.
We are still far from having
a systematic understanding of the exact class of numbers to which $\tau(S/S_\ast)$ can belong in general. 
We believe that this deserves further investigation.

\appendix

\section{Proof of Lemma \ref{Lem:diff}}\label{Sec:Lebesgue}

In this section we will refer to some fundamental facts from \cite{Heinonen}.
Let $X$ be a topological space endowed with a quasi-metric $\rho$ which is compatible with the topology on $X$,
and let us call $(X, \rho)$ a {\it quasi-metric space}.
Let $\m$ be a Borel regular measure on a topological space $(X, \rho)$.
For $r>0$ and $x \in X$,
we denote the ball of radius $r$ centered at $x$ relative to the quasi-metric $\rho$ by
\[
B(x, r):=\{y \in X \ : \ \rho(x, y)<r\}.
\]
A measure $\m$ is called {\it doubling} if all balls have {\it finite} and {\it positive} $\m$-measure and there exists a constant $C(\m)>0$ such that
$\m(B(x, 2 r)) \le C(\m)\m(B(x, r))$ for every ball $B(x, r)$ in $X$.
We call $C(\m)$ a {\it doubling constant} of $\m$.

Let $(X, \rho)$ be a quasi-metric space which admits a doubling measure $\m$.
For every Borel regular finite (non-negative) measure $\n$ on $X$, i.e., $\n(X)<\infty$,
let us decompose 
\[
\n=\n_{\ac}+\n_\sing,
\]
where $\n_\ac$ is the absolutely continuous part of $\n$ and $\n_\sing$ is the singular part of $\n$ relative to $\m$, respectively.
Since $\n$ is finite, $\n_\ac$ is also finite and thus $d\n_\ac/d\m$ is integrable.

\begin{lemma}\label{Lem:diff-general}
For $\m$-almost every $x \in X$, we have that
\[
\lim_{r \to 0}\frac{\n_\ac\(B(x, r)\)}{\m\(B(x, r)\)}=\frac{d\n_\ac}{d\m}(x)<\infty
\quad \text{and} \quad
\limsup_{r \to 0}\frac{\n_\sing\(B(x, r)\)}{\m\(B(x, r)\)}=0.
\]
\end{lemma}

\proof
By adapting the proof of the Lebesgue differentiation theorem \cite[Theorem 1.8]{Heinonen} to a quasi-metric,
we have 
\[
\lim_{r \to 0}\frac{\n_\ac\(B(x, r)\)}{\m\(B(x, r)\)}=\frac{d\n_\ac}{d\m}(x) \quad \text{for $\m$-almost every $x \in X$}.
\]
Now we show the second claim.
Let us write $\n=\n_\sing$.
Since $\n$ and $\m$ are mutually singular,
there exists a measurable set $N$ in $X$ such that $\n(N)=\n(X)$ and $\m(N)=0$.
Further (the singular part) $\n$ is also Borel regular and the inner regularity implies that
for all $\e>0$ there exists a compact set $K_\e$ in $N$ such that $\n\(N\setminus K_\e\)<\e$.
Let $\n_\e:=\n|_{X\setminus K_\e}$ be the restriction of $\n$ on $X \setminus K_\e$,
for which $\n_\e(X)<\e$.
For all $x \in X \setminus K_\e$, there exists a small enough $r>0$ such that $B(x, r) \subset X \setminus K_\e$.
If we define
\[
L \n(x):=\limsup_{r \to 0}\frac{\n\(B(x, r)\)}{\m\(B(x, r)\)}
\quad \text{and} \quad
M \n_\e(x):=\sup_{r>0} \frac{\n_\e\(B(x, r)\)}{\m\(B(x, r)\)}
\quad \text{for $x \in X$},
\]
then for all $t>0$,
\[
\{L \n>t\}\subset K_\e \cup \{M \n_\e>t\}.
\]
The weak maximal inequality shows that
\[
\m\(\{M \n_\e>t\}\)\le \frac{C_\m}{t}\n_\e(X),
\]
for all $t>0$ and for a constant $C_\m$ depending only on the doubling constant of $\m$
(where the proof follows exactly as for (2.3) in \cite[Theorem 2.2]{Heinonen})
and thus
\[
\m\(\{L \n>t\}\) \le \m\(K_\e\)+\m\(\{M \n_\e>t\}\) \le \frac{C_\m}{t}\n_\e(X)<\frac{C_\m}{t}\e,
\]
where we have used the fact that $\m\(K_\e\) \le \m(N)=0$.
Therefore for each $t>0$, letting $\e \to 0$, we have that
\[
\m\(\{L \n >t\}\)=0,
\]
and this shows that $L \n(x)=0$ for $\m$-almost every $x \in X$, as required.
\qed

\begin{remark}
The above fact is standard in metric spaces with doubling measures, and the part adapted to a quasi-metric is mainly the place where we use Vitali's covering theorem.
We avoid repeating all the details: see \cite[Chapters 1 and 2]{Heinonen}.
\end{remark}

\proof[Proof of Lemma \ref{Lem:diff}]
If we fix a large enough $R>0$ given $C$ for which $d \in \Dc_\G$ is a $C$-rough geodesic metric,
then Lemma \ref{Lem:shadow-ball} and Lemma \ref{Lem:diff-general} together with the fact that $\m$ is doubling show that
\[
\limsup_{n \to \infty}\frac{\n_\ac\(O(\x_n, R)\)}{\m\(O(\x_n, R)\)} \le C_{\m, R}\frac{d\n_\ac}{d\m}(\x)<\infty,
\quad \text{and} \quad
\limsup_{n \to \infty}\frac{\n_\sing\(O(\x_n, R)\)}{\m\(O(\x_n, R)\)}=0
\]
for $\m$-almost every $\x$ in $\partial \G$
where $C_{\m, R, \d}$ is a positive constant depending only on a doubling constant of $\m$ and $R$ as well as the $\d$-hyperbolicity constant of $d$,
and we recall that $\x_n=\g_\x(n)$ for $n \ge 1$ and $\g_\x$ is a $C$-rough geodesic ray from $o$ converging to $\x$.
This concludes the claim.
\qed

\subsection*{Acknowledgment}
We would like to thank Professor Richard Sharp for his suggestion in Remark \ref{Rem:Sharp}, and the anonymous referee for useful comments.
R.T.\ is partially supported by JSPS Grant-in-Aid for Scientific Research JP20K03602 and JST, ACT-X Grant Number JPMJAX190J, Japan.

\bibliographystyle{alpha}
\bibliography{Manhattan}

\end{document}